\documentclass[twoside]{article}

\usepackage[accepted]{halstyle}

\usepackage[utf8]{inputenc} 
\usepackage[T1]{fontenc}    
\usepackage{hyperref}
\hypersetup{
    colorlinks=true,
    linkcolor=Salmon,
    filecolor=magenta,      
    urlcolor=cyan,
    citecolor=RoyalBlue
    }
\usepackage{url}            
\usepackage{booktabs}       
\usepackage{amsfonts}       
\usepackage{nicefrac}       
\usepackage{microtype}      
\usepackage[dvipsnames,svgnames]{xcolor}         

\usepackage{graphicx} 
\graphicspath{{./}{../}{../figures/}}
\usepackage{microtype}
\usepackage{subcaption}
\usepackage{booktabs} 
\usepackage{microtype}
\usepackage{tikz}

\usepackage{amsmath}
\usepackage{pgfplots}
\usepgfplotslibrary{fillbetween}
\usepackage{amssymb}
\usepackage{mathtools}
\usepackage{amsthm}
\usepackage[capitalize,noabbrev]{cleveref}
\usetikzlibrary{calc}

\usepackage{dsfont} 
\usepackage{etoolbox}
\usepackage{enumerate}
\usepackage{wrapfig}
\usepackage[round]{natbib}

\global\long\def\esp{\mathbb{E}}%
\global\long\def\R{\mathbb{R}}%
\global\long\def\P{\mathbb{P}}%

\theoremstyle{plain}
\newtheorem{theorem}{Theorem}[section]
\newtheorem{propo}[theorem]{Proposition}
\newtheorem{lemma}[theorem]{Lemma}
\newtheorem{corollary}[theorem]{Corollary}
\newtheorem{assumption}{Assumption}

\theoremstyle{remark}
\newtheorem{remark}[theorem]{Remark}
\newtheorem{example}[theorem]{Example}

\newcommand{\e}[1]{\mathbb{E}\left[#1\right]  } 

\newcommand{\ind}{\mathds{1}}

\newcommand{\M}{\mathcal{M}}

\newcommand{\obs}{\mathrm{obs}}
\newcommand{\mis}{\mathrm{mis}}

\newcommand{\prob}[1]{\mathbb{P}\left(#1\right)} 

\begin{document}

\twocolumn[

\aistatstitle{A primer on linear classification with missing data}

\aistatsauthor{ Angel Reyero Lobo \And Alexis Ayme \And Claire Boyer \And Erwan Scornet }

\aistatsaddress{Université Paul Sabatier \And  Sorbonne Université \And   Université Paris-Saclay \And Sorbonne Université }]

\begin{abstract}
Supervised learning with missing data aims at building the best prediction of a target output based on partially-observed inputs. Major approaches to address this problem  can be decomposed into $(i)$ impute-then-predict strategies, which first fill in the empty input components and then apply a unique predictor and $(ii)$ Pattern-by-Pattern (P-b-P) approaches, where a predictor is built on each missing pattern. In this paper, we theoretically analyze how three classical linear classifiers, namely perceptron, logistic regression and linear discriminant analysis (LDA), behave with Missing Completely At Random (MCAR) data, depending on the strategy (imputation or P-b-P) used to handle missing values. We prove that both imputation and P-b-P approaches are ill-specified in a logistic regression framework, thus questioning the relevance of such approaches to handle missing data. The most favorable auspices to perform classification with missing data concern P-b-P LDA methods. We provide finite-sample bounds for the excess risk in this framework, even for high-dimensional or MNAR settings. Experiments illustrate our theoretical findings. 
\end{abstract}

\bigskip


\bigskip


\section{Introduction}

Due to the large size of modern data sets, and the automatization of data collection, missing values are ubiquitous in real-world applications. Missing data can arise due to various reasons, such as sensor malfunctions, survey respondents skipping questions, or integration of data from diverse sources, collected using different methods. In his seminal work, \citet{RUBIN76} categorizes missing value scenarios into three types: Missing Completely At Random (MCAR), Missing At Random (MAR), and Missing Not At Random (MNAR), depending on relationships between observed variables, missing variables, and the missing data pattern.

\textbf{Estimation} Much of the focus in missing value literature is on parameter estimation. Regarding linear models, closed-form coefficient estimators have been derived \citep{little1992regression,Jones1996IndicatorAS,robins1994estimation}, including sparsity constraints \citep{Rosenbaum2009, LohWainwright12} or the study of the optimization procedure \citep{sportisse2020debiasing}.
Regarding logistic regression models, no closed-form solutions are available and one may resort to the Expectation-Maximization algorithm \citep{consentino2011missing}. Using the EM for parameter estimation in generalized linear models was introduced by \citet{ibrahim1990incomplete} for MAR data (with asymptotic theoretical guarantees) and later extended to some MNAR settings by modelling the missing indicators \citep{ibrahim1999missing}. Methods for estimating the parameters in high-dimensional LDA frameworks with MCAR data have also been proposed \citep[see, e.g.,][]{AdaLDA}.

\textbf{Regression} Prediction tasks with missing values differ from model estimation: estimated model parameters alone cannot directly predict outcomes on test samples containing missing values. Impute-then-predict strategies, often encountered in practice, consists in imputing the training dataset, before applying standard algorithms.  \citet{josse2019consistency, bertsimas2024simple} prove the consistency of constant imputation strategies preceding non-parametric learning methods, a result later extended for almost all imputation functions by \citet{le2021sa}.
While these results are asymptotic and strongly rely on non-parametric estimators, \citet{ayme2023naive,ayme2024random} provides a finite-sample analysis of imputation in linear models. 
%
%
%
An alternative approach involves decomposing the Bayes predictor on a pattern-by-pattern basis, training a specific predictor for each missing pattern and leveraging the information provided by them. \citet{agarwal2019robustness} examined the Principal Component Regression strategy for handling missing values in high-dimensional settings.  \citet{lemorvan:hal-02888867,le2020linear} and \citet{ayme2022near} analyze pattern-by-pattern linear predictors, in finite-sample settings. 

\textbf{Classification} Regarding classification, 
 in fact, few theoretical analyses exist on prediction with missing data. \citet{pelckmans2005handling} adapted Support Vector Machine (SVM) classifiers to accommodate missing values. \citet{sell2023nonparametric} establish minimax rate for non-parametric prediction with missing values and propose a minimax algorithm called HAM. 

From a practical perspective, many methods have been proposed to deal with missing values. For example, \citet{garcia2009k}  propose to use $K$ nearest neighbors to impute and predict with missing data, a work later refined by \citet{choudhury2020missing}. Besides, MissForest \citep{stekhoven2012missforest} is one of the most versatile supervised learning algorithm, able to deal with discrete and continuous features. \citet{jiang2020logistic} is one of the few methods able to estimate parameter and predict in presence of missing values. Recently, a strong interest has been put on comparing the empirical performances of various imputation strategies \citep{bertsimas2018predictive,poulos2018missing, jager2021benchmark}.  The interested reader may refer to \citet{emmanuel2021survey} for a review of methods able to perform classification with missing values. However, most of these methods are not theoretically grounded.



\textbf{Contributions} 
While predictions in linear regression with missing data has been extensively studied (see paragraph \textit{regression} above), to the best of our knowledge, there exists no theoretical work analyzing the benefits of Pattern-by-Pattern (P-b-P) and impute-then-predict strategies on linear \textit{classifiers}. 
In this paper, we theoretically study the validity of these two approaches with linear classifiers for MCAR missing data.
First, we analyze the perceptron model (Section \ref{sec:perceptron}) and prove that both strategies to handle missing values are likely to fail, due to the intrinsic difficulty of maintaining the linear separability on each missing pattern. 
Since linear separability is a strong assumption, we turn to the widely-used logistic regression model. We prove that this model is ill-specified to handle MCAR data (Section \ref{subsec:Logistic}), whether imputation or P-b-P strategies are used. 
As modelling the joint distribution of input-output seems to be unavoidable, we turn into the linear discriminant analysis (LDA) framework (Section \ref{subsec:LDA}). We prove that imputation is not optimal for LDA but P-b-P is. We establish finite-sample bounds on the excess risk of LDA, in high-dimensional regimes, also allowing missing data to be more complex (MNAR). Experiments (Section~\ref{sec:experiments}) illustrate our findings. 

\section{Preliminaries on supervised statistical learning with missing values}
\label{sect:predictMissingValues}


\paragraph{Supervised learning}
The main objective of binary classification tasks  is to predict a target $Y \in \{-1,1\}$ given some  observation $X\in\mathbb{R}^d$.  
A canonical way of quantifying the performance of a classifier $h: \mathbb{R}^d \to \{-1,1\}$  is given by the   probability of misclassification
\begin{align*}
    \mathcal{R}_{\mathrm{comp}}(h) =\P(Y\neq h(X)),
\end{align*}
where the index ``\textrm{comp}'' stands for complete data. The Bayes predictor, minimizing the risk $\mathcal{R}_\mathrm{comp}$, takes the form $h^{\star}_{\mathrm{comp}}(X)=\mathrm{sign}(\e{Y|X}).$ 
As the data distribution is unknown, learning consists in estimating $h^{\star}_{\mathrm{comp}}$ given a training sample $\mathcal{D}_n := \left\{ (X_i, Y_i), i=1,\hdots,n\right\}$. 

\paragraph{Missing data in learning}


In the context of supervised learning with missing values, we assume that the input observation $X\in \mathbb{R}^d$ is only partially observed, with $M \in \{0,1\}^d$ the associated missing pattern: each coordinate $M_j = 1$ indicates that the $j$th component of the input vector $X_j$ is missing (and $M_j=0$ if $X_j$ is observed). 
Given a specific missing pattern $m\in \{0,1\}^d$, we define $\mathrm{obs}(m)$ (resp.\ $\mis(m)$)  as the set of indices where $m$ is 0 (resp.\ 1), representing the observed (resp.\ missing) variables. 
Subsequently, $X_{\mathrm{obs}(M)}$ (resp.\ $X_{\mis(M)}$) refers to the subvector of $X$ containing the observed (resp.\ missing) entries of $X$. 
Our aim is to predict the output $Y$ from a pair consisting of the masked observation and the missing pattern, denoted as $Z := (X_{\mathrm{obs}(M)}, M)$ belonging to $\mathcal{Z}$. In presence of missing data, the performance of a classifier $h: \mathcal{Z} \to \{-1,1\}$ is evaluated via 
\begin{align*}
    \mathcal{R}_{\mathrm{mis}}(h) =\P(Y\neq h(Z)),
\end{align*}
and the Bayes predictor $h^{\star}_{\mathrm{mis}}: \mathcal{Z} \to \{-1,1\}$ that minimizes $\mathcal{R}_{\mathrm{mis}}$ is defined as $h^{\star}_{\mathrm{mis}}(Z)=\mathrm{sign}(\e{Y|Z}).$ A procedure $h$ is Bayes optimal if it achieves the Bayes risk, that is $\mathcal{R}_{\mathrm{mis}}(h) = \mathcal{R}_{\mathrm{mis}}(h^{\star}_{\mathrm{mis}})$.
Our analysis is based on the fact  that the Bayes predictor $h^{\star}_{\mathrm{mis}}$ can be decomposed with respect to the missing patterns (see Lemma~\ref{lem_decomp_pbp_generic}), that is 
\begin{align} 
\label{eq:pbyp_bayes_predictor}
    \hspace{-0.3cm}&h^{\star}_{\mathrm{mis}}(Z)=\!\! \sum_{m\in\M}h_m^{\star}(X_{\mathrm{obs}(m)})\ind_{M=m} ,
\end{align}
with $h_m^{\star}(X_{\mathrm{obs}(m)})=\mathrm{sign}(\e{Y|X_{\mathrm{obs}(m)},M=m})$, where $\M \subset \{1, \ldots, d\}$ is the set of admissible missing patterns.
Learning with missing values can be seen as estimating $h_m^{\star}$ for all $m\in \{0,1\}^d$, given an incomplete i.i.d.\ training sample $\mathcal{D}_n := \left\{ (X_{i,\mathrm{obs}(M_i)}, M_i, Y_i), i=1,\hdots,n\right\}$.




\section{Perceptron}\label{sec:perceptron}

We explore in this section how missing values impact geometry-based predictors such as the perceptron.
The principle of the perceptron algorithm \citep{rosenblatt1958perceptron} is to iteratively find a hyperplane separating the data.  
The convergence of the method is ensured under the separability of the observations \citep{novikoff1962convergence}. 
In order to capture the influence of missing data, the goal is therefore to quantify the probability of maintaining linear separability in the presence of missing values, thus ensuring the validity of the perceptron algorithm in the presence of missing values.


When dealing with complete observations,  we say that the points $(X_i,Y_i)_{i=1,...,n}\in \R^d \times \{-1,+1\}$ are linearly separable if there exists a hyperplane, parameterized by $(w^{\star},b^{\star})\in \R^d\times\R$, such that for all $i\in \{1,...,n\}, Y_i \left(X_i^\top w^\star+b^{\star}\right)>0.$




When dealing with missing inputs, the training data 
 $(X_i\odot (1-M_i),M_i,Y_i)_{i=1,...,n}\in(\R^d \times \{0,1\}^d \times \{-1,+1\})^n$ 
is said linearly separable 
if $\forall m \in \{0,1\}^d$, $\exists (w^\star_{(m)},b^\star_{(m)}) \in \R^d\times\R $ such that $\forall i \text{ s.t.~} M_i =m$,
\begin{align*}
  Y_i\left((w^\star_{(m)})^\top (1-M_i)\odot X_i+b^\star_{(m)}\right)>0.
\end{align*}

\begin{remark}[Related work: the rare eclipse problem] 
    In \cite{bandeira2014compressive}, the authors investigate the preservation of linear separability between two convex sets under random Gaussian projections. 
    This particular problem is referred to as the \textit{rare eclipse problem}.
    Unlike the Gaussian projections covered in \cite{bandeira2014compressive}, the case of missing values involves random projections aligned with canonical axes. 
\end{remark}

\begin{lemma}
\label{lem:separable_complete_does_not_imply_miss_separability}
Linear separability of complete data does not imply that of incomplete data. 
\end{lemma}
In all generality, the perceptron model cannot be transferred from complete to missing data patterns. Thus, both imputation and P-b-P approaches fail in all generality, as the separability of incomplete data is required for the convergence of the perceptron algorithm. 
%
In  the following, we make additional assumptions on the input distribution (adapted to the perceptron model), to highlight favorable cases of predictor adaptability to missing inputs. 


\begin{assumption}[Fixed centroids and random radii]\label{ass:fixed_centroids}
For given centroids $c_1$ and $c_2$, both classes are arbitrarily distributed in disjoint Euclidean balls $B_1$ and $B_2$, of radii $R_1$ and $R_2$, centered around the centroids. 
Radii $R_1$ and $R_2$ are uniformly distributed as $R_1, R_2 \sim \mathcal{U}(0, \frac{1}{2}\left\|c_1-c_2\right\|_2)^{\otimes 2}$.
\end{assumption}
%

\begin{assumption}[MCAR]\label{ass:MCAR}
$M \perp\!\!\!\perp (X,Y)$.
\end{assumption}

Under MCAR assumption, remark that preserving linear separability despite missing values means that the Euclidean balls used to generate data remains disjoint when restricted to the support of observed entries.
%
\begin{propo}[Separability of two balls with different radius]\label{prop:finiteBallSepar}
Grant \Cref{ass:fixed_centroids} and \Cref{ass:MCAR} (MCAR). For all $1 \leq j \leq d$, let $\eta_{j}:=\P(M_j=1)$. Then
\begin{align}
    \alpha \leq \P\left(B_{1,\mathrm{obs}(M)}\cap B_{2,\mathrm{obs}(M)} =\emptyset\right) \leq \sqrt{\alpha},
    \label{eq_prop_bound_perceptron}
\end{align}
with $\alpha = \langle 1 - \eta, (c_1 - c_2)^2\rangle / \|c_1 - c_2\|_2^2 \leq 1.$
\end{propo}
The lower bound in \eqref{eq_prop_bound_perceptron} is informative when the probability of missing values on each coordinate remains low. 
When for any coordinate $j$, $\eta_j=\eta$, the bounds \eqref{eq_prop_bound_perceptron} become independent of the centroids: 
$$
(1-\eta) \leq
        \P\left(B_{1,\mathrm{obs}(M)}\cap B_{2,\mathrm{obs}(M)} =\emptyset\right)
        \leq 
        \sqrt{1-\eta}.
$$
On the contrary, when there is only one coordinate $j_0$ always missing ($\eta_{j_0}=1$ and $\eta_{j}=0$ for $j\neq j_0$), the bounds reveal that
\begin{align}
    1-\frac{(c_{1,j_0}-c_{2,j_0})^2}{\|c_1-c_2\|_2^2} & \leq
        \P\left(B_{1,\mathrm{obs}(M)}\cap B_{2,\mathrm{obs}(M)} =\emptyset\right). 
\end{align}
This highlights that for high proportions of missing values that are very localized at certain coordinates, the linear separation will be all the more preserved if the quantity $\|c_1-c_2\|_2^2$ is carried uniformly across the coordinates, i.e., when the vector $c_1-c_2$ is anti-sparse. 

While some positive results can be obtained with high probability in some specific settings (with few missing data, see above, or in high-dimensional settings, see \Cref{app:random_centroids}), linear separability for each missing pattern is impossible to obtain without additional assumptions. Thus, assuming only the linear separability of complete data, a P-b-P approach is not Bayes optimal and nor is the constant imputation, as shown below. 
\begin{lemma}
\label{lemma_P-b-P_imputation_negative}
    If a P-b-P approach with linear classifiers is not Bayes optimal, then constant imputation with linear classifiers is not Bayes optimal. 
\end{lemma}



\section{Logistic Regression}\label{subsec:Logistic}


Since linear separability is difficult to maintain for each missing pattern, we now relax this assumption by modelling  the distribution of $Y|X$, following the popular logistic regression framework.

\begin{assumption}[Logistic model]\label{ass:logistic_model_complete}
Let $\sigma (t)= 1/(1 + e^{-t})$. There exist $\beta_0^\star, \hdots, \beta_d^\star \in \mathds{R}$ such that
the distribution of the output $Y\in\{-1,1\}$ given the complete input $X$ satisfies $\prob{Y=1|X} = \sigma ( \beta_0^\star + \sum_{j=1}^d \beta_j^\star X_j)$.
\end{assumption}
Assuming that the logistic model holds on complete data, one could be tempted to use a P-b-P approach using a different logistic regression model for each missing pattern. 
\Cref{prop:NotLogis} below shows that such a strategy is doomed to fail when input variables are independent and missing data uninformative (MCAR). 

\begin{propo}\label{prop:NotLogis}
Grant \Cref{ass:MCAR} (MCAR) and \Cref{ass:logistic_model_complete} for complete data. 
Furthermore, assume that the components $X_1, \hdots, X_d$ are independent, each one with an unbounded support, satisfying $\e{\exp(\beta_j^\star X_j)} < \infty$. 
Let $m \in \{0,1\}^d$ and assume that the logistic model holds on the missing pattern $M=m$, that is there exist a vector $\beta_{m}^\star \in \mathds{R}^{|\obs (m)| +1}$ such that
 \begin{align*}
        & \prob{Y=1|X_{\mathrm{obs}(m)}, M=m} \!=  \sigma \Big( \beta_{0,m}^\star + \!\!\!\!\!\!\sum_{j \in \mathrm{obs}(m)} \!\!\beta_{j,m}^\star X_j \Big).
\end{align*} 
Then, for all $j \in \mis(m)$, $\beta_j^\star = 0$.
\end{propo}
\Cref{prop:NotLogis} emphasizes that under MCAR missing data, the logistic model cannot be valid on the complete input vector and on any incomplete vector simultaneously, unless the unobserved components are not involved in the original logistic regression model. Using logistic models for all missing patterns is thus an ill-specified strategy, which will lead to inconsistent probability estimators. Note that such a result highlights that constant imputation is also an ill-specified strategy, in the most simple case of independent entries. Interestingly, this result holds for each missing pattern separately. In particular, the logistic model should not be used even if only two missing patterns are possible.

Contrary to linear regression for which the prediction structure can be preserved when the inputs are partially observed \citep[see, e.g.,][]{le2020linear,ayme2022near}, logistic models are not suited for missing data, assuming in both settings independent input variables with MCAR missingness.  
Indeed, due to the non-linearity,
\begin{align*}
   \mathds{E} [Y | X_{\mathrm{obs}(M)}, M = m] 
   \neq &  \sigma \Big( \mathds{E}[ \beta_0^{\star} + \!\!\sum\limits_{j=1}^d \beta_j^{\star} X_j | X_{\mathrm{obs}(M)}]\Big).
\end{align*}
Therefore, the logistic model is not preserved on missing patterns and resulting P-b-P/imputation strategies are ill-posed, as  shown in \Cref{prop:NotLogis} for independent input. Thus, modelling the conditional distribution $Y|X$ without additional constraints on the distribution of $X$ appears to be insufficient to ensure the soundness of P-b-P/imputation strategies. This is therefore the direction we will explore in the next section.

\section{Linear Discriminant Analysis with missing data}\label{subsec:LDA}

In this section, we analyze the finite-sample property of pattern-by-pattern LDA. 

\subsection{Bayes optimality}\label{subsubsec:LDAMissingValues}

Linear discriminant analysis (LDA) relies on Gaussian assumptions of the distributions of $X|Y=k$ for each class $k$. 

\begin{assumption}[Balanced LDA]\label{ass:LDA}
Let $\Sigma$ be a positive semi-definite, symmetric matrix of size $d \times d$. Set $\pi_1 = \P(Y=1)$ and $\pi_{-1} = \P(Y=-1)$ such that $\pi_1 = \pi_{-1}$.
For each class $k \in \{-1, 1\}$,  $X|Y=k \sim \mathcal{N}(\mu_k,\Sigma)$, with $\mu_k \in \mathbb{R}^d$. 
\end{assumption}

LDA can be seen as a specific case of logistic regression, with well-chosen coefficients and a specific distribution on $X$ (mixture of two Gaussian components) as detailed in \Cref{lem_link_log_LDA}, thus restricting the setting of Section~\ref{subsec:Logistic}. Notably, in the complete case of LDA, the Bayes predictor reads 
\begin{align}\label{eq:LDAcomp}
h^{\star}_{\mathrm{comp}}(x):=\mathrm{sign}\left( \left(\mu_{1}-\mu_{-1}\right)^{\top}\Sigma^{-1}\left(x-\frac{\mu_{1}+\mu_{-1}}{2}\right) \right),
\end{align}
minimizing the misclassification probability $\mathcal{R}_{\mathrm{comp}}$ (see Section~\ref{subsec:results_on_LDA} for details).
When MCAR data occurs,  by denoting  $\Sigma_{\mathrm{obs}(M)}:=\Sigma_{\mathrm{obs}(M)\times\mathrm{obs}(M)}$ (and $\Sigma_{\mathrm{obs}(M)}^{-1} =(\Sigma_{\mathrm{obs}(M)})^{-1} $), the pattern-by-pattern Bayes predictor \eqref{eq:pbyp_bayes_predictor} can be written as follows.

\begin{propo}[Pattern-by-pattern Bayes predictor for LDA with MCAR data]\label{prop:MCARLDA}
    Under Assumptions \ref{ass:MCAR} (MCAR) and \ref{ass:LDA} (LDA), the pattern-by-pattern Bayes classifier is given by
    \begin{align*}
        h^{\star}_m(x_{\mathrm{obs}(m)}) & = \mathrm{sign}\Big( \big(\mu_{1, \mathrm{obs}(m)}-\mu_{-1, \mathrm{obs}(m)}\big)^{\top}\Sigma_{\mathrm{obs}(m)}^{-1} \\
        & \times \Big(x_{\mathrm{obs}(m)}-\frac{\mu_{1, \mathrm{obs}(m)}+\mu_{-1,\mathrm{obs}(m)}}{2}\Big)  \Big). 
    \end{align*}
    
\end{propo}




The decomposition provided in Proposition \ref{prop:MCARLDA} relies on the fact that, under MCAR assumption, the distribution of $X_{\mathrm{obs}(M)}|Y, M=m$ is Gaussian for all $m \in \mathcal{M}$ (see \cref{lemma:obsGaus}), similarly to the complete case. 
This does not hold anymore with a MAR missing mechanism, as shown below. Therefore, in the following, we first consider MCAR missing data. 
\begin{example}[LDA+MAR is not pattern-by-pattern LDA]\label{ex:notLDA}
Let $X\in \R^2$ be a random variable satisfying Assumption \ref{ass:LDA}, i.e., such that for each class $k$, $X|Y=k\sim\mathcal{N}(\mu_k,I_2)$. Let $M=(0,\ind_{X_1>0})$ be the MAR missing pattern, where the first coordinate is always observed, and the second is only observed if the first coordinate is positive. In this case, the input distribution of $X_{\mathrm{obs}(M)}|Y\!=\!k, M\!=\!m$, for the pattern $m=(0,1)$, is not Gaussian, as its first component is always positive.
\end{example}

According to \Cref{prop:MCARLDA}, Pattern-by-Pattern LDA is Bayes optimal for any choice of the covariance matrix $\Sigma$ and any signal $\mu_1 - \mu_{-1}$. This contrasts with perceptron and logistic framework for which the P-b-P strategy was not Bayes optimal or well-defined. 

On the other hand, \Cref{prop:imputation_LDA_negative} below shows that constant imputation strategies are not Bayes optimal in all LDA frameworks. Therefore, we focus solely on the pattern-by-pattern strategy in the sequel. 
\begin{propo}
\label{prop:imputation_LDA_negative}
Under Assumptions \ref{ass:MCAR} (MCAR) and \ref{ass:LDA} (LDA) with $d\geq 3$, constant imputation is Bayes optimal for all values of $\mu_{-1}$ and $\mu_{1}$ if and only if $\Sigma$ is diagonal. In this case, the optimal imputation is given by  $\alpha_j = (\mu_{1,j} - \mu_{-1, j})/2$ for all $1 \leq j \leq d$.
\end{propo}

\subsection{Finite-sample bounds}

Our goal is to study whether the Bayes risk with missing values converges to the Bayes risk with complete data as the dimension $d$ increases. 
To do so, we scrutinize the error $\mathcal{R}_{\mathrm{mis}}(h^\star)-\mathcal{R}_{\mathrm{comp}}(h^\star_{\rm comp})$.

\begin{assumption}[Constant $\P(M_j=1)$]
The random variables $M_1, \hdots, M_d$ are independent, and follow a Bernoulli distribution with parameter $\eta$. \label{ass:constEtaMi}
\end{assumption}
\begin{assumption}[Constant $(\mu_1-\mu_{-1})_j$]\label{ass:constMeanDiffe}
    $\forall j\in \{1,...d\}, (\mu_1-\mu_{-1})_j=\pm\mu $, with $\mu>0$.
\end{assumption}
Assumption \ref{ass:constEtaMi} ensures that the missingness probability is the same for each input coordinate. 
Assumption \ref{ass:constMeanDiffe} can be achieved up to a change of coordinates. In the sequel, we refer to $\mathrm{SNR}:=\mu/\sqrt{\lambda_{\max}(\Sigma)}$ as the signal-to-noise ratio, where $\lambda_{\max}(\Sigma)$ is the largest eigenvalue of the input covariance matrix. This quantity describes the overlapping of the classes, and thus the difficulty of the classification task. 

\begin{propo}
Under Assumptions \ref{ass:MCAR}, \ref{ass:LDA}, \ref{ass:constEtaMi} and \ref{ass:constMeanDiffe}, 
    \begin{align*}
         & \mathcal{R}_{\mathrm{mis}}(h^\star)-\mathcal{R}_{\mathrm{comp}}(h^\star_{\mathrm{comp}})\\
        &  \leq \frac{\eta^d}{2 } + \frac{\mu \eta  }{2\sqrt{2\pi}} \sqrt{\frac{d}{\lambda_{\min}(\Sigma)}}\left( \epsilon\left(\eta, \mathrm{SNR}\right)^{d-1}-\eta^{d-1}\right),
    \end{align*}
    with $\epsilon(\eta, \mathrm{SNR}):=\eta+e^{-\frac{\mathrm{SNR}^2}{8}}(1-\eta)<1$.
    \label{prop:BoundDiffLLcomp}
\end{propo}
The bound provided in Proposition \ref{prop:BoundDiffLLcomp} outlines that the difference between the Bayes risk with missing and complete data decreases exponentially fast with the input dimension $d$, assuming that the minimum eigenvalue of the covariance matrix is lower bounded or decreases at most polynomially with $d$ \citep[an assumption already considered in 
e.g.,][]{AdaLDA, cai2011direct}.
This is the first analysis of the bias term due to learning with missing data in a classification context. 
When the signal-to-noise ratio $\mathrm{SNR}$ goes to infinity, one should expect the classification rate to be improved.



\begin{corollary}\label{cor:limitLLcomp}
    Under Assumptions \ref{ass:MCAR}, \ref{ass:LDA}, \ref{ass:constEtaMi}, \ref{ass:constMeanDiffe}, 
    $$
    \lim_{\mathrm{SNR}\to \infty}
    \sqrt{\frac{\lambda_{\max}(\Sigma)}{\lambda_{\min}(\Sigma)}} \frac{\mathrm{SNR}}{e^{\mathrm{SNR}^2/8}}
    =0
    $$
    implies $\lim\limits_{{\mathrm{SNR}}\to \infty} \mathcal{R}_{\mathrm{mis}}(h^\star)-\mathcal{R}_{\mathrm{comp}}(h^\star_{\mathrm{comp}})= \frac{\eta^d}{2}.$
\end{corollary}
The limit established in Corollary \ref{cor:limitLLcomp} matches that of the limit of the bound of \cref{prop:BoundDiffLLcomp} when the SNR tends to infinity. It is important to note that the assumption on the structure of $\Sigma$ is mild (as $\lambda_{\max}(\Sigma)/\lambda_{\min}(\Sigma)$ may increase  exponentially) and encompasses various scenarios, for example when $\Sigma=\sigma^2I_d$ or when $\Sigma$ is arbitrary but constant, with increasingly separated classes.



\label{subsubsec:MuEstimLDA}

Based on \Cref{prop:MCARLDA}, we consider the pattern-by-pattern plug-in predictor $\widehat{h}$, in which 
\begin{align}
     \widehat{\mu}_{k, j}=\frac{\sum_{i=1}^nX_{i,j}\ind_{Y_i=k}\ind_{M_{i,j}=0}}{\sum_{i=1}^n\ind_{Y_i=k}\ind_{M_{i,j}=0}}
     \label{eq:estMu}
\end{align}
estimates $\mu_{k,j}$, with the convention $0/0=0$, where the covariance matrix $\Sigma$ is assumed to be known. More precisely,  
\begin{align}
    \widehat{h}_m(x_{\mathrm{obs}(m)})& = \mathrm{sign}\left( \left(\widehat{\mu}_{1, \mathrm{obs}(m)}-\widehat{\mu}_{-1, \mathrm{obs}(m)}\right)^{\top} \Sigma^{-1} \right. \nonumber \\
    &\times \left.\left(x_{\mathrm{obs}(m)}-\frac{\widehat{\mu}_{1, \mathrm{obs}(m)}+\widehat{\mu}_{-1,\mathrm{obs}(m)}}{2}\right)\right). \label{eq:PbPLDAEstimMuAllPatterns}
\end{align}
Remark that under MCAR assumption, the estimates $\widehat{\mu}_{k, j}$ are built with all the observed inputs, independently of their missing patterns. This departs from a pattern-by-pattern estimation strategy where each mean is computed pattern-wise, using each observation once. 
We define $\kappa :=\max_{i\in [d]}(\Sigma_{i,i})/\lambda_{\mathrm{min}}(\Sigma)$ as the largest value of the diagonal of the covariance over its smallest eigenvalue, which can be regarded as a non-standard condition number of $\Sigma$.
%

\begin{theorem}[Bound on P-b-P LDA with known $\Sigma$]\label{th:BoundPbPMuEstimLDAGen}
 Grant Assumptions \ref{ass:MCAR}, \ref{ass:LDA} and \ref{ass:constEtaMi}. Then the excess risk of the classifier $\widehat{h}$, defined in \eqref{eq:PbPLDAEstimMuAllPatterns}, satisfies
    \begin{align*}
        & \mathcal{R}_{\mathrm{mis}}(\widehat{h})-\mathcal{R}_{\mathrm{mis}}(h^\star)\\
        & \leq \frac{2}{\sqrt{2\pi}}\left(\left(\frac{1+\eta}{2}\right)^n \frac{\left\|\mu\right\|_\infty^2 d(1-\eta)}{\lambda_{\min}\left(\Sigma\right)}+\frac{4\kappa d}{n}\right)^\frac{1}{2}.
    \end{align*}
    Then, for $n$ large enough, we have 
    \begin{align}
    \mathcal{R}_{\mathrm{mis}}(\widehat{h})-\mathcal{R}_{\mathrm{mis}}(h^\star) \lesssim \sqrt{\kappa d/n}.    \label{upper_bound_LDA_estim}
    \end{align}
\end{theorem}
The convergence rate of the LDA classifier in presence of missing values (and with a known covariance) is of the order of $(d/n)^{1/2}$. Moreover, the dependence of the upper bound on the covariance matrix $\Sigma$ is mild, since the respective term decreases exponentially (corresponding to the case where all data are missing). 


The upper bound presented in \eqref{upper_bound_LDA_estim} is independent of the missingness probability $\eta$. If this could be surprising at first sight, it is important to note that the quantity of interest here is the difference between the misclassification probabilities of the estimated LDA predictor and the pattern-by-pattern LDA Bayes predictor given in Proposition \ref{prop:MCARLDA}. Both risks are integrated w.r.t.\ the distribution of missing inputs, so that both risks include the same missing data scenario. 
However, the influence of the probability of missingness should be expected when comparing predictors dealing with incomplete data on the one hand and the complete case on the other, as shown in the following corollary. 


\begin{corollary}\label{cor:final_LDA_upper_bound_Sigma_qcq}
Grant Assumptions \ref{ass:MCAR}, \ref{ass:LDA}, \ref{ass:constEtaMi}, \ref{ass:constMeanDiffe}. Then the classifier $\widehat{h}$, defined in \eqref{eq:PbPLDAEstimMuAllPatterns}  satisfies
   \begin{align*}
    &  \mathcal{R}_{\mathrm{mis}}(\widehat{h})-\mathcal{R}_{\mathrm{comp}}(h^\star_\mathrm{comp}) \\
 &\leq\frac{2}{\sqrt{2\pi}}\left(\left(\frac{1+\eta}{2}\right)^n \frac{\left\|\mu\right\|_\infty^2 d(1-\eta)}{\lambda_{\min}\left(\Sigma\right)}+\frac{4\kappa d}{n}\right)^\frac{1}{2}\\&\quad+ \frac{\eta^d}{2}+ \frac{\mu \eta  }{2\sqrt{2\pi}} \sqrt{\frac{d}{\lambda_{\min}(\Sigma)}}\left( \epsilon\left(\eta, \mathrm{SNR}\right)^{d -1}-\eta^{d -1 }\right)
\end{align*} 
with $\epsilon(\eta, \mathrm{SNR}):=\eta+e^{-\frac{\mathrm{SNR}^2}{8}}(1-\eta)<1$ and $\mathrm{SNR} := \mu / \sqrt{\lambda_{\max}(\Sigma)}$.
\end{corollary}
In the previous bound, the first term is the learning error $\mathcal{R}_{\mathrm{mis}}(\widehat{h})-\mathcal{R}_{\mathrm{mis}}(h^\star)$ and scales as $\sqrt{d/n}$; the second term is the bias $\mathcal{R}_{\mathrm{mis}}(h^\star)-\mathcal{R}_{\mathrm{comp}}(h^\star_\mathrm{comp})$ due to missing values. 
When
\begin{align}
n \ll \frac{1}{(\eta \cdot \mathrm{SNR})^2} \frac{1}{\epsilon(\eta, \mathrm{SNR})^d},
\end{align}
the learning error inherent to the estimation procedure prevails over the approximation error due to missing values. Then, the impact of missing values on the predictive performances is negligible, and, $\mathcal{R}_{\mathrm{mis}}(\widehat{h})-\mathcal{R}_{\mathrm{comp}}(h^\star_\mathrm{comp})=O( \sqrt{d/n})$,
which corresponds to classical rates \citep[see, e.g.][]{anderson2003}. Assuming that $d = o(n)$, the misclassification risk of the estimated LDA with missing values converges to the Bayes risk with complete data.


%

\begin{remark}[Related work on LDA with missing data] Previous work on LDA with missing values \citep[][]{cai2011direct, AdaLDA} focus on parameter estimation, which is not sufficient to design a procedure to predict with missing values. More precisely, \cite{cai2011direct} assume the $s$-sparsity of the so-called discriminant direction $\beta:=\Sigma^{-1}(\mu_1-\mu_{-1})$ and prove that, estimating this vector via linear programming discriminant (LPD) leads to a predictor $\widehat{h}_\mathrm{LPD}$ on complete data which satisfies $\mathcal{R}_{\mathrm{comp}}(\widehat{h}_\mathrm{LPD})-\mathcal{R}_{\mathrm{comp}}(h^\star_\mathrm{comp}) = O\left((s\log(d)/n)^{1/2}\right)$. Although \citet{AdaLDA} follow a completely different approach, their estimator applied on complete data reaches the same rate of convergence. 
\end{remark}

\subsection{Extension to MNAR settings}\label{subsec:LDAMNAR}

Extending LDA predictors to handle more general missing values is challenging. 
Indeed, as shown in \cref{ex:notLDA}, even under a MAR assumption, a pattern-wise approach for LDA is not valid. In this section, we exhibit a MNAR setting compatible with pattern-by-pattern LDA as follows.
\begin{assumption}[GPMM-LDA]  
	\label{ass:gpmmLDA} 
	For all $m\in\mathcal{M}$ and  $k\in\{-1,+1\}$, $ X_{\obs(m)}|(M\!\!=\!\!m,Y\!\!=\!\!k)\sim\mathcal{N}(\mu_{m,k},\Sigma_{m})$ with $\pi_{m,1} = \pi_{m,-1}$ where $\pi_{m,k}:=\P(Y=k, M=m)$.
\end{assumption}
Under \Cref{ass:gpmmLDA}, the Bayes predictor can be decomposed pattern by pattern as follows.
\begin{propo}[MNAR P-b-P LDA]\label{prop:MNARLDA}
    Under Assumption \ref{ass:gpmmLDA}, the pattern-by-pattern Bayes classifier is    \begin{align*}
        & h^{\star}_m(x_{\obs(m)})= 
        \mathrm{sign}\Bigg( \left(\mu_{m,1}-\mu_{m,-1}\right)^{\top}\Sigma_{m}^{-1}\\
        & \times \left(x_{\obs(m)}-\frac{\mu_{m,1}+\mu_{m,-1}}{2}\right) - \log\left(\frac{\pi_{m,-1}}{\pi_{m,1}}\right)\Bigg).
    \end{align*}
\end{propo}
Given the expression of the Bayes predictor in \Cref{prop:MNARLDA}, we build a plug-in estimate based on the estimation $\widehat{\mu}_{m,k}$ of the mean ${\mu}_{m,k}$ on pattern $m \in \{0,1\}^d$ and class $k$, defined as 
\begin{align}\label{eq:estimMuGPMM1}
    \widehat{\mu}_{m,k}:=\frac{\sum_{i=1}^nX_{i}\ind_{Y_i=k}\ind_{M_i=m}}{\ind_{Y_i=k}\ind_{M_i=m}}.
\end{align}
Due to the potential exponential number of missing patterns, it may be difficult to estimate the $2^{d+1}$ estimates $\widehat{\mu}_{m,k}$. 
In line with \cite{ayme2022near}, we employ a thresholded estimate, which boils down to estimating only the mean over the most frequent missing patterns, that is 
\begin{align}\label{eq:estimMuGPMM2}
    \widetilde{\mu}_{m,k}:=\widehat{\mu}_{m,k}\ind_{\frac{N_{m,k}}{n}>\tau},
\end{align}
with $\tau:=\sqrt{d/n} $ and $N_{m,k}:=\sum_{i=1}^n\ind_{M_i=m}\ind_{Y_i=k}$ the number of observations of the class $k$ with $m$ as missing pattern. 
Note that this estimate is only useful when $d<n$. Assuming that the covariance matrix for each missing pattern is known, we construct the pattern-by-pattern predictor $\widetilde{h}$ defined as 
\begin{align}
    & \widetilde{h}_m(x_{\obs(m)})= \mathrm{sign}\Big( \big(\widetilde{\mu}_{1, \obs(m)}-\widetilde{\mu}_{-1, \obs(m)}\big)^{\top} \Sigma_m^{-1} \nonumber \\
    & \times \big(x_{\obs(m)}-\frac{\widetilde{\mu}_{1, \obs(m)}+\widetilde{\mu}_{-1,\obs(m)}}{2}\big)\Big). \label{eq:PbPLDAEstimMuAllPatterns_tildebis}
\end{align}

\begin{theorem}[MNAR P-b-P LDA estimation]\label{th:MNAREstim}
    Grant \Cref{ass:gpmmLDA} and assume that the classes are balanced on each missing pattern. Let $\tau \geq \sqrt{d/n}$. Then, the plug-in classifier based on \eqref{eq:PbPLDAEstimMuAllPatterns_tildebis} satisfies
    \begin{align}
       &\mathcal{R}_{\mathrm{mis}}(\widetilde{h}) -\mathcal{R}_{\mathrm{mis}}(h^\star) \nonumber \\ 
       & \leq \sum_{m\in \{0,1\}^d} \left(  \frac{4}{\sqrt{2\pi}}+\frac{8}{\sqrt{\pi}}\frac{\left\|\mu_{m}\right\|}{\sqrt{\lambda_{\min}(\Sigma_m)}} \right) \tau\wedge p_m \nonumber \\
       & \quad + \sum_{\substack{m\in \{0,1\}^d,\\ p_m\geq\tau}}  \frac{\sqrt{2}  \left\|\mu_{m}\right\|}{\sqrt{ \pi \lambda_{\min}(\Sigma_m)}} p_m   (1 - p_m)^{n/2}.
    \end{align}
\end{theorem}
\Cref{th:MNAREstim} holds for various types of missingness. Indeed, \Cref{ass:gpmmLDA} is very generic and may correspond to very difficult MNAR settings in which there is no relation between any covariances matrices $\Sigma_m$ or any mean vector $\mu_{m,k}$. In this setting, building consistent predictions requires to build $2^d$ estimates of covariances matrices and $2^{d+1}$ mean estimates, an exponentially difficult task. On the other hand, assuming that there exists unique $\mu_{-1}, \mu_1, \Sigma$ such that $\mu_{\pm 1,m} = \mu_{\pm 1, \obs(m)}$ and $\Sigma_m = \Sigma_{\obs(m)}$ allows us to study a MCAR setting in which proportion of missing values are different across coordinates, a generalization of Section~\ref{subsubsec:MuEstimLDA}. 

The upper bound established in \Cref{th:MNAREstim} is low when few missing patterns are admissible, but it appears to be very large when all $2^d$ missing patterns may occur. However, when the missing distribution is concentrated enough, one can control this upper bound. To see this, let us introduce the missing pattern distribution complexity $\mathfrak{C}_p(\tau):=\sum_{m\in \{0,1\}^d}\tau\wedge p_m$ used in \cite{ayme2022near}, and 
assume that the missingness indicators $M_1, \hdots, M_d$ are independent, distributed as a Bernoulli variable with parameter $\eta \leq d/n$. In such a setting, even if each missing pattern is admissible,  
\begin{align}
        \mathcal{R}_{\mathrm{mis}}(\widetilde{h}) -\mathcal{R}_{\mathrm{mis}}(h^\star) 
       &  \lesssim  \frac{d^2}{n} + (1 - \min_{p_m >0} p_m)^{n/2}, 
    \end{align}
which is much better than the initial upper bound, scaling as $d2^d/n$. This upper bound benefits from the concentration of the missing patterns, as a high number of missing components is unlikely to occur for independent Bernoulli distribution, with a small parameter $\eta \leq d/n$. 

Contrary to \Cref{cor:final_LDA_upper_bound_Sigma_qcq}, we do not compare $ \mathcal{R}_{\mathrm{mis}}(\widehat{h})$ to $\mathcal{R}_{\mathrm{comp}}(h^\star_\mathrm{comp})$ as, in a MNAR setting, the distribution of the fully observed pattern may not be identifiable from the distribution of all missing patterns. Indeed, note that, in \Cref{ass:gpmmLDA}, the distribution of the complete pattern (corresponding to $m = 0$) may be chosen independently of the other distributions ($m \neq 0$). Thus, the difference $ \mathcal{R}_{\mathrm{mis}}(\widehat{h}) - \mathcal{R}_{\mathrm{comp}}(h^\star_\mathrm{comp})$ may be arbitrary large. This highlights the fact that all strategies that first estimate parameters from the complete distribution and then predict on each missing pattern by using these estimations are doomed to fail. 


\section{Experiments}
\label{sec:experiments}

\paragraph{Simulated data.}
We run experiments with an input dimension $d=5$. 
We simulate data according to the following models: 
\begin{enumerate}
    \item the LDA framework  (\Cref{ass:LDA}) with $\Sigma$ equal to $I_d$ or Toeplitz of the form $(0.6^{|i-j|})_{i j}$, and where $\mu_1$ and $\mu_{-1}$ are such that $\mu_1-\mu_{-1}$ is a vector with equal non-zero components. More precisely, we first sample $\mu_0\sim \mathcal{N}(0, \sigma^2 \mathrm{I}_d)$ with $\sigma^2 = 25$, and then $\mu_1|\mu_0\sim\mu_0 + 1.5 \epsilon$, with each coordinate of $\epsilon$ being an independent Rademacher.  

    \item a logistic setting in which $X \sim \mathcal{N}(0, \Sigma)$ where $\Sigma$ is as above and $Y$ follows a logistic model defined in \Cref{ass:logistic_model_complete} with $\beta^{\star}\sim\mathcal{N}(0,I_d)$.
\end{enumerate}
Missing values are either MCAR with the same probability of missingness $\eta=0.5$ on each component or self-masking MNAR (the missingness probability only depends on its value) that is $\mathds{P}(M_{j}=1|X)=\frac{1}{1+\mathrm{exp}(-X^j)}$, where coefficients have been chosen to ensure a probability of missingness close to $\eta$.



\begin{figure*}[t]
    \centering
    \begin{subfigure}{0.24\textwidth}
        \centering
        \includegraphics[width=1\textwidth]{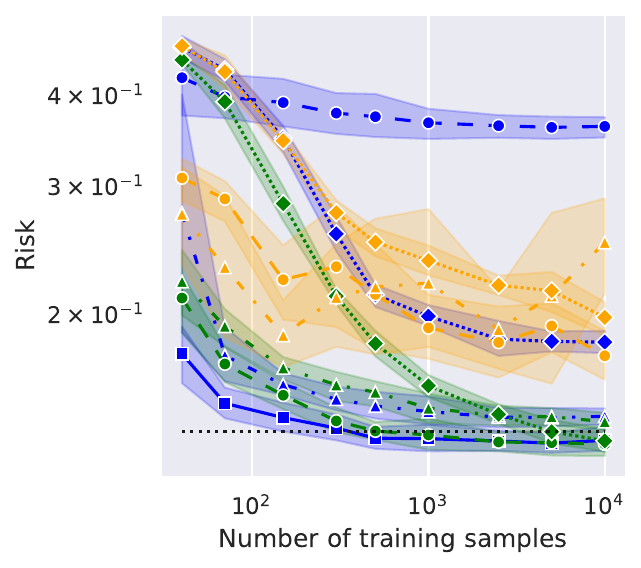}
        \caption{LDA+MCAR}
        \label{fig:LDA_MCAR}
    \end{subfigure}
    \begin{subfigure}{0.24\textwidth}
        \centering
        \includegraphics[width=0.95\textwidth]{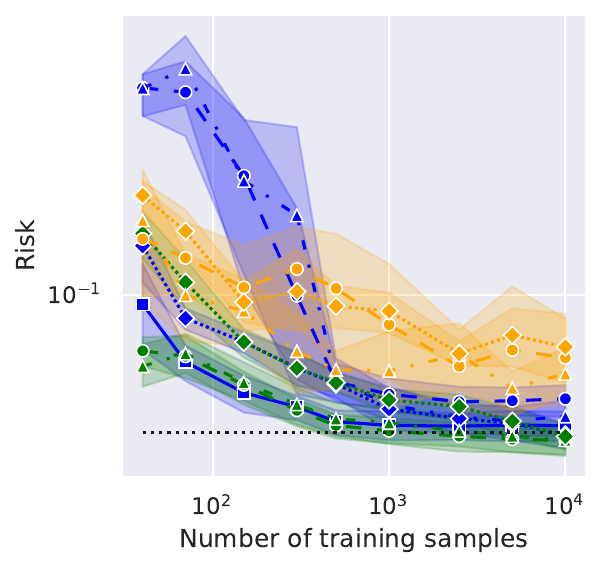}
        \caption{LDA+MNAR}
        \label{fig:LDA_MNAR}
    \end{subfigure}
    \begin{subfigure}{0.24\textwidth}
        \centering
        \includegraphics[width=1\textwidth]{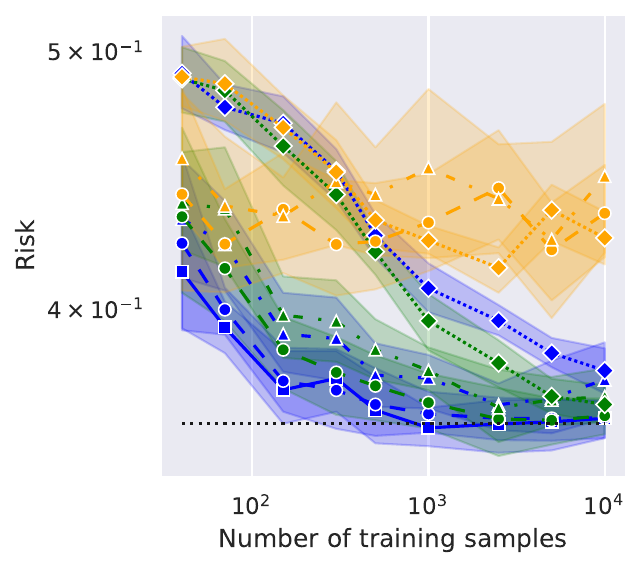}
        \caption{Logistic+MCAR}
        \label{fig:Logistic_MCAR}
    \end{subfigure}
    \begin{subfigure}{0.24\textwidth}
        \centering
        \includegraphics[width=1\textwidth]{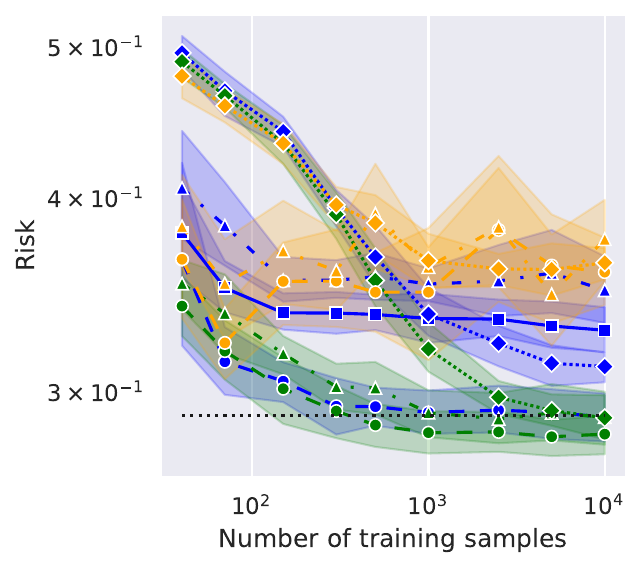}
        \caption{Logistic+MNAR}
        \label{fig:Logistic_MNAR}
    \end{subfigure}
    \begin{minipage}[c]{0.5\textwidth}
    \includegraphics[width=\textwidth]{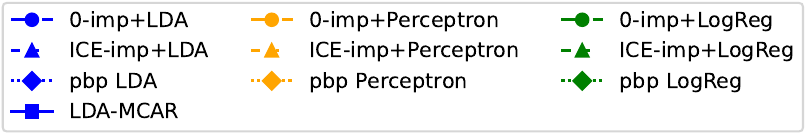}
  \end{minipage}\hfill
  \begin{minipage}[c]{0.48\textwidth}
    \caption{Excess risks of several classifiers on generated data (LDA or Logistic framework) with $\Sigma = I_d$ and MCAR or MNAR missing mechanisms. Dotted lines stand for the Bayes risk $\mathcal{R}_{\mathrm{mis}}(h^{\star}_{\mathrm{mis}})$ with missing data.}
    \label{fig:the_only_one}
    \end{minipage}
\end{figure*}

\paragraph{Algorithms.} 
For prediction purposes, we assess the performances of the following algorithms on simulated data: 
\begin{itemize}
    \item Logistic regression with Pattern-by-pattern strategy (\texttt{pbp LogReg}), imputed-by-0 data (\texttt{0-imp+LogReg}) or ICE-imputed data (\texttt{ICE-imp+LogReg}); ICE imputation refers to single imputation by
    chained equations \citep{van2011mice} implemented as \texttt{IterativeImputer} in scikit-learn \citep{pedregosa2011scikit}.
    \item Perceptron with Pattern-by-pattern strategy (\texttt{pbp Perceptron}),  imputed-by-0 data (\texttt{0-imp+Perceptron}) or ICE-imputed data (\texttt{ICE-imp+Perceptron});
    \item LDA with Pattern-by-pattern strategy (\texttt{pbp LDA}), imputed-by-0 data (\texttt{0-imp+LDA}) or (\texttt{ICE-imp+LDA});
    \item LDA described by Equation (7) and denoted by \texttt{LDA-MCAR}.
\end{itemize}


\subsection{Results}

 Results are displayed in \Cref{fig:the_only_one} for data generated via LDA or Logistic framework ($\Sigma = I_d$) and MCAR or MNAR missingness. Experiments with a Toeplitz covariance matrix can be found in \Cref{app:add_experiments}. 

In \Cref{fig:LDA_MCAR} and \Cref{fig:LDA_MNAR}, we observe that P-b-P strategies show bad performances for small sample sizes. Indeed, in this case, few samples are available in each missing pattern, which makes it difficult (or even impossible) to train each P-b-P classifier. Note also that Perceptron-based methods are not very accurate, since such methods do not converge for non-linearly separable data.

In \Cref{fig:LDA_MCAR}, we see that \texttt{LDA-MCAR} is among the best methods, in terms of predictive accuracy. Notably, P-b-P LDA requires a significantly larger training sample compared to \texttt{LDA-MCAR}, as it does not use information from other missing patterns. Surprisingly, imputation by zero followed by a logistic classifier is on par with \texttt{LDA-MCAR} with MCAR data and outperforms this method for MNAR data (\Cref{fig:LDA_MNAR}). 

In \Cref{fig:Logistic_MCAR} and \Cref{fig:Logistic_MNAR}, we first note that all methods have poor performances, with an excess risk around $0.3$ with $n=10.000$ samples. Such bad performances highlight that classification with logistic data and missing values is a challenging task, despite its simplistic nature. Trying to distinguish methods in this difficult scenario, we observe that zero-imputation with logistic regression and \texttt{LDA-MCAR} are among the best classifiers both in MCAR and MNAR scenario (even if data do not satisfy LDA assumptions).





\section{Conclusion}

In this paper, we analyze the two main classes of methods able to handle missing values for classification: Pattern-by-Pattern (P-b-P) and imputation. 
We prove that both imputation and P-b-P strategies are theoretically deficient for perceptron or logistic classifiers. In the LDA framework, P-b-P is Bayes optimal in all generality, whereas imputation is for diagonal covariance matrices only. We establish finite sample bounds for the excess risk of LDA classifier with MCAR data, which vanishes as $n$ and $d$ grows at some prescribed rates. We extend our analysis to specific MNAR scenarios. Experiments illustrate our theoretical findings, showing that P-b-P LDA outperforms other classifiers for LDA-generated data with MCAR missingness. Suprisingly, despite its ill-posedness, imputation with logistic regression exhibits good performances in this setting. A proper analysis of this phenomenon, by identifying favorable underlying behaviors during training, remains to be conducted.
In this regard, even if our result shows that the probability of classification cannot be properly estimated for any missing pattern, it may be possible that the decision frontier is close to the correct one, which should deserve further study. 
Regarding the discriminant analysis, the Gaussian assumption of the (conditional) distribution of the inputs helps to be theoretically conclusive when copping with missing data. Adapting this study framework to manage categorical inputs would confirm the applicability and relevance of LDA-type predictors in the presence of missing data.

\bibliographystyle{plainnat}
\bibliography{biblio_hal}

\clearpage
\appendix
\onecolumn

\paragraph{Notations.} For $n\in \mathbb N$, we denote $[n] = \lbrace 1, \dots, n\rbrace$. We use $\lesssim$ to denote inequality up to a universal constant. For any $x \in \R^d$ and for any set $J \subset [d]$ of indices, we let $x_J$ be the subvector of $x$ composed of the components indexed by $J$. The abbreviation \textit{P-b-P} refers to 
 \textit{pattern-by-pattern}.  The values $\lambda_{\max}(A)$ and $\lambda_{\min}(A)$ respectively  designate the largest and the smallest eigenvalues of any matrix A. We denote $a\wedge b= \min(a,b)$ and $a\vee b= \max(a,b)$.

\section{Proofs of Section \ref{sect:predictMissingValues}}

\begin{lemma}
\label{lem_decomp_pbp_generic}
Let $h^{\star}$ be a minimizer of $\mathcal{R}_{\mathrm{mis}}(h) :=\P(Y\neq h(Z))$, where $Z = (X_{\mathrm{obs}(M)},M)$. Then, 
\begin{align} 
\notag
    h^{\star}(Z)&= \sum_{m\in\M}h_m^{\star}(X_{\mathrm{obs}(m)})\ind_{M=m},
\end{align}
with $h_m^{\star}(X_{\mathrm{obs}(m)}):=\mathrm{sign}(\e{Y|X_{\mathrm{obs}(m)},M=m})$.
\end{lemma}

\begin{proof}[Proof of \Cref{lem_decomp_pbp_generic}]
    Recall that we quantify the accuracy of a classifier using the probability of misclassification given by 
\begin{align}\label{eq:defMisclass}
    \mathcal{R}_{\mathrm{mis}}(h) :=\P(Y\neq h(Z))
\end{align}
Therefore, we would like to find a classifier minimizing this probability of misclassification. As $|Y-h(Z)|\in\{0,2\}$, then, 
\begin{align}
    \mathcal{R}_{\mathrm{mis}}(h) &=\frac{1}{4}\e{(Y-h(Z))^2}=\frac{1}{4}\e{(Y-\e{Y|Z})^2}+\frac{1}{4}\e{(\e{Y|Z}-h(Z))^2}.
\end{align}
Thus, the Bayes predictor is
\begin{align}
    h^{\star}(Z):=\mathrm{sign}(\e{Y|Z})=\mathrm{sign}(\e{Y|X_{\mathrm{obs}(M)},M}) \text{ where sign}(x)=\ind_{x\geq 0}-\ind_{x< 0}. 
\end{align}
As we have that 
\begin{align}
    \e{Y|X_{\mathrm{obs}(M)},M}= \sum_{m\in\M}\e{Y|X_{\mathrm{obs}(m)},M=m}\ind_{M=m},
\end{align}
then, the Bayes predictor can be written as
\begin{align}\label{eq:BayesClassif}
\notag
    h^{\star}(Z)&=\mathrm{sign}(\e{Y|Z})\\
    \notag
    &=\mathrm{sign}\left(\sum_{m\in\M}\e{Y|X_{\mathrm{obs}(m)},M=m}\ind_{M=m}\right)\\
    \notag
    &=\sum_{m\in\M}\mathrm{sign}(\e{Y|X_{\mathrm{obs}(m)},M=m})\ind_{M=m}\\&=\sum_{m\in\M}h_m^{\star}(X_{\mathrm{obs}(m)})\ind_{M=m}
\end{align}
with 
\begin{align}\label{eq:BayesClassifPbP}
    h_m^{\star}(X_{\mathrm{obs}(m)}):=\mathrm{sign}(\e{Y|X_{\mathrm{obs}(m)},M=m}).
\end{align}
\end{proof}

\section{(Perceptron) Proofs of Section \ref{sec:perceptron}}

\subsection{Proof of \Cref{lem:separable_complete_does_not_imply_miss_separability}}

\begin{proof}
Suppose that we only have two points $X_1,X_2\in \R^d$ where $x_2=(x_{1,1},...,x_{1,(k-1)},x_{2,k},x_{1,{(k+1)}},...,x_{1,d})$ with $x_{1,k}\neq x_{2,k}$. We have $y_2=-y_1$. We also suppose that $m_{1,k}=m_{2,k}=1$ and $m_1=m_2$. Then, $\mathcal{W}$ is not empty, but $\mathcal{W}_{mis}$ is empty as $(1-m_1)\odot x_1=(1-m_2)\odot x_2$, thus for any $w\in\R^d$ if $y_1w^\top(1-m_1)\odot x_1>0$ then $y_2w^\top(1-m_2)\odot x_2=-y_1w^\top(1-m_1)\odot x_1<0$, or the symmetric case.     
\end{proof}

\subsection{Separability characterization}\label{subsec:sepChar}

\begin{lemma}[Separability characterization]\label{lem: sepChar}
Consider the $\ell^p$-balls $B_1$ and $B_2$ resp. centered at  $c_1, c_2$ and of respective radius  $R_1, R_2$. They are disjoint for the $p$-norm if and only if $R_1+R_2<\left\|(c_1-c_2)\right\|_p$.
\end{lemma}
\begin{proof}
    On the one hand, if $\left\|C_1-C_2 \right\|_p\leq R_1+R_2$, then $B_1\cap B_2\neq \emptyset$. For example, $x\in B_1\cap B_2$ for $x:= C_1+\frac{R_1}{R_1+R_2}(C_2-C_1)$  because 
    \begin{align*}
        \left\|x-C_1\right\|_p = \left\| \frac{R_1}{R_1+R_2}(C_2-C_1) \right\|_p\leq \frac{R_1}{R_1+R_2}(R_1+R_2) = R_1
    \end{align*}
    then $x\in B_1$ and 
    \begin{align*}
        \left\|x-C_2\right\|_p = \left\| \frac{R_2}{R_1+R_2}(C_2-C_1) \right\|_p\leq \frac{R_2}{R_1+R_2}(R_1+R_2) = R_2
    \end{align*}
    so $x\in B_2$. 
    
    On the other hand, if there exist an $x$ such that $x\in B_1\cap B_2\neq \emptyset$, then
    $\left\|x-C_1\right\|_p \leq R_1$ and $\left\|x-C_2\right\|_p \leq R_2$. Using the triangle inequality, 
    \begin{align*}
        \left\|(C_1-C_2)\right\|_p\leq \left\|(C_1-x)\right\|_p +\left\|(x-C_2)\right\|_p\leq R_1+R_2
    \end{align*}
\end{proof}

By utilizing this characterization, note that we can redefine the linear separability of two balls as the condition where the distance between their centers is greater than the sum of their individual radii. In the context of our projected balls, we observe that
\begin{align*}
    \P\left(B_{1,\mathrm{obs}(M)}\cap B_{2, \mathrm{obs}(M)} =\emptyset\right)&=\P\left(R_1+R_2<\left\|c_{1,\mathrm{obs}(M)}-c_{2, \mathrm{obs}(M)}\right\|_p \right)
    \\ 
    &=\P\left(R_1+R_2<\left\|\Pi_M(c_1)-\Pi_M(c_2)\right\|_p \right)
    \\ &=\P\left(R_1+R_2<\left\|(1-M)\odot (c_1-c_2)\right\|_p \right).
\end{align*}
In the remainder, we fix $p=2$ (the Euclidean norm).

\subsection{Proof of Proposition \ref{prop:finiteBallSepar}}\label{subsec:finiteBallSepar}

\begin{proof}
    In order to study the separability of the two balls after projection through the missing pattern, we need to study the probability that the sum of the radii is still smaller than the distance between the two centers after projection. Equivalently, 
    \begin{align*}
        \P\left(R_1+R_2<\left\|(1-M)\odot (c_1-c_2)\right\|_2 \right)
    \end{align*}
    as shown in Lemma \ref{lem: sepChar}.
    We have that 
    \begin{align*}
        \P\left(R_1+R_2<\left\|(1-M)\odot (c_1-c_2)\right\|_2 \right)&\geq
        \P\left(\max(R_1,R_2)<\frac{1}{2}\left\|(1-M)\odot (c_1-c_2)\right\|_2 \right)\\
        &=\prod_{i=1}^2\P\left(R_i<\frac{1}{2}\left\|(1-M)\odot (c_1-c_2)\right\|_2 \right)\tag{using that $R_1 \perp\!\!\!\perp R_2$} \\
        &=\P\left(R_1<\frac{1}{2}\left\|(1-M)\odot (c_1-c_2)\right\|_2 \right)^2 \tag{using that $R_1 \sim R_2$} .
    \end{align*}
    By Assumption \ref{ass:fixed_centroids}, $(R_1, R_2) \sim U(0, \frac{1}{2}\left\|c_1-c_2\right\|_2)^{\otimes 2}$ and assuming MCAR data ($R_1 \perp\!\!\!\perp M$),

    \begin{align*}
        \P\left(R_1<\frac{1}{2}\left\|(1-M)\odot (c_1-c_2)\right\|_2 \mid M\right) &=\frac{\left\|(1-M)\odot (c_1-c_2)\right\|_2}{\left\|(c_1-c_2)\right\|_2}.
    \end{align*}
    
    Moreover, note that
    
    \begin{align*}
        \e{\frac{\left\|(1-M)\odot (c_1-c_2)\right\|_2^2}{\left\|(c_1-c_2)\right\|_2^2}}&=
        \e{\frac{\sum_{j=1}^d(1-M_j) (c_{1j}-c_{2j})^2}{\sum_{j=1}^d (c_{1j}-c_{2j})^2}}\\&=
        \frac{\sum_{j=1}^d\e{(1-M_j)} (c_{1j}-c_{2j})^2}{\sum_{j=1}^d (c_{1j}-c_{2j})^2}
        \\&=
        \frac{\sum_{j=1}^d (1-\eta_j) (c_{1j}-c_{2j})^2}{\sum_{j=1}^d (c_{1j}-c_{2j})^2}
    \end{align*}
    Therefore, the lower bound is obtained using Jensen's inequality as follows
      \begin{align*}
        \P\left(R_1+R_2<\left\|(1-M)\odot (c_1-c_2)\right\|_2 \right)&\geq 
        \left( \esp\left[ \sqrt{\frac{\left\|(1-M)\odot (c_1-c_2)\right\|_2^2}{\left\|(c_1-c_2)\right\|_2^2}} \right] \right)^2
        \\
        &\geq \esp\left[ \frac{\left\|(1-M)\odot (c_1-c_2)\right\|_2^2}{\left\|(c_1-c_2)\right\|_2^2}
         \right]  \\
         &=\frac{\sum_{j=1}^d (1-\eta_j) (c_{1j}-c_{2j})^2}{\sum_{j=1}^d (c_{1j}-c_{2j})^2}.
    \end{align*}

    To obtain the upper bound, one can proceed similarly, by using Jensen's inequality,
    \begin{align*}
        \P\left(R_1+R_2<\left\|(1-M)\odot (c_1-c_2)\right\|_2 \right)
        &\leq \P\left(R_1<\left\|(1-M)\odot (c_1-c_2)\right\|_2 \right) \\
        &= \esp\left[ {\frac{\left\|(1-M)\odot (c_1-c_2)\right\|_2}{\left\|(c_1-c_2)\right\|_2}} \right] \\
        &= \sqrt{\left(\esp\left[ {\frac{\left\|(1-M)\odot (c_1-c_2)\right\|_2}{\left\|(c_1-c_2)\right\|_2}} \right] \right)^2} \\
        &\leq \sqrt{\esp\left[ {\frac{\left\|(1-M)\odot (c_1-c_2)\right\|_2^2}{\left\|(c_1-c_2)\right\|_2^2}} \right]}\\
        &= \sqrt{\frac{\sum_{j=1}^d (1-\eta_j) (c_{1j}-c_{2j})^2}{\sum_{j=1}^d (c_{1j}-c_{2j})^2}}.      
    \end{align*}
\end{proof}

\subsection{Random centroids}
\label{app:random_centroids}
The bounds derived in the previous section strongly depends on the geometry of the problem, via the centroid coordinates. To establish more general result, we consider  random  centroids $C_1$ and $C_2 \in \mathbb{R}^d$ and work with disjoint $\ell^p$-balls (of same radius for simplicity).
The former point is particularly suited to preserve the data geometry after random projections induced by missing entries.
\begin{assumption}\label{ass:random_centroids} We assume that $(i)$ the coordinates of  $C_1-C_2$ are i.i.d., $(ii)$ for all $j\in \{1,..., d\}$, $\e{ (C_1-C_2)_j^p}<\infty$ and $(iii)$ conditional to the centers $C_1$ and $C_2$, the radii $R_1$ is uniformly distributed as $R_1|(C_1,C_2)   \sim \mathcal{U}(0, \frac{1}{2}\left\|C_1-C_2\right\|_p)$, with $R_2 = R_1$.
\end{assumption}
\Cref{ass:random_centroids} trivially includes the cases where $(C_1,C_2) \sim \mathcal{N}(\mu_1,\lambda_1 I_d)\otimes\mathcal{N}(\mu_2,\lambda_2 I_d)$, or where $(C_1,C_2)\sim \mathcal{U}(a_1, b_1)^{\otimes d}\otimes\mathcal{U}(a_2, b_2)^{\otimes d}$.

%
\begin{assumption}[Uniform $s$-missing patterns]\label{ass:M-uniform} The missing pattern $M$ is sampled uniformly at random among missing patterns admitting $s$ missing values in total, i.e., 
$M \sim \mathcal{U}(\left\{m\in \{0,1\}^d,\left\|m\right\|_0=s \right\})$.
\end{assumption}
%
In the next proposition, we characterize the probability of preserving linear separability despite missing values, when the dimension $d$ tends to $\infty$.
\begin{propo}[Asymptotic separability of two balls with the same radius]\label{prop:asymptBallSepar}
    Under Assumption \ref{ass:random_centroids} and Assumption \ref{ass:M-uniform}, let $\rho := \lim_{d \to \infty}\frac{s}{d}$. Then,
    \begin{equation}
        \lim_{d\to +\infty}\P\left(B_{1,\mathrm{obs}(M)}\cap B_{2,\mathrm{obs}(M)} =\emptyset\right)=\sqrt[p]{1-\rho}.
    \end{equation}

\end{propo}
    Therefore, in high-dimensional regimes, pattern-by-pattern perceptron is a valid procedure with a probability converging to $\sqrt[p]{1-\rho}$, where $\rho$ is the asymptotic ratio of missing values. 
    Note that when $s/d$ tends to zero, as $s$ and $d$ tend to infinity, the separability of the balls is ensured with probability 1. 
    Besides this asymptotic separability probability $\sqrt[p]{1-\rho}$ increases when $p$ increases. 
    This is due to the fact that when $p$ increases, the radius $R_1|(C_1,C_2)\sim \mathcal{U}(0, \frac{1}{2}\left\|C_1-C_2\right\|_p)$ is shrinked ($p \mapsto\left\|x\right\|_{p}$ is non-increasing) and the balls are more and more separated. 




\subsection{Proof of Proposition \ref{prop:asymptBallSepar}}\label{subsec:asymptBallSepar}

\begin{proof} 
    In order to study the separability of the two balls after projection through the missing pattern, we need to study the probability that the sum of radii is still smaller than the distance between the two centers after projection as shown in Lemma \ref{lem: sepChar}. 
    Since $R:=R_1=R_2$, this probability corresponds to
    \begin{align*}
        \P\left(R<\frac{1}{2}\left\|(1-M)\odot (C_1-C_2)\right\|_p \right).
    \end{align*}
    Using Assumption \ref{ass:random_centroids}, we have that 
    \begin{align*}
        \P\left(R<\frac{1}{2}\left\|(1-M)\odot (C_1-C_2)\right\|_p |M, C_1, C_2\right) &=\frac{\left\|(1-M)\odot (C_1-C_2)\right\|_p}{\left\|(C_1-C_2)\right\|_p}.
    \end{align*}
    Therefore, if we define $\mathcal{M}_s=\left\{ m\in\{0,1\}^d, \left\|m \right\|_0=s \right\}$,
    \begin{align*}
    \P\left(R<\frac{1}{2}\left\|(1-M)\odot (C_1-C_2)\right\|_p \right) &= \e{\frac{\left\|(1-M)\odot (C_1-C_2)\right\|_p}{\left\|(C_1-C_2)\right\|_p}}\\
        &=\e{\e{\frac{\left\|(1-M)\odot (C_1-C_2)\right\|_p}{\left\|(C_1-C_2)\right\|_p}\mid C_1, C_2}} \\
        \\
        &=\e{\e{\sqrt[p]{\frac{\sum_{j=1}^d(1-M_j) (C_{1j}-C_{2j})^p}{\sum_{j=1}^d (C_{1j}-C_{2j})^p}}\mid C_1, C_2}}\\
        &=\e{\sum_{m\in \mathcal{M}_s}\frac{1}{\binom{d}{s}}\sqrt[p]{\frac{\sum_{j, m_j=0}(C_{1j}-C_{2j})^p}{\sum_{j=1}^d (C_{1j}-C_{2j})^p}}}\tag{using $M\sim \mathcal{U}(\mathcal{M}_s)$}\\
        &=\e{\sqrt[p]{\frac{\sum_{j=1}^{d-s}(C_{1j}-C_{2j})^p}{\sum_{j=1}^d (C_{1j}-C_{2j})^p}}}\\
    \end{align*}
    after having reordered the terms using the exchangeability of the $(C_1-C_2)_j$ (Assumption \ref{ass:random_centroids}(i)). One has 
    \begin{align*}
        \frac{\sum_{j=1}^{d-s}(C_{1j}-C_{2j})^p}{\sum_{j=1}^d (C_{1j}-C_{2j})^p} &= \frac{\frac{1}{d-s}\sum_{j=1}^{d-s}(C_{1j}-C_{2j})^p}{\frac{1}{d-s}\sum_{j=1}^{d-s} (C_{1j}-C_{2j})^p+\frac{s}{d-s}\frac{1}{s}\sum_{j=d-s+1}^{d} (C_{1j}-C_{2j})^p}\\
        &=\frac{1}{1+\frac{\frac{s}{d-s}\frac{1}{s}\sum_{j=d-s+1}^{d} (C_{1j}-C_{2j})^p}{\frac{1}{d-s}\sum_{j=1}^{d-s}(C_{1j}-C_{2j})^p}}.
    \end{align*}
    As $d$ goes to infinity, we assume that the number of missing values $s$ goes to infinity.
    Otherwise, if $s$ is bounded, then $\rho = \lim_{d \to \infty}\frac{s}{d}=0$ and $\frac{s}{d-s}\frac{1}{s}\sum_{j=d-s+1}^{d} (C_{1j}-C_{2j})^p \xrightarrow{d\to\infty} 0$, so we would have the result using that $$\P\left(R<\frac{1}{2}\left\|(1-M)\odot (C_1-C_2)\right\|_p \right)\xrightarrow[d\to\infty]{} 1=\sqrt[p]{1-\rho}.$$

    Then, combining Assumption \ref{ass:random_centroids}  and the law of large numbers, we get
    \begin{align*}
        \frac{1}{d-s}\sum_{j=1}^{d-s}(C_{1j}-C_{2j})^p \xrightarrow[d\to\infty]{\P} \e{(C_{11}-C_{21})^p} \\
        \frac{1}{s}\sum_{j=d-s+1}^{d} (C_{1j}-C_{2j})^p \xrightarrow[d\to\infty]{\P} \e{(C_{11}-C_{21})^p}.
    \end{align*}
    Using Slutsky's theorem, 
    \begin{align*}
        \frac{s}{d-s}\frac{1}{s}\sum_{j=d-s+1}^{d} (C_{1j}-C_{2j})^p \xrightarrow[d\to\infty]{\P} \frac{\rho}{1-\rho}(\e{(C_{11}-C_{21})^p}).
    \end{align*}
    Re-using Slutsky's theorem, 
    \begin{align*}
        \frac{\frac{s}{d-s}\frac{1}{s}\sum_{j=d-s+1}^{d} (C_{1j}-C_{2j})^p}{\frac{1}{d-s}\sum_{j=1}^{d-s}(C_{1j}-C_{2j})^p} \xrightarrow[d\to\infty]{\P} \frac{\rho}{1-\rho}.
    \end{align*}
    Finally, using the continuous mapping theorem, we have that 
    \begin{align*}
        \P\left(R<\frac{1}{2}\left\|(1-M)\odot (C_1-C_2)\right\|_p \right)\underset{d\to\infty}{\longrightarrow} \sqrt[p]{1-\rho}.
    \end{align*}
\end{proof}

\subsection{Link between non-optimality of P-b-P and constant imputation approaches}

\begin{proof}[Proof of \Cref{lemma_P-b-P_imputation_negative}]
Recall that $Z := (X_{\mathrm{obs}(M)}, M)$. A Pattern-by-Pattern (P-b-P) approach with linear classifier $h_{\textrm{PbP}}$ is defined as,
\begin{align*}
 h_{\textrm{PbP}}(Z) & =  \sum_{m\in\M}h_m(X_{\mathrm{obs}(m)})\ind_{M=m},
\end{align*}
where,  for all $m \in \{0,1\}^d$, there exists $\beta_{j,m}$, such that
\begin{align}
    h_m(X_{\mathrm{obs}(m)}) & = \beta_{0,m} + \sum_{j \in \obs(m)} \beta_{j,m} X_j\\
    & = \beta_{0,m} + \sum_{j =1}^d \beta_{j,m} X_j \mathds{1}_{M_j=0}. \label{eq_proof_PbP_imputation1}
\end{align}
On the contrary, an imputation strategy that consists in replacing missing values in $X_j$ by $\alpha_j$ for all $1\leq j \leq d$ can be written as 
\begin{align}
    h_{\textrm{imp}}(Z) & = \beta_0 + \sum_{j =1}^d  \beta_{j} ( X_j \mathds{1}_{M_j=0} + \alpha_j \mathds{1}_{M_j=1})\\
    & = \beta_0 + \sum_{j =1}^d  \beta_{j} \alpha_j \mathds{1}_{M_j=1} + \sum_{j =1}^d \beta_j X_j \mathds{1}_{M_j=0}. \label{eq_proof_PbP_imputation2}
\end{align}
Thus, comparing \eqref{eq_proof_PbP_imputation1} and \eqref{eq_proof_PbP_imputation2}, we see that a constant imputation strategy with linear classifiers can always be written as a P-b-P approach with linear classifiers. Thus, assuming that P-b-P approach with linear classifiers are not Bayes optimal, we conclude that constant imputation strategies with linear classifiers are not Bayes optimal either. 
\end{proof}

\section{(Logistic Model) Proof of Proposition \ref{prop:NotLogis}}
\label{subsec:proofPropNotLogis} 
\begin{proof} 
Let $m\in\{0,1\}^d,$
\begin{align}
   \prob{Y=1|X_{\mathrm{obs}(m)},M=m}   &=\prob{Y=1|X_{\mathrm{obs}(m)}}\tag{using Assumption \ref{ass:MCAR}}\\
&=\e{\prob{Y=1|X}|X_{\mathrm{obs}(m)}}\\
  &= \e{\frac{1}{1+\exp(-\beta_{0}^{\star} - \sum_{j=1}^d \beta_j^{\star} X_j)}|X_{\mathrm{obs}(m)}}.
  \end{align}
  Now, assume that there exists $\beta_m^{\star}\in \R^{d-\left\|m\right\|_0} $ such that 
  \begin{align}
        \prob{Y=1|X_{\mathrm{obs}(m)},M=m} &=\frac{1}{1+\exp (- \beta_{0,m}^\star - \sum\limits_{j \in \mathrm{obs}(m)} \beta_{j,m}^\star X_j)}.
  \end{align}
  Combining the two previous equations leads to 
  \begin{align}
        \notag
		& \frac{1}{1+\exp (- \beta_{0,m}^\star - \sum\limits_{j \in \mathrm{obs}(m)} \beta_{j,m}^\star X_j)}\\
        &=\e{\frac{1}{1+\exp(-\beta_{0}^{\star} - \sum_{j=1}^d \beta_j^{\star} X_j)}|X_{\mathrm{obs}(m)}} \label{eq_proof1}\\ 
        &\geq \frac{1}{\e{1+\exp(-\beta_{0}^{\star} - \sum_{j=1}^d \beta_j^{\star} X_j) |X_{\mathrm{obs}(m)} } }\tag{using Jensen Inequality}\\
        &=\frac{1}{1+\e{\exp\left(-\beta_0^{\star} - \sum\limits_{j \in \mathrm{obs}(m)} \beta_{j}^\star X_j - \sum\limits_{j \in \mis(m)} \beta_{j}^\star X_j  \right)\mid X_{\mathrm{obs}(m)}}}\\
        &=\frac{1}{1+\exp\left(-\beta_0^{\star} - \sum\limits_{j \in \mathrm{obs}(m)} \beta_{j}^\star X_j \right) \e{\exp\left(- \sum\limits_{j \in \mis(m)} \beta_{j}^\star X_j\right)\mid X_{\mathrm{obs}(m)}}},\\
  \end{align}
  which is equivalent to
  \begin{align}
      \exp \left(- (\beta_{0,m}^\star - \beta_0^\star ) - \sum\limits_{j \in \mathrm{obs}(m)} (\beta_{j,m}^\star - \beta_j^\star) X_j \right) \leq 
      \e{\exp\left(- \sum\limits_{j \in \mis(m)} \beta_{j}^\star X_j\right)\mid X_{\mathrm{obs}(m)}}.\label{eq:logitMCAR}
  \end{align}
  
   Now, assuming that variables $X_1, \hdots, X_d$ are independent, we have
   \begin{align}
      \exp \left(- (\beta_{0,m}^\star - \beta_0^\star ) - \sum\limits_{j \in \mathrm{obs}(m)} (\beta_{j,m}^\star - \beta_j^\star) X_j \right) \leq 
      \e{\exp\left(- \sum\limits_{j \in \mis(m)} \beta_{j}^\star X_j\right) }.
  \end{align}
   Let $1 \leq k \leq d$. Letting $X_j = 0$ for all $j \in \mathrm{obs}(m) $ with $j \neq k$, we have
    \begin{align}
      \exp \left(- (\beta_{0,m}^\star - \beta_0^\star ) -  (\beta_{k,m}^\star - \beta_k^\star) X_k \right) \leq 
      \e{\exp\left(- \sum\limits_{j \in \mis(m)} \beta_{j}^\star X_j\right) }.
  \end{align}
 By assumption, the support of $X_k$ is $\mathds{R}$. Thus, letting $X_k$ tending to $\pm \infty$, we deduce that 
 \begin{align}
     \beta_{k,m}^\star = \beta_k^\star.
 \end{align}
Injecting this into \eqref{eq_proof1} leads to 
\begin{align}
    \frac{1}{1+\exp (- \beta_{0,m}^\star - \sum\limits_{j \in \mathrm{obs}(m)} \beta_{j}^\star X_j)}
        &=\e{\frac{1}{1+\exp(-\beta_{0}^{\star} - \sum_{j=1}^d \beta_j^{\star} X_j)}|X_{\mathrm{obs}(m)}},
\end{align}
  that is 
\begin{align}
    \e{\frac{1+\exp (- \beta_{0,m}^\star - \sum\limits_{j \in \mathrm{obs}(m)} \beta_{j}^\star X_j)}{1+\exp(-\beta_{0}^{\star} - \sum_{j=1}^d \beta_j^{\star} X_j)}|X_{\mathrm{obs}(m)}} = 1. \label{eq_proof2}
\end{align}
Let 
\begin{align}
u = \exp \left( - \sum\limits_{j \in \mathrm{obs}(m)} \beta_{j}^\star X_j \right) \quad \textrm{and} \quad   Z_{\mis(m)} = \exp \left( - \sum\limits_{j \in \mis(m)} \beta_{j}^\star X_j \right).
\end{align} 
According to \eqref{eq_proof2}, for all $u \in (0, \infty)$,
\begin{align}
    \e{\frac{1+u \exp (- \beta_{0,m}^\star)}{1+ u Z_{\mis(m)} \exp(-\beta_{0}^{\star})  }} = 1. \label{eq_proof3}
\end{align}
Assume that $\e{1/Z_{\mis(m)}}$ exists. Take the limit when $u$ tends to infinity. According to Lebesgue dominated convergence theorem, we have
\begin{align}
    \lim\limits_{u \to \infty} \e{\frac{1+u \exp (- \beta_{0,m}^\star)}{1+ u Z_{\mis(m)} \exp(-\beta_{0}^{\star})  }} 
    & =  \e{ \lim\limits_{u \to \infty} \frac{1+u \exp (- \beta_{0,m}^\star)}{1+ u Z_{\mis(m)} \exp(-\beta_{0}^{\star})  }} \\
    & = \e{   \frac{ \exp (- \beta_{0,m}^\star)}{ Z_{\mis(m)} \exp(-\beta_{0}^{\star})  }}.\end{align}
  Thus, 
\begin{align}
    \e{   \frac{ 1}{ Z_{\mis(m)}   }} = \exp ( \beta_{0,m}^\star - \beta_{0}^{\star}).
\end{align}
  By definition of $ Z_{\mis(m)}$, we have
\begin{align}
  \exp ( \beta_{0,m}^\star - \beta_{0}^{\star}) & =   \e{ \frac{ 1}{ \prod\limits_{j \in \mis(m)} \exp \left( - \beta_j^{\star} X_j \right) }} \\
  & =   \e{  \prod\limits_{j \in \mis(m)} \exp \left(  \beta_j^{\star} X_j \right) } \\
& =  \prod\limits_{j \in \mis(m)}  \e{   \exp \left(  \beta_j^{\star} X_j \right) }.
\end{align}
Thus, 
\begin{align}
    \exp ( - \beta_{0,m}^\star) & = \frac{\exp( - \beta_{0}^{\star}) }{ \prod\limits_{j \in \mis(m)}  \e{   \exp \left(  \beta_j^{\star} X_j \right) }}\\
    & = \frac{\exp( - \beta_{0}^{\star}) }{  \e{Z_{\mis(m)}' }},
\end{align}
where 
\begin{align}
Z_{\mis(m)}' = 1/Z_{\mis(m)} = \exp \left( \sum\limits_{j \in \mis(m)} \beta_{j}^\star X_j \right).    
\end{align} 
Injecting this equality into \eqref{eq_proof3} leads to, for all $u \in (0,\infty)$, 
\begin{align}
   &  \e{\frac{1+u \exp( - \beta_{0}^{\star})/ \e{Z_{\mis(m)}' }}{1+ u  \exp(-\beta_{0}^{\star})/Z_{\mis(m)}'  }} = 1 \\
    \Longleftrightarrow \quad & \e{\frac{\e{Z_{\mis(m)}' }+u \exp( - \beta_{0}^{\star}) }{\e{Z_{\mis(m)}' }+ u \e{Z_{\mis(m)}' } \exp(-\beta_{0}^{\star})/Z_{\mis(m)}'  }} = 1\\
     \Longleftrightarrow \quad & \e{\frac{\e{Z_{\mis(m)}' }+ v }{\e{Z_{\mis(m)}' }+ v \e{Z_{\mis(m)}' }/Z_{\mis(m)}'  }} = 1.
\end{align}
where $v = u \exp( - \beta_{0}^{\star})$. As this holds for all $v \in (0, \infty)$, taking the derivative of the expectation leads to, for all $v \in (0, \infty)$, 
\begin{align}
    \e{\frac{\e{Z_{\mis(m)}'} \left(1 - \frac{\e{Z_{\mis(m)}'}}{Z_{\mis(m)}'} \right)}{\left(\e{Z_{\mis(m)}'} + \frac{\e{Z_{\mis(m)}'}}{Z_{\mis(m)}'} v \right)^2}} = 0.
\end{align}
Letting $v$ tend to zero leads to
\begin{align}
    \e{\frac{1}{Z_{\mis(m)}'}} = \frac{1}{\e{Z_{\mis(m)}'}}, 
\end{align}
  which holds only if the random variable $Z_{\mis(m)}'$ is degenerated. By definition of $Z_{\mis(m)}'$, we deduce that for all $j \in \mis(m)$, $X_j$ is degenerated or $\beta_j^\star = 0$. Since the support of $X_j$ is $\mathds{R}$, we have that $\beta_j^\star = 0$.

\end{proof}

\section{(LDA + MCAR) Proofs of Section \ref{subsubsec:LDAMissingValues}}\label{subsec:proofsSettingMCARLDA}

\subsection{Preliminary}
\label{subsec:results_on_LDA}

\begin{lemma}
\label{lem_link_log_LDA}

\Cref{ass:LDA} (LDA) is equivalent to 
\begin{enumerate}

    \item[$(i)$] The following logistic regression model 
    \begin{align}
    \mathds{P}[Y=1|X] = \sigma (\log(\pi C) + (\mu_1 - \mu_{-1})^\top \Sigma^{-1} x),    
    \end{align}
    with $ C = \exp\left(\frac{1}{2} \mu_{-1}^\top \Sigma^{-1} \mu_{-1} - \frac{1}{2} \mu_1^\top \Sigma^{-1} \mu_1 \right)$ and $\pi = \mathds{P}[Y=1] / \mathds{P}[Y=-1]$, 

    \item[$(ii)$] and $X$ is distributed as a Gaussian mixture $X = \pi_{-1} Z_0 + (1 - \pi_0) Z_1,$ where $Z_k \sim \mathcal{N}(\mu_k, \Sigma)$ for all $k \in \{0,1\}$. 
\end{enumerate}
\end{lemma}

\begin{proof}[Proof of \Cref{lem_link_log_LDA}]
Grant \Cref{ass:LDA}. Let $f_k$ be the density of $X|Y=k$. Let $ \pi = \mathds{P}[Y=1] / \mathds{P}[Y=0]$. We have
\begin{align}
    \frac{\mathds{P}[Y=1|X=x]}{\mathds{P}[Y=0|X=x]} 
    & = \pi \frac{f_1(x)}{f_0(x)}  \\
    & = \pi \exp\left(- \frac{1}{2} (x - \mu_1)^\top \Sigma^{-1} (x-\mu_1) + \frac{1}{2} (x - \mu_0)^\top \Sigma^{-1} (x - \mu_0) \right)\\
     & = \pi C \exp \left( (\mu_1 - \mu_0)^\top \Sigma^{-1} x\right),
 \end{align}
with $ C = \exp\left(\frac{1}{2} \mu_0^\top \Sigma^{-1} \mu_0 - \frac{1}{2} \mu_1^\top \Sigma^{-1} \mu_1 \right)$. Consequently, 
\begin{align}
    \log \left( \frac{\mathds{P}[Y=1|X=x]}{\mathds{P}[Y=0|X=x]}\right) & = \log (\pi C ) + (\mu_1 - \mu_0)^\top \Sigma^{-1} x,
\end{align}
    which concludes the first part of the proof.     Now, assume that $(i)$ and $(ii)$ of \Cref{lem_link_log_LDA} hold. Recall that $f(x), f_0(x), f_1(x)$ are respectively the density of $X$, $X|Y=0$ and $X|Y=1$. We have
\begin{align}
    f_0(x) = \frac{\mathds{P}[Y=0 | X= x] f(x)}{\mathds{P}[Y=0]}. \label{eq_proof_lda_logistic_1}
\end{align}
Since $X$ is a mixture of Gaussian, we have
\begin{align}
    f(x) & = C' \exp\left( - \frac{1}{2} x^\top \Sigma^{-1} x \right) \left( \pi_0 c_0 \exp\left( x^\top \Sigma^{-1} \mu_0 \right) + (1 - \pi_0) c_1 \exp\left( x^\top \Sigma^{-1} \mu_1 \right) \right) \\
    & = C' \pi_0 c_0 \exp\left( - \frac{1}{2} x^\top \Sigma^{-1} x \right) \exp\left( x^\top \Sigma^{-1} \mu_0 \right) \left( 1 + \frac{(1 - \pi_0)c_1}{\pi_0 c_0} \exp\left( x^\top \Sigma^{-1} (\mu_1-\mu_0)\right) \right), \label{eq_proof_lda_logistic_2}
\end{align}
with $C' = (2 \pi)^{-d/2} (\textrm{det}(\Sigma))^{-1/2}$, $c_0 = \exp(-\frac{1}{2} \mu_0^\top \Sigma^{-1} \mu_0)$ and $c_1 = \exp(-\frac{1}{2} \mu_1^\top \Sigma^{-1} \mu_1)$.
Besides, by assumption, we have
\begin{align}
    \mathds{P}[Y=0|X=x]  &= \frac{1}{1 + \exp\left(\log(\pi C) + (\mu_1 - \mu_0)^\top \Sigma^{-1} x\right)} \\
    & = \frac{1}{1 + \pi C \exp\left((\mu_1 - \mu_0)^\top \Sigma^{-1} x\right)}. \label{eq_proof_lda_logistic_3}
\end{align}
Gathering \eqref{eq_proof_lda_logistic_2} and \eqref{eq_proof_lda_logistic_3} in \eqref{eq_proof_lda_logistic_1}, we obtain
\begin{align}
   f_0(x) & =  C'  c_0 \exp\left( - \frac{1}{2} x^\top \Sigma^{-1} x \right) \exp\left( x^\top \Sigma^{-1} \mu_0 \right)  
   \frac{1 + \frac{(1 - \pi_0)c_1}{\pi_0 c_0} \exp\left( x^\top \Sigma^{-1} (\mu_1-\mu_0)\right)}{1 + \pi C \exp\left((\mu_1 - \mu_0)^\top \Sigma^{-1} x\right)}  \\
   & = C'  c_0 \exp\left( - \frac{1}{2} x^\top \Sigma^{-1} x \right) \exp\left( x^\top \Sigma^{-1} \mu_0 \right)  \\
   & = (2 \pi)^{-d/2} (\textrm{det}(\Sigma))^{-1/2} \exp\left(- \frac{1}{2} (x - \mu_0)^\top \Sigma^{-1} (x-\mu_0) \right),
\end{align}
since $C = c_1/c_0$ and $\pi = (1 - \pi_0)/\pi_0$. Consequently, $X|Y=0$ follows a multivariate Gaussian distribution $\mathcal{N}(\mu_0, \Sigma)$. The same calculus can be carried out for the distribution of $X|Y=1$, which concludes the proof. 
\end{proof}

The Bayes predictor $h^\star_{\mathrm{comp}}$ satisfies 
\begin{align}\label{eq:BalancedComplBayesRisk}
    \mathcal{R}_{\mathrm{comp}}(h^\star_{\mathrm{comp}})=
    \Phi\left(-a_m - b_m \right)\pi_{-1}  +\Phi\left( a_m - b_m \right),\pi_1
\end{align}
where $\Phi(x) = \mathds{P}[\mathcal{N}(0,1) \leq x]$ is the c.d.f. of a standard Gaussian random variable, $a_m =  \log\left(\frac{\pi_{-1}}{\pi_1}\right)/\|\Sigma^{-\frac{1}{2}}(\mu_{1}-\mu_{-1})\|$ and $b_m = \|\Sigma^{-\frac{1}{2}}(\mu_{1}-\mu_{-1})\|/2$.

\begin{corollary}[Bayes Risk of P-b-P LDA]\label{cor:GenBayesRiskpbpLDA}
Under Assumptions \ref{ass:MCAR} and \ref{ass:LDA}, the Bayes risk is given by  
\begin{align}
    \mathcal{R}_{\mathrm{mis}}(h^\star)&=\sum_{m\in\{0,1\}^d}\Phi\left(-a_m - b_m \right)\pi_{-1}p_m +\Phi\left( a_m - b_m \right)\pi_1p_m, 
    \label{Bayes_risk_general_case_LDA_missing_MCAR}
\end{align}
where, for all $m \in \mathcal{M}$, 
\begin{align}
    a_m = \frac{\log\left(\frac{\pi_{-1}}{\pi_1}\right)}{\left\|\Sigma_{\mathrm{obs}(m)}^{-\frac{1}{2}}(\mu_{1, \mathrm{obs}(m)}-\mu_{-1, \mathrm{obs}(m)})\right\| } \quad \textrm{and} \quad b_m = \frac{\left\|\Sigma_{\mathrm{obs}(m)}^{-\frac{1}{2}}(\mu_{1, \mathrm{obs}(m)}-\mu_{-1, \mathrm{obs}(m)})\right\|}{2}
\end{align}
\end{corollary}
The proof can be found in Appendix~\ref{subsubsec_cor:GenBayesRiskpbpLDA}. 
Note that, from \cref{cor:GenBayesRiskpbpLDA} (using that $\pi_1=\pi_{-1}$) and Equation  \eqref{eq:BalancedComplBayesRisk}, we have that 
\begin{align}
\notag
    L&(h^\star)-\mathcal{R}_{\mathrm{comp}}(h^\star_{\mathrm{comp}})\\
    &=\sum_{m\in\{0,1\}^d}\left(\Phi\left(-\frac{\left\|\Sigma^{-\frac{1}{2}}_{\mathrm{obs}(m)}(\mu_{1, \mathrm{obs}(m)}-\mu_{-1, \mathrm{obs}(m)})\right\|}{2}\right)-\Phi\left(-\frac{\left\|\Sigma^{-\frac{1}{2}}(\mu_1-\mu_{-1})\right\|}{2}\right)\right)p_m,
    \label{eq:firstDifferRisk}
\end{align}
with $\Phi$ the c.d.f. of a standard Gaussian variable. 

\subsection{Proof of Proposition \ref{prop:MCARLDA}}\label{subsubsec:MCARLDA}
\begin{proof}
    Expanding \eqref{eq:BayesClassifPbP},  
    \begin{align}
    \notag
        h_m^{\star}(X_{\mathrm{obs}(m)})&=\mathrm{sign}(\e{Y|X_{\mathrm{obs}(m)},M=m})\\
        &=\mathrm{sign}\left( \P\left(Y=1\mid  X_{\mathrm{obs}(m)}, M=m\right)-\P\left(Y=-1\mid  X_{\mathrm{obs}(m)}, M=m\right)\right).\label{eq:interLDAMCAR}
    \end{align}
    Note that, for any Borelian $B \subset \mathds{R}^{|\mathrm{obs}(m)|}$,
    \begin{align*}
        \P\left(Y=k\mid  X_{\mathrm{obs}(m)}\in B, M=m\right)&= \frac{\P\left(Y=k,  X_{\mathrm{obs}(m)}\in B\mid  M=m\right)}{\P\left( X_{\mathrm{obs}(m)}\in B\mid  M=m\right)}\\
        &= \frac{\P\left(Y=k,  X_{\mathrm{obs}(m)}\in B\right)}{\P\left( X_{\mathrm{obs}(m)}\in B\right)}\tag{using Assumption \ref{ass:MCAR}}\\
        &=\frac{\P\left(X_{\mathrm{obs}(m)}\in B\mid  Y=k\right)\pi_k}{\P\left( X_{\mathrm{obs}(m)}\in B\right)}.
    \end{align*}
    Thus, 
    \begin{align}
    & \P\left(Y=1\mid  X_{\mathrm{obs}(m)}\in B, M=m\right) > \P\left(Y=-1\mid  X_{\mathrm{obs}(m)}\in B, M=m\right) \\
    \Longleftrightarrow & \quad \P\left(X_{\mathrm{obs}(m)}\in B\mid  Y=1\right)\pi_1 > \P\left(X_{\mathrm{obs}(m)}\in B\mid  Y=-1\right)\pi_{-1}.
    \end{align}
    As this holds for any Borelian  $B \subset \mathds{R}^{|\mathrm{obs}(m)|}$, $h_m^{\star}$ can be rewritten as 
    \begin{align}
        h_m^{\star}(x) & =  \textrm{sign}\left( \pi_1 f_{X_{\mathrm{obs}(m)} | Y=1}(x) - \pi_{-1} f_{X_{\mathrm{obs}(m)} | Y=1}(x) \right)\\
        & = \textrm{sign}\left( \log\left(\frac{f_{X_{\mathrm{obs}(m)} | Y=1}(x)}{f_{X_{\mathrm{obs}(m)} | Y=-1}(x)}\right) - \log\left(\frac{\pi_{-1}}{\pi_1}\right) \right), 
    \end{align}
    where $f_{X_{\mathrm{obs}(m)}\mid  Y=k}$ is the density of 
    $X_{\mathrm{obs}(m)}\mid  Y=k$ for all $k \in \{-1,1\}$. 
    Under LDA model (Assumption \ref{ass:LDA}), the objective is to determine the distribution of $X_{\mathrm{obs}(m)}|Y=k$ for each $m\in \{0,1\}^d$. To this end, ~\cref{lemma:obsGaus} proves that the projection of a Gaussian vector onto a subset of coordinates preserves the Gaussianity with projected parameters. Hence, $X_{\mathrm{obs}(m)}|Y=k \sim \mathcal{N}(\mu_{k, \mathrm{obs}(m)},\Sigma_{\mathrm{obs}(m)})$ and therefore,
    \begin{align*}
         & \log  \left(\frac{f_{X_{\mathrm{obs}(m)} | Y=1}(x)}{f_{X_{\mathrm{obs}(m)} | Y=-1}(x)}\right) \\&= \log\left( \frac{(\sqrt{2\pi})^{-(d-\left\|m\right\|_0)}\sqrt{\det(\Sigma^{-1}_{\mathrm{obs}(m)})}\exp\left(-\frac{1}{2}(x-\mu_{1, \mathrm{obs}(m)})^{\top}\Sigma^{-1}_{\mathrm{obs}(m)}(x-\mu_{1,\mathrm{obs}(m)})\right)}{(\sqrt{2\pi})^{-(d-\left\|m\right\|_0)}\sqrt{\det(\Sigma^{-1}_{\mathrm{obs}(m)})}\exp\left(-\frac{1}{2}(x-\mu_{-1, \mathrm{obs}(m)})^{\top}\Sigma^{-1}_{\mathrm{obs}(m)}(x-\mu_{-1,\mathrm{obs}(m)})\right)} \right)\\
        &=-\frac{1}{2}(x-\mu_{1, \mathrm{obs}(m)})^{\top}\Sigma^{-1}_{\mathrm{obs}(m)}(x-\mu_{1,\mathrm{obs}(m)})+\frac{1}{2}(x-\mu_{-1, \mathrm{obs}(m)})^{\top}\Sigma^{-1}_{\mathrm{obs}(m)}(x-\mu_{-1,\mathrm{obs}(m)})\\
        &=(\mu_{1,\mathrm{obs}(m)}-\mu_{-1,\mathrm{obs}(m)})^{\top}\Sigma^{-1}_{\mathrm{obs}(m)}\left(x-\frac{\mu_{1,\mathrm{obs}(m)}+\mu_{-1,\mathrm{obs}(m)}}{2}\right). 
    \end{align*}
    Consequently, 
     \begin{align}
        h_m^{\star}(x) 
        & = \textrm{sign}\left( (\mu_{1,\mathrm{obs}(m)}-\mu_{-1,\mathrm{obs}(m)})^{\top}\Sigma^{-1}_{\mathrm{obs}(m)}\left(x-\frac{\mu_{1,\mathrm{obs}(m)}+\mu_{-1,\mathrm{obs}(m)}}{2}\right) - \log\left(\frac{\pi_{-1}}{\pi_1}\right) \right), 
    \end{align}
    which concludes the proof. 
\end{proof}

\subsection{Proof of \cref{cor:GenBayesRiskpbpLDA}}
\label{subsubsec_cor:GenBayesRiskpbpLDA}

\begin{proof}
Let $N=\Sigma_{\mathrm{obs}(m)}^{-\frac{1}{2}}(X_{\mathrm{obs}(m)}-\mu_{-1, \mathrm{obs}(m)})$. Using Proposition \ref{prop:MCARLDA}, we have 
   \begin{align*}
        & \P\left(h_m^{\star}(X_{\mathrm{obs}(m)})=1\mid Y=-1\right) \\
          = & \P\Big( \left(\mu_{1, \mathrm{obs}(m)}-\mu_{-1, \mathrm{obs}(m)}\right)^{\top}\Sigma_{\mathrm{obs}(m)}^{-1}\left(X_{\mathrm{obs}(m)}-\frac{\mu_{1, \mathrm{obs}(m)}+\mu_{-1,\mathrm{obs}(m)}}{2}\right) \\
          & \qquad - \log\left(\frac{\pi_{-1}}{\pi_1}\right)>0\mid Y=-1\Big) \\
         = &  \P\left(\gamma^{\top}N-\frac{1}{2}\left\|\gamma\right\|^2>\log\left(\frac{\pi_{-1}}{\pi_1}\right)\mid Y=-1\right),
    \end{align*}
    where $\gamma=\Sigma_{\mathrm{obs}(m)}^{-\frac{1}{2}}(\mu_{1,\mathrm{obs}(m)}-\mu_{-1,\mathrm{obs}(m)})$. By~\cref{lemma:obsGaus}, $N|Y=-1\sim\mathcal{N}(0, Id_{d-\left\|m\right\|_0})$. Thus, 
   \begin{align*}
       \P\left(h_m^{\star}(X_{\mathrm{obs}(m)})=1\mid Y=-1\right)
       &=\P\left(\frac{\gamma^{\top}N}{\left\|\gamma\right\|}>\frac{1}{2}\left\|\gamma\right\|+\frac{1}{\left\|\gamma\right\|}\log\left(\frac{\pi_{-1}}{\pi_1}\right)\mid Y=-1\right)\\
       &=\Phi\left(-\frac{1}{2}\left\|\gamma\right\|-\frac{1}{\left\|\gamma\right\|}\log\left(\frac{\pi_{-1}}{\pi_1}\right)\right).
   \end{align*}    
   Similarly, letting $N' =\Sigma_{\mathrm{obs}(m)}^{-\frac{1}{2}}(X_{\mathrm{obs}(m)}-\mu_{1, \mathrm{obs}(m)})$,
   \begin{align*}
        & \P\left(h_m^{\star}(X_{\mathrm{obs}(m)})=-1\mid Y=1\right) \\
          = & \P\Big( \left(\mu_{1, \mathrm{obs}(m)}-\mu_{-1, \mathrm{obs}(m)}\right)^{\top}\Sigma_{\mathrm{obs}(m)}^{-1}\left(X_{\mathrm{obs}(m)}-\frac{\mu_{1, \mathrm{obs}(m)}+\mu_{-1,\mathrm{obs}(m)}}{2}\right) \\
          & \qquad - \log\left(\frac{\pi_{-1}}{\pi_1}\right)<0\mid Y= 1\Big) \\
         = &  \P\left(\gamma^{\top}N + \frac{1}{2}\left\|\gamma\right\|^2<\log\left(\frac{\pi_{-1}}{\pi_1}\right)\mid Y=1\right)\\
         = & \P\left(\frac{\gamma^{\top}N}{\left\|\gamma\right\|} > \frac{1}{2}\left\|\gamma\right\|-\frac{1}{\left\|\gamma\right\|}\log\left(\frac{\pi_{-1}}{\pi_1}\right)\mid Y=-1\right)\\
       = & \Phi\left(-\frac{1}{2}\left\|\gamma\right\|+\frac{1}{\left\|\gamma\right\|}\log\left(\frac{\pi_{-1}}{\pi_1}\right)\right).
    \end{align*}
Finally, 
\begin{align*}
        & \mathcal{R}_{\mathrm{mis}}(h^\star)\\
        &=\P\left(h^\star(X_{\mathrm{obs}(M)},M)\neq Y\right)\\
        &=\sum_{m\in\{0,1\}^d}\P\left(h^\star(X_{\mathrm{obs}(m)},M)\neq Y\mid M=m\right)p_m\\
         &=\sum_{m\in\{0,1\}^d}\P\left(h^\star_m(X_{\mathrm{obs}(m)})\neq Y\right)p_m\tag{using Assumption \ref{ass:MCAR}}\\
         &=\sum_{m\in\{0,1\}^d}\P\left(h^\star_m(X_{\mathrm{obs}(m)})=-1\mid Y=1\right) \pi_1 p_m  + \P\left(h^\star_m(X_{\mathrm{obs}(m)})=1\mid Y=-1\right) \pi_{-1} p_m\\
         &=\sum_{m\in\{0,1\}^d}\Phi( a_m - b_m )\pi_1 p_m  + \Phi\left(- a_m - b_m \right) \pi_{-1} p_m,
\end{align*}
where, for all $m \in \mathcal{M}$, 
\begin{align}
    a_m = \frac{\log\left(\frac{\pi_{-1}}{\pi_1}\right)}{\left\|\Sigma_{\mathrm{obs}(m)}^{-\frac{1}{2}}(\mu_{1, \mathrm{obs}(m)}-\mu_{-1, \mathrm{obs}(m)})\right\| } \quad \textrm{and} \quad b_m = \frac{\left\|\Sigma_{\mathrm{obs}(m)}^{-\frac{1}{2}}(\mu_{1, \mathrm{obs}(m)}-\mu_{-1, \mathrm{obs}(m)})\right\|}{2}.
\end{align}
\end{proof}

\subsection{Proof of \Cref{prop:imputation_LDA_negative}}

\begin{proof}[Proof]
By \Cref{prop:MCARLDA}, we know that the Bayes classifier takes the form
\begin{align}
        h^{\star}_m(x_{\mathrm{obs}(m)}) & = \mathrm{sign}\Big( c_m \big(\mu_{1, \mathrm{obs}(m)}-\mu_{-1, \mathrm{obs}(m)}\big)^{\top}\Sigma_{\mathrm{obs}(m)}^{-1} \Big(x_{\mathrm{obs}(m)}-\frac{\mu_{1, \mathrm{obs}(m)}+\mu_{-1,\mathrm{obs}(m)}}{2}\Big)  \Big) \nonumber \\
        & = \mathrm{sign}\Big( c_m \sum_{j=1}^d  ( \Sigma_{\mathrm{obs}(m)}^{-1} \beta_{\mathrm{obs}(m)} )_j  x_{j} \mathds{1}_{m_j=0} - c_m \sum_{j=1}^d ( \Sigma_{\mathrm{obs}(m)}^{-1} \beta_{\mathrm{obs}(m)} )_j v_j \mathds{1}_{m_j=0}  \Big), \label{eq_proof_imputation_LDA4}
\end{align}
for some $c_m > 0$, where, for all $1 \leq j \leq d$, $\beta_j = \mu_{1, j}-\mu_{-1, j}$ and $v_j = (\mu_{1, j}+\mu_{-1,j})/2$. The predictor associated with constant imputation $\alpha\in \mathbb{R}^d$ takes the form
\begin{align}
    h_{\textrm{imp}}(z) 
    & = \mathrm{sign}\Big( \sum_{j =1}^d \gamma_j x_j \mathds{1}_{m_j=0} + \gamma_0 + \sum_{j =1}^d  \gamma_{j} \alpha_j \mathds{1}_{m_j=0}  \Big). \label{eq_proof_imputation_LDA5}
\end{align}
The Bayes predictor takes the form of an imputation strategy if and only if \eqref{eq_proof_imputation_LDA4} equals \eqref{eq_proof_imputation_LDA5}, that is $\gamma_0=0$ and 
\begin{align}
    \sum_{j=1}^d (c_m ( \Sigma_{\mathrm{obs}(m)}^{-1} \beta_{\mathrm{obs}(m)} )_j - \gamma_j ) x_j \mathds{1}_{m_j=0} & = \sum_{j=1}^d (c_m ( \Sigma_{\mathrm{obs}(m)}^{-1} \beta_{\mathrm{obs}(m)} )_j v_j - \gamma_j \alpha_j) \mathds{1}_{m_j=0}, \label{eq_proof_imputation_LDA8}
\end{align}
which is equivalent to, for all $m \in \{0,1\}^d$ and all $j \in \textrm{obs}(m)$, 
\begin{align}
    \left\lbrace
    \begin{array}{lll}
         &  c_m (\Sigma_{\mathrm{obs}(m)}^{-1} \beta_{\mathrm{obs}(m)} )_j & = \gamma_j\\
         &  c_m ( \Sigma_{\mathrm{obs}(m)}^{-1} \beta_{\mathrm{obs}(m)} )_j v_j & = \gamma_j \alpha_j 
    \end{array}
    \right. 
    \Longleftrightarrow
    \left\lbrace
    \begin{array}{lll}
         &  \gamma_j & = c_m (\Sigma_{\mathrm{obs}(m)}^{-1} \beta_{\mathrm{obs}(m)} )_j \\
         &  ( \Sigma_{\mathrm{obs}(m)}^{-1} \beta_{\mathrm{obs}(m)} )_j v_j & = \alpha_j  (\Sigma_{\mathrm{obs}(m)}^{-1} \beta_{\mathrm{obs}(m)} )_j
    \end{array}
    \right. 
\end{align}
Choosing $m=e_j$, that is the missing pattern in which only the $j$th component is observed, we have
\begin{align}
   \frac{\beta_j}{\sigma_{jj}}  v_j & = \alpha_j  \frac{\beta_j}{\sigma_{jj}}. 
\end{align}
Since this equality must holds for all $\beta$, choosing $\beta\neq 0$ leads to 
\begin{align}
    \alpha_j = v_j.
\end{align}
Besides, for all $m \in \{0,1\}^d$ and all $j \in \textrm{obs}(m)$, 
\begin{align}
   (\Sigma_{\mathrm{obs}(m)}^{-1} \beta_{\mathrm{obs}(m)} )_j = \frac{\gamma_j}{c_m},
\end{align}
since $c_m >0$ by assumption. 
Considering $m = e_j$, the missing pattern in which only the $j$th component is observed, we have
\begin{align}
    \frac{\beta_j}{\sigma_{jj}} = \frac{\gamma_j}{c_m}.
\end{align}
Since $\beta_j \neq 0$, we have $\gamma_j \neq 0$. Now, letting $m=\textbf{0}$ leads to, for all $j \in \{1, \hdots, d\}$, 
\begin{align}
\gamma_j = c_0 (\Sigma^{-1} \beta)_j.
\end{align}
Consequently, 
\begin{align}
\frac{\gamma_j}{\gamma_i} = \frac{(\Sigma^{-1} \beta)_j}{(\Sigma^{-1} \beta)_i} \label{eq_proof_neg_imputation_lda1}
\end{align}
Let $i \neq j \in \{1, \hdots, d\}$. Consider $m = 1 - e_j - e_i$, that is the missing pattern in which only the $j$th or the $i$th components are observed. Then, 
\begin{align}
   \gamma_j = c_m (\Sigma_{\{i,j\}}^{-1} \beta_{\{i,j\}} )_j,
\end{align}
and similarly for $\gamma_i$, where $\Sigma_{\{i,j\}}$ is the submatrix composed of the $i$th and $j$th rows and columns of $\Sigma$. Simple calculations show that 
\begin{align}
    \Sigma_{\{i,j\}}^{-1} = \frac{1}{\sigma_{ii} \sigma_{jj} - \sigma_{ij}^2}\begin{bmatrix}
        \sigma_{jj} & - \sigma_{ij}\\
        - \sigma_{ij} & \sigma_{ii}
    \end{bmatrix}.
\end{align}
Thus, 
\begin{align}
    \Sigma_{\{i,j\}}^{-1} \beta_{\{i,j\}} = \frac{1}{\sigma_{ii} \sigma_{jj} - \sigma_{ij}^2}\begin{bmatrix}
        \sigma_{jj} \beta_i - \sigma_{ij} \beta_j\\
        - \sigma_{ij} \beta_i + \sigma_{ii} \beta_j
    \end{bmatrix}.
\end{align}
Hence, 
\begin{align}
    \frac{\gamma_j}{\gamma_i} = \frac{- \sigma_{ij} \beta_i + \sigma_{ii} \beta_j}{\sigma_{jj} \beta_i - \sigma_{ij} \beta_j} \label{eq_proof_neg_imputation_lda2}
\end{align}
Gathering \eqref{eq_proof_neg_imputation_lda1} and \eqref{eq_proof_neg_imputation_lda2}, 
\begin{align}
    \frac{(\Sigma^{-1} \beta)_j}{(\Sigma^{-1} \beta)_i} & =  \frac{- \sigma_{ij} \beta_i + \sigma_{ii} \beta_j}{\sigma_{jj} \beta_i - \sigma_{ij} \beta_j}. 
\end{align}
Letting $\Sigma^{-1} = (s_{ij})_{1 \leq i,j\leq d}$, we obtain
\begin{align}
    & \left(\sum_{k=1}^d s_{jk}\beta_k \right) \left( \sigma_{jj} \beta_i - \sigma_{ij} \beta_j\right) = \left( \sum_{k=1}^d s_{ik}\beta_k\right) \left(- \sigma_{ij} \beta_i + \sigma_{ii} \beta_j \right)\\
    \Longleftrightarrow & 
    \sum_{k \neq i,j} (s_{jk} \sigma_{jj} + \sigma_{ij} s_{ik}) \beta_k \beta_i 
    + \sum_{k \neq i,j} (- s_{jk} \sigma_{ij} + \sigma_{ii} s_{ik}) \beta_k \beta_j
    + (s_{ji} \sigma_{jj} + s_{ii} \sigma_{ij}) \beta_i^2 \\
    & \quad + (-s_{jj} \sigma_{ij} - s_{ij} \sigma_{ii}) \beta_j^2 + (s_{jj} \sigma_{jj} - s_{ii} \sigma_{ii}) \beta_i \beta_j = 0
\end{align}
As this equality must hold for all $\beta \neq0$, we have
\begin{align}
    \left\lbrace
    \begin{array}{cc}
         &  s_{jk} \sigma_{jj} + \sigma_{ij} s_{ik} = 0 \\
         & - s_{jk} \sigma_{ij} + \sigma_{ii} s_{ik} = 0
    \end{array}
    \right. .
\end{align}
Multiplying the first line by $\sigma_{ij}/\sigma_{jj}$ and adding it to the second one leads to 
\begin{align}
   \sigma_{jj} s_{ik} \left( \sigma_{ij}^2
   + \sigma_{ii} \sigma_{jj} \right) = 0.
\end{align}
Since $\Sigma_{i,j}$ is invertible by assumption, $\sigma_{ij}^2 \neq - \sigma_{ii} \sigma_{jj}$ and $\sigma_{jj} \neq 0$. Thus,  for all $k \neq i,j$, $s_{jk} = s_{ik} = 0$.  As the matrix $\Sigma^{-1}$ is symmetric, we deduce that for $d\geq 3$, for all $k \neq k'$, $s_{kk'}=0$. Consequently, the matrix $\Sigma^{-1}$ is diagonal, and so is $\Sigma$. 
Thus, the matrix $\Sigma$ takes the form $\Sigma = \textrm{diag}(\sigma_1, \hdots, \sigma_d)$. Using \eqref{eq_proof_imputation_LDA8}, and taking $\obs(m)=\{j\}$, we have
\begin{align}
    \sum_{j=1}^d (c_m ( \Sigma_{\mathrm{obs}(m)}^{-1} \beta_{\mathrm{obs}(m)} )_j - \gamma_j ) x_j \mathds{1}_{m_j=0} & = \sum_{j=1}^d (c_m ( \Sigma_{\mathrm{obs}(m)}^{-1} \beta_{\mathrm{obs}(m)} )_j v_j - \gamma_j \alpha_j) \mathds{1}_{m_j=0}, 
\end{align}
\begin{align}
    \sum_{j=1}^d (c_m  \sigma_j^{-1} \beta_{j}  - \gamma_j ) x_j \mathds{1}_{M_j=0} & = \sum_{j=1}^d (c_m \beta_j v_j \sigma_j^{-1} - \gamma_j \alpha_j) \mathds{1}_{M_j=0},
\end{align}
which is equivalent to, for all $j \in \textrm{obs}(m)$, 
\begin{align}
    \left\lbrace
    \begin{array}{lll}
         &  c_m \sigma_j^{-1} \beta_j & = \gamma_j\\
         &  c_m \beta_j \sigma_j^{-1} v_j & = \gamma_j \alpha_j 
    \end{array}
    \right. 
    \Longleftrightarrow
    \left\lbrace
    \begin{array}{lll}
         &  \gamma_j & = c_m \sigma_j^{-1} \beta_j \\
         &  \alpha_j & =  v_j    
    \end{array}
    \right.,
\end{align}
which concludes the proof.

\end{proof}

\subsection{Proof of Proposition \ref{prop:BoundDiffLLcomp}}\label{subsec:proofBoundDiffLLcomp}
\begin{proof}
    Using Assumption \ref{ass:constMeanDiffe}, we have that 
    \begin{align*}
        \left\|\Sigma^{-\frac{1}{2}}(\mu_1-\mu_{-1})\right\|&\leq \frac{\left\|\mu_1-\mu_{-1}\right\|}{\sqrt{\lambda_{\min}(\Sigma)}}=\mu \sqrt{\frac{d}{\lambda_{\min}(\Sigma)}}
        \\\left\|\Sigma^{-\frac{1}{2}}_{\mathrm{obs}(m)}(\mu_{1, \mathrm{obs}(m)}-\mu_{-1, \mathrm{obs}(m)})\right\|&\geq \frac{\left\|\mu_{1, \mathrm{obs}(m)}-\mu_{-1, \mathrm{obs}(m)}\right\|}{\sqrt{\lambda_{\max}(\Sigma)}}=\mu\sqrt{\frac{d-\left\|m\right\|_0}{\lambda_{\max}(\Sigma)}}
    \end{align*} 
    Recall that $\Phi$ is the c.d.f. of a standard Gaussian random variable, according to Equation \eqref{eq:firstDifferRisk}, we have
    \begin{align}
    \notag
        & \mathcal{R}_{\mathrm{mis}}(h^\star)-\mathcal{R}_{\mathrm{comp}}(h^\star_{\mathrm{comp}}) \\&=\sum_{m\in\{0,1\}^d}\left(\Phi\left(-\frac{\left\|\Sigma_{\mathrm{obs}(m)}^{-\frac{1}{2}}(\mu_{1, \mathrm{obs}(m)}-\mu_{-1, \mathrm{obs}(m)})\right\|}{2}\right)-\Phi\left(-\frac{\left\|\Sigma^{-\frac{1}{2}}(\mu_1-\mu_{-1})\right\|}{2}\right)\right)p_m\\ \notag
        &\leq\sum_{m\in\{0,1\}^d}\left(\Phi\left(-\frac{\mu}{2}\sqrt{\frac{d-\left\|m\right\|_0}{\lambda_{\max}(\Sigma)}}\right)-\Phi\left(-\frac{\mu}{2}\sqrt{\frac{d}{\lambda_{\min}(\Sigma)}}\right)\right)p_m\\ \notag
        &=\sum_{i=0}^d\sum\limits_{\substack{m\in\{0,1\}^d\\\textrm{s.t.}~ \left\|m\right\|_0=i}}\left(\Phi\left(-\frac{\mu}{2}\sqrt{\frac{d-i}{\lambda_{\max}(\Sigma)}}\right)-\Phi\left(-\frac{\mu}{2}\sqrt{\frac{d}{\lambda_{\min}(\Sigma)}}\right)\right)p_m\\ \notag
        &=\sum_{i=0}^d\left(\Phi\left(-\frac{\mu}{2}\sqrt{\frac{d-i}{\lambda_{\max}(\Sigma)}}\right)-\Phi\left(-\frac{\mu}{2}\sqrt{\frac{d}{\lambda_{\min}(\Sigma)}}\right)\right)\binom{d}{i}\eta^i(1-\eta)^{d-i} \tag{using Assumption \ref{ass:constEtaMi}}\\
        &=\e{\Phi\left(-\frac{\mu}{2}\sqrt{\frac{d-B}{\lambda_{\max}(\Sigma)}}\right)-\Phi\left(-\frac{\mu}{2}\sqrt{\frac{d}{\lambda_{\min}(\Sigma)}}\right)}\label{eq:decompBin}
    \end{align}
    where $B\sim \mathcal{B}(d,\eta)$.
    The decomposition of this last expression gives us
    \begin{align}
    \notag
        L&(h^\star)-\mathcal{R}_{\mathrm{comp}}(h^\star_{\mathrm{comp}})\\
        \notag
        &\leq\e{\Phi\left(-\frac{\mu}{2}\sqrt{\frac{d-B}{\lambda_{\max}(\Sigma)}}\right)-\Phi\left(-\frac{\mu}{2}\sqrt{\frac{d}{\lambda_{\min}(\Sigma)}}\right)\mid B=d}\P(B=d)\\ \notag
        & \quad +\e{\Phi\left(-\frac{\mu}{2}\sqrt{\frac{d-B}{\lambda_{\max}(\Sigma)}}\right)-\Phi\left(-\frac{\mu}{2}\sqrt{\frac{d}{\lambda_{\min}(\Sigma)}}\right)\mid B\neq d}\P(B\neq d)\\ 
        &=\left(\frac{1}{2}-\Phi\left(-\frac{\mu}{2}\sqrt{\frac{d}{\lambda_{\min}(\Sigma)}}\right)\right)\eta^d\\
        & \quad +\e{\Phi\left(-\frac{\mu}{2}\sqrt{\frac{d-B}{\lambda_{\max}(\Sigma)}}\right)-\Phi\left(\frac{\mu}{2}\sqrt{\frac{d}{\lambda_{\min}(\Sigma)}}\right)\mid B\neq d}(1-\eta^d)\label{eq:firstDecompExpLcompVSLmiss}
    \end{align}
    Now, we study the second term in \eqref{eq:firstDecompExpLcompVSLmiss}. Letting $Q(x)=\int^\infty_{x}e^{-\frac{t^2}{2}}dt$, we have $        \Phi(x)=\frac{1}{\sqrt{2\pi}}Q(-x)$, which leads to 
    \begin{align*}
        \mathbb{E}&\left[\Phi\left(-\frac{\mu}{2}\sqrt{\frac{d-B}{\lambda_{\max}(\Sigma)}}\right)-\Phi\left(-\frac{\mu}{2}\sqrt{\frac{d}{\lambda_{\min}(\Sigma)}}\right)\mid B\neq d\right]\\   \notag &=\e{\frac{1}{\sqrt{2\pi}}\left(Q\left(T_B\right)-Q\left(t\right)\right)\mid B\neq d}, 
    \end{align*}
    where $T_B:=\frac{\mu}{2}\sqrt{\frac{d-B}{\lambda_{\max}(\Sigma)}}$ and $t=\frac{\mu}{2}\sqrt{\frac{d}{\lambda_{\min}(\Sigma)}}$. Applying the the mean-value inequality to the function $Q$ on the interval $[T_B, t]$ leads to 
    \begin{align}
        Q(T_B)-Q(t) \leq e^{-\frac{T_B^2}{2}}(t-T_B).\label{eq:LDATVI}
    \end{align}
    Thus, 
        \begin{align}
        \notag
            \mathbb{E}&\left[\Phi\left(-\frac{\mu}{2}\sqrt{\frac{d-B}{\lambda_{\max}(\Sigma)}}\right)-\Phi\left(-\frac{\mu}{2}\sqrt{\frac{d}{\lambda_{\min}(\Sigma)}}\right)\mid B\neq d\right]\\
            &\leq \frac{1}{\sqrt{2\pi}}\e{e^{-\frac{t_B^2}{2}}(t-T_B)\mid B\neq d} \nonumber \\
            &=\frac{\mu}{2\sqrt{2\pi}}\e{e^{-\frac{\mu^2(d-B)}{8\lambda_{\max}(\Sigma)}}\left(\sqrt{\frac{d}{\lambda_{\min}(\Sigma)}}-\sqrt{\frac{d-B}{\lambda_{\max}(\Sigma)}}\right)\mid B\neq d}.
            \label{eq:expectIneqLcompVSLmiss}
        \end{align}
Besides, since
\begin{align*}            \mathbb{E}&\left[e^{-\frac{\mu^2(d-B)}{8\lambda_{\max}(\Sigma)}}\left(\sqrt{\frac{d}{\lambda_{\min}(\Sigma)}}-\sqrt{\frac{d-B}{\lambda_{\max}(\Sigma)}}\right)\right]\\&=\e{e^{-\frac{\mu^2(d-B)}{8\lambda_{\max}(\Sigma)}}\left(\sqrt{\frac{d}{\lambda_{\min}(\Sigma)}}-\sqrt{\frac{d-B}{\lambda_{\max}(\Sigma)}}\right)\mid B\neq d}\P(B\neq d)+\sqrt{\frac{d}{\lambda_{\min}(\Sigma)}}\P(B=d),
        \end{align*}
we have
\begin{align} 
\mathbb{E}&\left[\Phi\left(-\frac{\mu}{2}\sqrt{\frac{d-B}{\lambda_{\max}(\Sigma)}}\right)-\Phi\left(-\frac{\mu}{2}\sqrt{\frac{d}{\lambda_{\min}(\Sigma)}}\right)\mid B\neq d\right] \notag\\
           &=\frac{\mu}{2\sqrt{2\pi}}\frac{1}{\P(B\neq d)}\e{e^{-\frac{\mu^2(d-B)}{8\lambda_{\max}(\Sigma)}}\left(\sqrt{\frac{d}{\lambda_{\min}(\Sigma)}}-\sqrt{\frac{d-B}{\lambda_{\max}(\Sigma)}}\right)}\\
           & \qquad - \frac{\mu}{2\sqrt{2\pi}} \frac{\P(B=d)}{\P(B\neq d)} \sqrt{\frac{d}{\lambda_{\min}(\Sigma)}}.\label{eq:espCondLDA}
        \end{align}
        Looking at the expectation in \eqref{eq:espCondLDA}, we obtain 
        \begin{align*}
        \mathbb{E}&\left[e^{-\frac{\mu^2(d-B)}{8\lambda_{\max}(\Sigma)}}\left(\sqrt{\frac{d}{\lambda_{\min}(\Sigma)}}-\sqrt{\frac{d-B}{\lambda_{\max}(\Sigma)}}\right)\right]\\
            &=\e{e^{-\frac{\mu^2(d-B)}{8\lambda_{\max}(\Sigma)}}\left(\sqrt{\frac{d}{\lambda_{\min}(\Sigma)}}-\sqrt{\frac{d}{\lambda_{\max}(\Sigma)}}+\sqrt{\frac{d}{\lambda_{\max}(\Sigma)}}-\sqrt{\frac{d-B}{\lambda_{\max}(\Sigma)}}\right)}\\
            &=\sqrt{d}\left(\frac{1}{\sqrt{\lambda_{\min}(\Sigma)}}-\frac{1}{\sqrt{\lambda_{\max}(\Sigma)}}\right)\e{e^{-\frac{\mu^2(d-B)}{8\lambda_{\max}(\Sigma)}}}\\
            & \qquad +\frac{1}{\sqrt{\lambda_{\max}(\Sigma)}}\e{e^{-\frac{\mu^2(d-B)}{8\lambda_{\max}(\Sigma)}}\left(\sqrt{d}-\sqrt{d-B}\right)}\\
            &\leq\sqrt{d}\left(\frac{1}{\sqrt{\lambda_{\min}(\Sigma)}}-\frac{1}{\sqrt{\lambda_{\max}(\Sigma)}}\right)\e{e^{-\frac{\mu^2(d-B)}{8\lambda_{\max}(\Sigma)}}}+\frac{1}{\sqrt{\lambda_{\max}(\Sigma)d}}\e{e^{-\frac{\mu^2(d-B)}{8\lambda_{\max}(\Sigma)}}B},
        \end{align*}
        since 
        \begin{align}
           \sqrt{d}-\sqrt{d-B} =  \frac{d-d+B}{\sqrt{d}+\sqrt{d-B}} \leq \frac{B}{\sqrt{d}}.
        \end{align}
        Simple calculation shows that 
        \begin{align*}
            \e{e^{-\frac{\mu^2(d-B)}{8\lambda_{\max}(\Sigma)}}}
            &=\left(\eta+e^{-\frac{\mu^2}{8\lambda_{\max}(\Sigma)}}(1-\eta)\right)^d.
        \end{align*}
    Besides,  
    \begin{align}
    \notag
        \e{e^{-\frac{\mu^2(d-B)}{8\lambda_{\max}(\Sigma)}}B}&=\sum_{i=0}^d\binom{d}{i}e^{-\frac{\mu^2(d-i)}{8\lambda_{\max}(\Sigma)}}i\eta^i(1-\eta)^{d-i}\\ 
        \notag &= \eta d\sum_{i=1}^d\frac{(d-1)!}{(i-1)!(d-1-(i-1))!}\eta^{i-1}\left(e^{-\frac{\mu^2}{8\lambda_{\max}(\Sigma)}}(1-\eta)\right)^{d-1-(i-1)} \\
        &=\eta d \left(\eta+e^{-\frac{\mu^2}{8\lambda_{\max}(\Sigma)}}(1-\eta)\right)^{d-1}.\label{eq:expectExpBinLcompvsLmis}
    \end{align}
    Therefore, letting $A = e^{-\frac{\mu^2}{8\lambda_{\max}(\Sigma)}}(1-\eta)$, we have that
    \begin{align}
    \notag
        \mathbb{E}&\left[e^{-\frac{\mu^2(d-B)}{8\lambda_{\max}(\Sigma)}}\left(\sqrt{\frac{d}{\lambda_{\min}(\Sigma)}}-\sqrt{\frac{d-B}{\lambda_{\max}(\Sigma)}}\right)\right]\\ \notag
        &\leq\sqrt{d}\left(\frac{1}{\sqrt{\lambda_{\min}(\Sigma)}}-\frac{1}{\sqrt{\lambda_{\max}(\Sigma)}}\right)\left(\eta+ A \right)^d +\frac{1}{\sqrt{\lambda_{\max}(\Sigma)d}}\eta d \left(\eta+ A \right)^{d-1} \\ 
        &=\frac{\sqrt{d}}{\sqrt{\lambda_{\min}(\Sigma)}}\left(\eta+ A \right)^d -\sqrt{\frac{d}{\lambda_{\max}(\Sigma)}}\left(\eta+A\right)^{d-1}A .\label{eq:decompLcompvsLEsp}
    \end{align}
Gathering equations \eqref{eq:firstDecompExpLcompVSLmiss}, \eqref{eq:espCondLDA} and \eqref{eq:decompLcompvsLEsp}, we obtain 
\begin{align*}
    L&(h^\star)-\mathcal{R}_{\mathrm{comp}}(h^\star_{\mathrm{comp}})\\
        &=\left(\frac{1}{2}-\Phi\left(-\frac{\mu}{2}\sqrt{\frac{d}{\lambda_{\min}(\Sigma)}}\right)\right)\eta^d \\
        & \quad +\e{\Phi\left(-\frac{\mu}{2}\sqrt{\frac{d-B}{\lambda_{\max}(\Sigma)}}\right)-\Phi\left(\frac{\mu}{2}\sqrt{\frac{d}{\lambda_{\min}(\Sigma)}}\right)\mid B\neq d}(1-\eta^d)\\
        & \leq \left(\frac{1}{2}-\Phi\left(-\frac{\mu}{2}\sqrt{\frac{d}{\lambda_{\min}(\Sigma)}}\right)\right)\eta^d + \frac{\mu}{2\sqrt{2\pi}} (1-\eta^d)  \\
        &\qquad \times\left(\frac{1}{\P(B\neq d)}\e{e^{-\frac{\mu^2(d-B)}{8\lambda_{\max}(\Sigma)}}\left(\sqrt{\frac{d}{\lambda_{\min}(\Sigma)}}-\sqrt{\frac{d-B}{\lambda_{\max}(\Sigma)}}\right)}-\sqrt{\frac{d}{\lambda_{\min}(\Sigma)}}\frac{\P(B=d)}{\P(B\neq d)}\right) \\
        &=\left(\frac{1}{2}-\Phi\left(-\frac{\mu}{2}\sqrt{\frac{d}{\lambda_{\min}(\Sigma)}}\right)\right)\eta^d\\
        &\qquad+\frac{\mu}{2\sqrt{2\pi}}\left(\e{e^{-\frac{\mu^2(d-B)}{8\lambda_{\max}(\Sigma)}}\left(\sqrt{\frac{d}{\lambda_{\min}(\Sigma)}}-\sqrt{\frac{d-B}{\lambda_{\max}(\Sigma)}}\right)}-\sqrt{\frac{d}{\lambda_{\min}(\Sigma)}}\eta^d\right)\\
        &\leq \left(\frac{1}{2}-\Phi\left(-\frac{\mu}{2}\sqrt{\frac{d}{\lambda_{\min}(\Sigma)}}\right)\right)\eta^d\\&\qquad +\frac{\mu}{2\sqrt{2\pi}}\left(\sqrt{\frac{d}{\lambda_{\min}(\Sigma)}}\left(\eta+A\right)^d-\sqrt{\frac{d}{\lambda_{\max}(\Sigma)}}\left(\eta+A\right)^{d-1}A-\sqrt{\frac{d}{\lambda_{\min}(\Sigma)}}\eta^d\right).
\end{align*}
An upper bound of this inequality is given by 
\begin{align}
L&(h^\star)-\mathcal{R}_{\mathrm{comp}}(h^\star_{\mathrm{comp}})  \leq \frac{\eta^d}{2} +\frac{\mu \eta }{2\sqrt{2\pi}} \sqrt{\frac{d}{\lambda_{\min}(\Sigma)}}\left( \left(\eta+A\right)^{d-1} - \eta^{d-1} \right).
\end{align}
\end{proof}

\subsection{Proof of Corollary \ref{cor:limitLLcomp}}\label{subsec:limitLLcomp}
\begin{proof}
    Recall that, by Equation \eqref{eq:firstDifferRisk}, 
    \begin{align*}
        L&(h^\star)-\mathcal{R}_{\mathrm{comp}}(h^\star_{\mathrm{comp}})\\
    &=\sum_{m\in\{0,1\}^d}\left(\Phi\left(-\frac{\left\|\Sigma^{-\frac{1}{2}}_{\mathrm{obs}(m)}(\mu_{1, \mathrm{obs}(m)}-\mu_{-1, \mathrm{obs}(m)})\right\|}{2}\right)-\Phi\left(-\frac{\left\|\Sigma^{-\frac{1}{2}}(\mu_1-\mu_{-1})\right\|}{2}\right)\right)p_m \\
    &\geq \left(\Phi\left(0\right)-\Phi\left(-\frac{\left\|\Sigma^{-\frac{1}{2}}(\mu_1-\mu_{-1})\right\|}{2}\right)\right)\eta^d,
    \end{align*}
    using only $m=$\textbf{1}, since all terms in the above sum are positive. By Assumption \ref{ass:constMeanDiffe}, $\|\Sigma^{-\frac{1}{2}}(\mu_1-\mu_{-1})\|\geq  d\mu / \sqrt{\lambda_{\max}(\Sigma)}$. Hence
    \begin{align*}
        \left(\Phi\left(0\right)-\Phi\left(-\frac{\left\|\Sigma^{-\frac{1}{2}}(\mu_1-\mu_{-1})\right\|}{2}\right)\right)\eta^d&\geq \left(\Phi\left(0\right)-\Phi\left(-\frac{d\mu}{2\sqrt{\lambda_{\max}(\Sigma)}}\right)\right)\eta^d\\&=
        \left(\frac{1}{2}-\Phi\left(-\frac{d\lambda}{2}\right)\right)\eta^d.
    \end{align*}
    Consequently, 
    \begin{align}
        L&(h^\star)-\mathcal{R}_{\mathrm{comp}}(h^\star_{\mathrm{comp}}) \geq  \left(\frac{1}{2}-\Phi\left(-\frac{d\lambda}{2}\right)\right)\eta^d\xrightarrow[\lambda\to\infty]{}\frac{\eta^d}{2}.
    \end{align}
On the other hand, by \cref{prop:BoundDiffLLcomp}, we have
\begin{align}
L&(h^\star)-\mathcal{R}_{\mathrm{comp}}(h^\star_{\mathrm{comp}})  \leq \frac{\eta^d}{2} +\frac{\mu}{2\sqrt{2\pi}} \sqrt{\frac{d}{\lambda_{\min}(\Sigma)}}\left( \left(\eta+A\right)^d - \eta^d \right).
\end{align}
Note that
    \begin{align*}
        & \left|\mu\sqrt{\frac{d}{\lambda_{\min}(\Sigma)}}\left(\left(\eta+e^{-\frac{\mu^2}{8\lambda_{\max}(\Sigma)}}(1-\eta)\right)^d-\eta^d\right)\right|\\
        &= \mu\sqrt{\frac{d}{\lambda_{\min}(\Sigma)}}\left(\sum_{i=0}^d\binom{d}{i}\eta^{d-i}e^{-\frac{i\lambda^2}{8}}(1-\eta)^i-\eta^d\right)\\&=\frac{\mu}{\sqrt{\lambda_{\max}(\Sigma)}}\sqrt{\frac{d\lambda_{\max}(\Sigma)}{\lambda_{\min}(\Sigma)}}\left(\sum_{i=1}^d\binom{d}{i}\eta^{d-i}e^{-\frac{i\lambda^2}{8}}(1-\eta)^i\right)\\&=\lambda\sqrt{\frac{d\lambda_{\max}(\Sigma)}{\lambda_{\min}(\Sigma)}}\left(\sum_{i=1}^d\binom{d}{i}\eta^{d-i}e^{-\frac{i\lambda^2}{8}}(1-\eta)^i\right),
    \end{align*}
    which tends to zero by assumption. This concludes the proof.  
\end{proof}

\subsection{Proofs of Section \ref{subsubsec:MuEstimLDA}}\label{sect:proofsMuEstimLDA}

\subsubsection{General lemmas for LDA misclassification control. }\label{subsec:generalLemmasBoundEstim}
\begin{lemma}[$\widehat{\mu}$ misclassification probability]\label{lemma:probMisMuEstim}
        Given a sample satisfying Assumptions \ref{ass:MCAR} and \ref{ass:LDA}, with balanced classes, then 
    \begin{align}
    \notag
        \P &\left(\widehat{h}_m(X_{\mathrm{obs}(m)})=1\mid Y=-1, \mathcal{D}_n\right) \\
        &= \Phi\left(\frac{\left(\Sigma_{\mathrm{obs}(m)}^{-\frac{1}{2}}(\widehat{\mu}_{1,\mathrm{obs}(m)}-\widehat{\mu}_{-1,\mathrm{obs}(m)})\right)^\top\Sigma_{\mathrm{obs}(m)}^{-\frac{1}{2}}\left(\mu_{-1,\mathrm{obs}(m)}-\frac{\widehat{\mu}_{1,\mathrm{obs}(m)}+\widehat{\mu}_{-1, \mathrm{obs}(m)}}{2}\right)}{\left\|\Sigma_{\mathrm{obs}(m)}^{-\frac{1}{2}}(\widehat{\mu}_{1,\mathrm{obs}(m)}-\widehat{\mu}_{-1,\mathrm{obs}(m)})\right\|}\right) \label{eq:probMissLDAEstim}
    \end{align}
    and symmetrically, 
    \begin{align}
        \notag 
        \P&\left(\widehat{h}_m(X_{\mathrm{obs}(m)})=-1\mid Y=1, \mathcal{D}_n\right) \\
        &= \Phi\left(-\frac{\left(\Sigma_{\mathrm{obs}(m)}^{-\frac{1}{2}}(\widehat{\mu}_{1,\mathrm{obs}(m)}-\widehat{\mu}_{-1,\mathrm{obs}(m)})\right)^\top\Sigma_{\mathrm{obs}(m)}^{-\frac{1}{2}}\left(\mu_{1,\mathrm{obs}(m)}-\frac{\widehat{\mu}_{1,\mathrm{obs}(m)}+\widehat{\mu}_{-1, \mathrm{obs}(m)}}{2}\right)}{\left\|\Sigma_{\mathrm{obs}(m)}^{-\frac{1}{2}}(\widehat{\mu}_{1,\mathrm{obs}(m)}-\widehat{\mu}_{-1,\mathrm{obs}(m)})\right\|}\right)
        \label{eq:probMissLDAEstim2}
    \end{align}
    with $\Phi$ the standard Gaussian cumulative function.
\end{lemma}
\begin{proof}
We follow the same strategy as in the proof of \cref{cor:GenBayesRiskpbpLDA}. We have
   \begin{align*}
   & \P\left(\widehat{h}_m(X_{\mathrm{obs}(m)})=1\mid Y=-1, \mathcal{D}_n\right) \\
    &= \P\left( \left(\widehat{\mu}_{1, \mathrm{obs}(m)}-\widehat{\mu}_{-1, \mathrm{obs}(m)}\right)^{\top}\Sigma_{\mathrm{obs}(m)}^{-1}\left(X_{\mathrm{obs}(m)}-\frac{\widehat{\mu}_{1, \mathrm{obs}(m)}+\widehat{\mu}_{-1,\mathrm{obs}(m)}}{2}\right)>0\mid  Y=-1, \mathcal{D}_n\right)
    \end{align*}
   Let $N=\Sigma_{\mathrm{obs}(m)}^{-\frac{1}{2}}(X_{\mathrm{obs}(m)}-\mu_{-1, \mathrm{obs}(m)})$. By \Cref{lemma:obsGaus},  $N|Y=-1\sim\mathcal{N}(0, Id_{d-\left\|m\right\|_0})$. Since $(X_{\mathrm{obs}(m)},Y)$ and  $\mathcal{D}_n$ are independent
   \begin{align}
    N|Y=-1, \mathcal{D}_n \sim\mathcal{N}(0, Id_{d-\left\|m\right\|_0}).
   \end{align}
   Letting $\widehat{\gamma}=\Sigma_{\mathrm{obs}(m)}^{-\frac{1}{2}}(\widehat{\mu}_{1,\mathrm{obs}(m)}-\widehat{\mu}_{-1,\mathrm{obs}(m)})$, we have
   \begin{align*}
       & \P\left(h_m^{\star}(X_{\mathrm{obs}(m)})=1\mid Y=-1, \mathcal{D}_n\right)\\
       &=\P\left(\widehat{\gamma}^{\top}N+\widehat{\gamma}^\top\Sigma_{\mathrm{obs}(m)}^{-\frac{1}{2}}\left(\mu_{-1,\mathrm{obs}(m)}-\frac{\widehat{\mu}_{1, \mathrm{obs}(m)}+\widehat{\mu}_{-1,\mathrm{obs}(m)}}{2}\right)>0\mid Y=-1, \mathcal{D}_n\right)\\
       &=\P\left(\frac{\widehat{\gamma}^{\top}N}{\left\|\widehat{\gamma}\right\|}>-\frac{\widehat{\gamma}^\top}{\left\|\widehat{\gamma}\right\|}\Sigma_{\mathrm{obs}(m)}^{-\frac{1}{2}}\left(\mu_{-1,\mathrm{obs}(m)}-\frac{\widehat{\mu}_{1, \mathrm{obs}(m)}+\widehat{\mu}_{-1,\mathrm{obs}(m)}}{2}\right)\mid Y=-1, \mathcal{D}_n\right)\\
       &=\Phi\left(\frac{\widehat{\gamma}^\top}{\left\|\widehat{\gamma}\right\|}\Sigma_{\mathrm{obs}(m)}^{-\frac{1}{2}}\left(\mu_{-1,\mathrm{obs}(m)}-\frac{\widehat{\mu}_{1, \mathrm{obs}(m)}+\widehat{\mu}_{-1,\mathrm{obs}(m)}}{2}\right)\right).
   \end{align*} 

   Now we prove the second statement. 
   According to the proof of \Cref{cor:GenBayesRiskpbpLDA}, we have
   \begin{align*}
        & \P\left(\widehat{h}_m(X_{\mathrm{obs}(m)})=-1\mid Y=1, \mathcal{D}_n\right) \\
         & = \P\left( \left(\widehat{\mu}_{1, \mathrm{obs}(m)}-\widehat{\mu}_{-1, \mathrm{obs}(m)}\right)^{\top}\Sigma_{\mathrm{obs}(m)}^{-1}\left(X_{\mathrm{obs}(m)}-\frac{\widehat{\mu}_{1, \mathrm{obs}(m)}+\widehat{\mu}_{-1,\mathrm{obs}(m)}}{2}\right)<0\mid  Y=1\right).
    \end{align*}
   Let $N=\Sigma_{\mathrm{obs}(m)}^{-\frac{1}{2}}(X_{\mathrm{obs}(m)}-\mu_{1, \mathrm{obs}(m)})$. By \Cref{lemma:obsGaus}, and since $(X_{\mathrm{obs}(m)},Y)$ and  $\mathcal{D}_n$ are independent, 
   \begin{align}
   N|Y=1, \mathcal{D}_n \sim\mathcal{N}(0, Id_{d-\left\|m\right\|_0}).    
   \end{align}
   Letting $\widehat{\gamma}=\Sigma_{\mathrm{obs}(m)}^{-\frac{1}{2}}(\widehat{\mu}_{1,\mathrm{obs}(m)}-\widehat{\mu}_{-1,\mathrm{obs}(m)})$, we have
   \begin{align*}
       & \P\left(\widehat{h}_m(X_{\mathrm{obs}(m)})=-1\mid Y=1, \mathcal{D}_n\right)\\
       &=\P\left(\widehat{\gamma}^{\top}N+\widehat{\gamma}^\top\Sigma_{\mathrm{obs}(m)}^{-\frac{1}{2}}\left(\mu_{1,\mathrm{obs}(m)}-\frac{\widehat{\mu}_{1, \mathrm{obs}(m)}+\widehat{\mu}_{-1,\mathrm{obs}(m)}}{2}\right)<0\mid Y=1, \mathcal{D}_n\right)\\
       &=\P\left(\frac{\widehat{\gamma}^{\top}N}{\left\|\widehat{\gamma}\right\|}<-\frac{\widehat{\gamma}^\top}{\left\|\widehat{\gamma}\right\|}\Sigma_{\mathrm{obs}(m)}^{-\frac{1}{2}}\left(\mu_{1,\mathrm{obs}(m)}-\frac{\widehat{\mu}_{1, \mathrm{obs}(m)}+\widehat{\mu}_{-1,\mathrm{obs}(m)}}{2}\right)\mid Y=1, \mathcal{D}_n\right)\\
       &=\Phi\left(-\frac{\widehat{\gamma}^\top}{\left\|\widehat{\gamma}\right\|}\Sigma_{\mathrm{obs}(m)}^{-\frac{1}{2}}\left(\mu_{1,\mathrm{obs}(m)}-\frac{\widehat{\mu}_{1, \mathrm{obs}(m)}+\widehat{\mu}_{-1,\mathrm{obs}(m)}}{2}\right)\right).
   \end{align*} 
\end{proof}

\begin{lemma}\label{lemma:boundMissProbMuEstim}
Grant Assumptions \ref{ass:MCAR} and \ref{ass:LDA}. Assume that we are given two estimators $\widehat{\mu}_1$ and $\widehat{\mu}_{-1}$ of $\mu_1$ and $\mu_{-1}$. Then, the classifier $\widehat{h}_m$ defined in Equation \eqref{eq:PbPLDAEstimMuAllPatterns} satisfies 
    \begin{align}
    \notag
        \bigl| 
    \P&\left(\widehat{h}_m(X_{\mathrm{obs}(m)})=1\mid  Y=-1, \mathcal{D}_n\right)-\P\left(h_m^\star(X_{\mathrm{obs}(m)})=1 \mid Y=-1\right) 
    \bigr|\\
    &\leq \frac{3}{2\sqrt{2\pi}}\left\|\Sigma_{\mathrm{obs}(m)}^{-\frac{1}{2}}(\mu_{-1,\mathrm{obs}(m)}-\widehat{\mu}_{-1,\mathrm{obs}(m)})\right\|+
    \frac{1}{2\sqrt{2\pi}}\left\|\Sigma_{\mathrm{obs}(m)}^{-\frac{1}{2}}(\mu_{1, \mathrm{obs}(m)}-\widehat{\mu}_{1,\mathrm{obs}(m)})\right\|\label{eq:boundMissProbMuEstim1}
    \end{align}
    and symmetrically, 
    \begin{align}\notag
        \bigl| 
    \P&\left(\widehat{h}_m(X_{\mathrm{obs}(m)})=-1\mid  Y=1, \mathcal{D}_n\right)-\P\left(h_m^\star(X_{\mathrm{obs}(m)})=-1 \mid Y=1\right) 
    \bigr|\\
    &\leq\frac{3}{2\sqrt{2\pi}}\left\|\Sigma_{\mathrm{obs}(m)}^{-\frac{1}{2}}(-\widehat{\mu}_{1,\mathrm{obs}(m)}
    +\mu_{1, \mathrm{obs}(m)})\right\|+\frac{1}{2\sqrt{2\pi}}\left\|\Sigma_{\mathrm{obs}(m)}^{-\frac{1}{2}}(\widehat{\mu}_{-1,\mathrm{obs}(m)}-\mu_{-1, \mathrm{obs}(m)})\right\|\label{eq:boundMissProbMuEstim2}
    \end{align}
\end{lemma}
\begin{proof}
We only prove the first inequality, the other one can be handled in the same manner. According to using \cref{cor:GenBayesRiskpbpLDA} and ~\cref{lemma:probMisMuEstim}, 

\begin{align*}
    \bigl| 
    \P&\left(\widehat{h}_m(X_{\mathrm{obs}(m)})=1\mid  Y=-1, \mathcal{D}_n\right)-\P\left(h_m^\star(X_{\mathrm{obs}(m)})=1 \mid Y=-1\right) 
    \bigr|\\
    &=\left|\Phi\left(\frac{\left(\Sigma_{\mathrm{obs}(m)}^{-\frac{1}{2}}(\widehat{\mu}_{1,\mathrm{obs}(m)}-\widehat{\mu}_{-1,\mathrm{obs}(m)})\right)^\top\Sigma_{\mathrm{obs}(m)}^{-\frac{1}{2}}\left(\mu_{-1,\mathrm{obs}(m)}-\frac{\widehat{\mu}_{1,\mathrm{obs}(m)}+\widehat{\mu}_{-1, \mathrm{obs}(m)}}{2}\right)}{\left\|\Sigma_{\mathrm{obs}(m)}^{-\frac{1}{2}}(\widehat{\mu}_{1,\mathrm{obs}(m)}-\widehat{\mu}_{-1,\mathrm{obs}(m)})\right\|}\right)\right.\\&\qquad\qquad\left.-\Phi\left(-\frac{\left\|\Sigma_{\mathrm{obs}(m)}^{-\frac{1}{2}}(\mu_{1, \mathrm{obs}(m)}-\mu_{-1, \mathrm{obs}(m)})\right\|}{2}\right)\right|\\
    &\leq \frac{1}{\sqrt{2\pi}}\left|\frac{\left(\Sigma_{\mathrm{obs}(m)}^{-\frac{1}{2}}(\widehat{\mu}_{1,\mathrm{obs}(m)}-\widehat{\mu}_{-1,\mathrm{obs}(m)})\right)^\top\Sigma_{\mathrm{obs}(m)}^{-\frac{1}{2}}\left(\mu_{-1,\mathrm{obs}(m)}-\frac{\widehat{\mu}_{1,\mathrm{obs}(m)}+\widehat{\mu}_{-1, \mathrm{obs}(m)}}{2}\right)}{\left\|\Sigma_{\mathrm{obs}(m)}^{-\frac{1}{2}}(\widehat{\mu}_{1,\mathrm{obs}(m)}-\widehat{\mu}_{-1,\mathrm{obs}(m)})\right\|}\right.\\&\qquad\qquad\left.+\frac{\left\|\Sigma_{\mathrm{obs}(m)}^{-\frac{1}{2}}(\mu_{1, \mathrm{obs}(m)}-\mu_{-1, \mathrm{obs}(m)})\right\|}{2}\right|,
    \end{align*}
    since $\Phi$ is $(1/\sqrt{2\pi})$-Lipschitz. Note that, by injecting $\pm \widehat{\mu}_{-1,\mathrm{obs}(m)}$, the numerator of the first term can be rewritten as 
    \begin{align}
        & \left(\Sigma_{\mathrm{obs}(m)}^{-\frac{1}{2}}(\widehat{\mu}_{1,\mathrm{obs}(m)}-\widehat{\mu}_{-1,\mathrm{obs}(m)})\right)^\top\Sigma_{\mathrm{obs}(m)}^{-\frac{1}{2}}\left(\mu_{-1,\mathrm{obs}(m)}-\frac{\widehat{\mu}_{1,\mathrm{obs}(m)}+\widehat{\mu}_{-1, \mathrm{obs}(m)}}{2}\right)\\
        & = \left(\Sigma_{\mathrm{obs}(m)}^{-\frac{1}{2}}(\widehat{\mu}_{1,\mathrm{obs}(m)}-\widehat{\mu}_{-1,\mathrm{obs}(m)})\right)^\top\Sigma_{\mathrm{obs}(m)}^{-\frac{1}{2}}\left(\mu_{-1,\mathrm{obs}(m)}-\widehat{\mu}_{-1,\mathrm{obs}(m)}\right) \\
        & \qquad + \frac{1}{2}\left(\Sigma_{\mathrm{obs}(m)}^{-\frac{1}{2}}(\widehat{\mu}_{1,\mathrm{obs}(m)}-\widehat{\mu}_{-1,\mathrm{obs}(m)})\right)^\top\Sigma_{\mathrm{obs}(m)}^{-\frac{1}{2}}\left(\widehat{\mu}_{-1,\mathrm{obs}(m)}-\widehat{\mu}_{1,\mathrm{obs}(m)}\right)\\
        & \leq \left\|\Sigma_{\mathrm{obs}(m)}^{-\frac{1}{2}}(\widehat{\mu}_{1,\mathrm{obs}(m)}-\widehat{\mu}_{-1,\mathrm{obs}(m)})\right\|\left\|\Sigma_{\mathrm{obs}(m)}^{-\frac{1}{2}}(\mu_{-1,\mathrm{obs}(m)}-\widehat{\mu}_{-1,\mathrm{obs}(m)})\right\|  \\
        & \qquad - \frac{1}{2} \left\|\Sigma_{\mathrm{obs}(m)}^{-\frac{1}{2}}(\widehat{\mu}_{1,\mathrm{obs}(m)}-\widehat{\mu}_{-1,\mathrm{obs}(m)})\right\|^2,
    \end{align}
    where the last line results from Cauchy-Schwarz inequality. Thus, by the Triangle inequality, followed by the reverse triangle inequality, we obtain
\begin{align}
    \bigl| 
    \P&\left(\widehat{h}_m(X_{\mathrm{obs}(m)})=1\mid  Y=-1, \mathcal{D}_n\right)-\P\left(h_m^\star(X_{\mathrm{obs}(m)})=1 \mid Y=-1\right) 
    \bigr|\\
    &\leq \frac{1}{\sqrt{2\pi}}\frac{\left\|\Sigma_{\mathrm{obs}(m)}^{-\frac{1}{2}}(\widehat{\mu}_{1,\mathrm{obs}(m)}-\widehat{\mu}_{-1,\mathrm{obs}(m)})\right\|\left\|\Sigma_{\mathrm{obs}(m)}^{-\frac{1}{2}}(\mu_{-1,\mathrm{obs}(m)}-\widehat{\mu}_{-1,\mathrm{obs}(m)})\right\|}{\left\|\Sigma_{\mathrm{obs}(m)}^{-\frac{1}{2}}(\widehat{\mu}_{1,\mathrm{obs}(m)}-\widehat{\mu}_{-1,\mathrm{obs}(m)})\right\|} \\
    &\qquad+
    \frac{1}{\sqrt{2\pi}}\left|-\frac{\left\|\Sigma_{\mathrm{obs}(m)}^{-\frac{1}{2}}(\widehat{\mu}_{1,\mathrm{obs}(m)}-\widehat{\mu}_{-1,\mathrm{obs}(m)})\right\|}{2}
    +\frac{\left\|\Sigma_{\mathrm{obs}(m)}^{-\frac{1}{2}}(\mu_{1, \mathrm{obs}(m)}-\mu_{-1, \mathrm{obs}(m)})\right\|}{2}\right|  \\
    & \leq  \frac{1}{\sqrt{2\pi}}\left\|\Sigma_{\mathrm{obs}(m)}^{-\frac{1}{2}}(\mu_{-1,\mathrm{obs}(m)}-\widehat{\mu}_{-1,\mathrm{obs}(m)})\right\| \\
    & \qquad +
    \frac{1}{2\sqrt{2\pi}}\left\|\Sigma_{\mathrm{obs}(m)}^{-\frac{1}{2}}(-\widehat{\mu}_{1,\mathrm{obs}(m)}+\widehat{\mu}_{-1,\mathrm{obs}(m)}+\mu_{1, \mathrm{obs}(m)}-\mu_{-1, \mathrm{obs}(m)})\right\|
    \\&\leq \frac{1}{\sqrt{2\pi}} \left\|\Sigma_{\mathrm{obs}(m)}^{-\frac{1}{2}}(\mu_{-1,\mathrm{obs}(m)}-\widehat{\mu}_{-1,\mathrm{obs}(m)})\right\|+
    \frac{1}{2\sqrt{2\pi}}\left\|\Sigma_{\mathrm{obs}(m)}^{-\frac{1}{2}}(\mu_{1, \mathrm{obs}(m)}-\widehat{\mu}_{1,\mathrm{obs}(m)})\right\|\\&\qquad+\frac{1}{2\sqrt{2\pi}}\left\|\Sigma_{\mathrm{obs}(m)}^{-\frac{1}{2}}(\widehat{\mu}_{-1,\mathrm{obs}(m)}-\mu_{-1, \mathrm{obs}(m)})\right\|\\
    & \leq \frac{3}{2\sqrt{2\pi}} \left\|\Sigma_{\mathrm{obs}(m)}^{-\frac{1}{2}}(\mu_{-1,\mathrm{obs}(m)}-\widehat{\mu}_{-1,\mathrm{obs}(m)})\right\|+
    \frac{1}{2\sqrt{2\pi}}\left\|\Sigma_{\mathrm{obs}(m)}^{-\frac{1}{2}}(\mu_{1, \mathrm{obs}(m)}-\widehat{\mu}_{1,\mathrm{obs}(m)})\right\|.
\end{align}
The second statement of the Lemma can be proven in the same way. 
\end{proof}

\begin{lemma}\label{lemma:generalDecompLestim}
Grant Assumptions \ref{ass:MCAR},  \ref{ass:LDA} and assume the classes are balanced. Assume that we are given two estimators $\widehat{\mu}_1$ and $\widehat{\mu}_{-1}$ of $\mu_1$ and $\mu_{-1}$. Then, the classifier $\widehat{h}$ defined in Equation  \eqref{eq:PbPLDAEstimMuAllPatterns} satisfies 
    \begin{align*}
         & \mathcal{R}_{\mathrm{mis}}(\widehat{h}) -\mathcal{R}_{\mathrm{mis}}(h^\star)\\
    &\leq  \sum_{m\in\M}\left(\e{\left\|\Sigma_{\mathrm{obs}(m)}^{-\frac{1}{2}}(-\widehat{\mu}_{1,\mathrm{obs}(m)}
    +\mu_{1, \mathrm{obs}(m)})\right\|+\left\|\Sigma_{\mathrm{obs}(m)}^{-\frac{1}{2}}(\widehat{\mu}_{-1,\mathrm{obs}(m)}-\mu_{-1, \mathrm{obs}(m)})\right\|}\right) \frac{p_m}{\sqrt{2\pi}}.
    \end{align*}
\end{lemma}
\begin{proof}
We have
    \begin{align*}
         & \mathcal{R}_{\mathrm{mis}}(\widehat{h}) -\mathcal{R}_{\mathrm{mis}}(h^\star)\\
     = &\P\left(\widehat{h}(X_{\mathrm{obs}(M)},M)\neq Y\right)-\P\left(h^\star(X_{\mathrm{obs}(M)},M)\neq Y\right)\\ 
    =&\sum_{m\in\M}\left(\P\left(\widehat{h}(X_{\mathrm{obs}(M)},M)\neq Y\mid  M=m\right)-\P\left(h^\star(X_{\mathrm{obs}(M)},M)\neq Y\mid M=m\right)\right)p_m\\
    =&\sum_{m\in\M}\left(\P\left(\widehat{h}_m(X_{\mathrm{obs}(m)})\neq Y\right)-\P\left(h_m^\star(X_{\mathrm{obs}(m)})\neq Y\right)\right)p_m\tag{using Assumption \ref{ass:MCAR}}\\ 
    =&\sum_{m\in\M}\frac{1}{2}\left(\e{\P\left(\widehat{h}_m(X_{\mathrm{obs}(m)})=1\mid  Y=-1, \mathcal{D}_n\right)-\P\left(h_m^\star(X_{\mathrm{obs}(m)})=1 \mid Y=-1\right)}\right)p_m\\ 
    +&\sum_{m\in\M}\frac{1}{2}\left(\e{\P\left(\widehat{h}_m(X_{\mathrm{obs}(m)})=-1\mid  Y=1, \mathcal{D}_n\right)-\P\left(h_m^\star(X_{\mathrm{obs}(m)})=-1 \mid Y=1\right)}\right)p_m\\
    \leq& \sum_{m\in\M}\frac{p_m}{4\sqrt{2\pi}}\left(\e{3\left\|\Sigma_{\mathrm{obs}(m)}^{-\frac{1}{2}}(\mu_{-1,\mathrm{obs}(m)}-\widehat{\mu}_{-1,\mathrm{obs}(m)})\right\|+
    \left\|\Sigma_{\mathrm{obs}(m)}^{-\frac{1}{2}}(\mu_{1, \mathrm{obs}(m)}-\widehat{\mu}_{1,\mathrm{obs}(m)})\right\|}\right)\\ 
    +&\sum_{m\in\M}\frac{p_m}{4\sqrt{2\pi}}\left(\e{3\left\|\Sigma_{\mathrm{obs}(m)}^{-\frac{1}{2}}(-\widehat{\mu}_{1,\mathrm{obs}(m)}
    +\mu_{1, \mathrm{obs}(m)})\right\| 
    +\left\|\Sigma_{\mathrm{obs}(m)}^{-\frac{1}{2}}(\widehat{\mu}_{-1,\mathrm{obs}(m)}-\mu_{-1, \mathrm{obs}(m)})\right\|}\right),
    \end{align*}
    by Lemma \ref{lemma:boundMissProbMuEstim}. Thus, 
    \begin{align*}
& \mathcal{R}_{\mathrm{mis}}(\widehat{h}) -\mathcal{R}_{\mathrm{mis}}(h^\star)\\
    &=\sum_{m\in\M}\frac{p_m}{\sqrt{2\pi}}\left(\e{\left\|\Sigma_{\mathrm{obs}(m)}^{-\frac{1}{2}}(-\widehat{\mu}_{1,\mathrm{obs}(m)}
    +\mu_{1, \mathrm{obs}(m)})\right\|+\left\|\Sigma_{\mathrm{obs}(m)}^{-\frac{1}{2}}(\widehat{\mu}_{-1,\mathrm{obs}(m)}-\mu_{-1, \mathrm{obs}(m)})\right\|}\right).
    \end{align*}
\end{proof}
It is worth noting that, at this juncture, neither the structure of the estimate nor the structure of the covariance matrix have been incorporated. 

\subsubsection{Lemma for Theorem \ref{th:BoundPbPMuEstimLDAGen}}
\begin{lemma}\label{lemma:boundNormDifferMuSigmaGen}
    For all $m\in \M$ and all $k\in\{-1,1\}$, 
    \begin{align*}
        \e{\left\|\Sigma^{-\frac{1}{2}}_{\mathrm{obs}(m)}(\widehat{\mu}_{k,\mathrm{obs}(m)}
    -\mu_{k, \mathrm{obs}(m)})\right\|}& 
    \leq \left(\left(\frac{1+\eta}{2}\right)^n \frac{\left\|\mu\right\|_\infty^2(d-\left\|m\right\|_0)}{\lambda_{\min}\left(\Sigma\right)}+\frac{4\kappa(d-\left\|m\right\|_0)}{(n+1)\left(1-\eta\right)}\right)^\frac{1}{2},
    \end{align*}
    with $\widehat{\mu}_{k,\mathrm{obs}(m)}$ defined in \eqref{eq:estMu} and $\kappa:=\max_{i\in [n]}\Sigma_{i,i}/\lambda_{\min}(\Sigma)$ the greatest value of the diagonal of the covariance matrix divided by its smallest eigenvalue. 
\end{lemma}
\begin{proof}
First, by Jensen's inequality, 
\begin{align}
        \mathbb{E}&\left[\left\|\Sigma^{-\frac{1}{2}}_{\mathrm{obs}(m)}(\widehat{\mu}_{k,\mathrm{obs}(m)}
    -\mu_{k, \mathrm{obs}(m)})\right\|\right]\leq \e{\left\|\Sigma^{-\frac{1}{2}}_{\mathrm{obs}(m)}(\widehat{\mu}_{k,\mathrm{obs}(m)}-\mu_{k, \mathrm{obs}(m)})\right\|^2}^\frac{1}{2}\\
    &=\e{\left(\Sigma^{-\frac{1}{2}}_{\mathrm{obs}(m)}(\widehat{\mu}_{k,\mathrm{obs}(m)}-\mu_{k, \mathrm{obs}(m)})\right)^\top\left(\Sigma^{-\frac{1}{2}}_{\mathrm{obs}(m)}(\widehat{\mu}_{k,\mathrm{obs}(m)}-\mu_{k, \mathrm{obs}(m)})\right)}^\frac{1}{2}\\
    &=\e{\mathrm{tr}\left(\left(\Sigma^{-\frac{1}{2}}_{\mathrm{obs}(m)}(\widehat{\mu}_{k,\mathrm{obs}(m)}-\mu_{k, \mathrm{obs}(m)})\right)^\top\left(\Sigma^{-\frac{1}{2}}_{\mathrm{obs}(m)}(\widehat{\mu}_{k,\mathrm{obs}(m)}-\mu_{k, \mathrm{obs}(m)})\right)\right)}^\frac{1}{2}\\
    &=\e{\mathrm{tr}\left(\left(\Sigma^{-\frac{1}{2}}_{\mathrm{obs}(m)}(\widehat{\mu}_{k,\mathrm{obs}(m)}-\mu_{k, \mathrm{obs}(m)})\right)\left(\Sigma^{-\frac{1}{2}}_{\mathrm{obs}(m)}(\widehat{\mu}_{k,\mathrm{obs}(m)}-\mu_{k, \mathrm{obs}(m)})\right)^\top\right)}^\frac{1}{2}\\
     &=\mathrm{tr}\left(\e{\left(\Sigma^{-\frac{1}{2}}_{\mathrm{obs}(m)}(\widehat{\mu}_{k,\mathrm{obs}(m)}-\mu_{k, \mathrm{obs}(m)})\right)\left(\Sigma^{-\frac{1}{2}}_{\mathrm{obs}(m)}(\widehat{\mu}_{k,\mathrm{obs}(m)}-\mu_{k, \mathrm{obs}(m)})\right)^\top}\right)^\frac{1}{2}\\
     &=\mathrm{tr}\left(\Sigma^{-\frac{1}{2}}_{\mathrm{obs}(m)}\e{\left(\widehat{\mu}_{k,\mathrm{obs}(m)}-\mu_{k, \mathrm{obs}(m)}\right)\left(\widehat{\mu}_{k,\mathrm{obs}(m)}-\mu_{k, \mathrm{obs}(m)}\right)^\top}\Sigma^{-\frac{1}{2}}_{\mathrm{obs}(m)}\right)^\frac{1}{2}\\
     & = \mathrm{tr}\left(\Sigma^{-\frac{1}{2}}_{\mathrm{obs}(m)} \mathcal{C}(k,m) \Sigma^{-\frac{1}{2}}_{\mathrm{obs}(m)}\right)^\frac{1}{2}, \label{proof_136_first_ineq}
    \end{align}
    where 
    \begin{align}
    \mathcal{C}(k,m):=\e{\left(\widehat{\mu}_{k,\mathrm{obs}(m)}-\mu_{k, \mathrm{obs}(m)}\right)\left(\widehat{\mu}_{k,\mathrm{obs}(m)}-\mu_{k, \mathrm{obs}(m)}\right)^\top}.    
    \end{align}
    Now, we compute the elements $\mathcal{C}(k,m)_{r,l} =$, for all $r,l\in\mathrm{obs}(m)$. 

    \medskip 
    
    \textbf{First case.} We start by computing  $\mathcal{C}(k,m)_{l,l}$ for all $l$. Note that 
    \begin{align}
      \mathcal{C}(k,m)_{l,l} =  \e{\left(\widehat{\mu}_{k,l}-\mu_{k,l}\right)^2}.
    \end{align}
    The estimator $\widehat{\mu}_{k,l}$ equals zero if all samples of class $k$ have a missing $l$-th coordinate, which corresponds to the event 
    \begin{align}\label{eq:Akl}
        \mathcal{A}_{k,l}:=\{\forall i\in\{1,...n\}, \quad  Y_i=-k \text{ or }  M_{i,l}=1 \},
    \end{align}
    where 
    \begin{align}
    \P(\mathcal{A}_{k,l})=\prod_{i=1}^nP(Y_i=-k \text{ or } M_{i,l}=1)=\left(\frac{1+\eta}{2}\right)^n.    
    \end{align}
    Thus, 
    \begin{align*}
        \mathbb{E}&\left[\left(\widehat{\mu}_{k,l}-\mu_{k,l}\right)^2\right]\\
        & =\e{\left(\widehat{\mu}_{k,l}-\mu_{k,l}\right)^2\mid \mathcal{A}_{k,l}}\P\left(\mathcal{A}_{k,l}\right)+\e{\left(\widehat{\mu}_{k,l}-\mu_{k,l}\right)^2\mid \mathcal{A}_{k,l}^c}\P\left(\mathcal{A}_{k,l}^c\right)\\
        &=\mu_{k,l}^2\left(\frac{1+\eta}{2}\right)^n+\e{\left(\frac{\sum_{i=1}^n(X_{i,l}-\mu_{k,l})\ind_{Y_i=k}\ind_{M_{i,l}=0}}{\sum_{i=1}^n\ind_{Y_i=k}\ind_{M_{i,l}=0}}\right)^2\mid \mathcal{A}_{k,l}^c}\left(1-\left(\frac{1+\eta}{2}\right)^n\right)\\
        \end{align*}
       The second term can be rewritten as       
        \begin{align*}
        &= \sum_{i=1}^n\e{\frac{(X_{i,l}-\mu_{k,l})^2\ind_{Y_i=k}\ind_{M_{i,l}=0}}{\left(\sum_{i=1}^n\ind_{Y_i=k}\ind_{M_{i,l}=0}\right)^2}\mid \mathcal{A}_{k,l}^c}\left(1-\left(\frac{1+\eta}{2}\right)^n\right)\\
        &= \sum_{i=1}^n\e{\frac{(X_{i,l}-\mu_{k,l})^2}{\left(1+\sum_{j\neq i}^n\ind_{Y_j=k}\ind_{M_{j,l}=0}\right)^2}\mid \mathcal{A}_{k,l}^c, Y_i=k, M_{i,l}=0} \\
        & \qquad \times \P\left( Y_i=k, M_{i,l}=0\mid \mathcal{A}_{k,l}^c\right)\left(1-\left(\frac{1+\eta}{2}\right)^n\right)\\
        &= \left( \frac{1-\eta}{2} \right) \sum_{i=1}^n\e{\frac{(X_{i,l}-\mu_{k,l})^2}{\left(1+\sum_{j\neq i}^n\ind_{Y_j=k}\ind_{M_{j,l}=0}\right)^2}\mid  Y_i=k, M_{i,l}=0} \\
        &= n \left( \frac{1-\eta}{2} \right) \e{\frac{(X_{1,l}-\mu_{k,l})^2}{\left(1+\sum\limits_{j\neq 1}\ind_{Y_j=k}\ind_{M_{j,l}=0}\right)^2}\mid  Y_1=k} \tag{using ~\cref{ass:MCAR}}\\
        &= \left( \frac{1-\eta}{2} \right) n\e{(X_{1,l}-\mu_{k,l})^2\mid  Y_1=k}\e{\frac{1}{\left(1+\sum_{j\neq 1}^n\ind_{Y_j=k}\ind_{M_{j,l}=0}\right)^2}} \tag{using the independence}\\
        &= \left( \frac{1-\eta}{2} \right) n\Sigma_{l,l}\e{\frac{1}{\left(1+\sum_{j\neq 1}^n\ind_{Y_j=k}\ind_{M_{j,l}=0}\right)^2}}.\\
    \end{align*}
    In the sequel, we denote $A(n,\eta):=\e{\frac{1}{(1+B)^2}},$ where $B\sim \mathcal{B}(n-1,(1-\eta)/2).$
    Then, we have that 
    \begin{align}
        \mathcal{C}(k,m)_{l,l}=\mu_{k,l}^2\left(\frac{1+\eta}{2}\right)^n+n\left(\frac{1-\eta}{2}\right) \Sigma_{l,l}A(n,\eta) .\label{eq:cllProofLestim}
    \end{align}

    \textbf{Second case. }
    Now, we want to compute, for all $r\neq l$, 
    \begin{align}
    \mathcal{C}(k,m)_{r,l}=\e{\left(\widehat{\mu}_{k,r}-\mu_{k,r}\right)\left(\widehat{\mu}_{k,l}-\mu_{k,l}\right)}.    
    \end{align}
    To this aim, we distinguish three cases, depending on the presence of available samples to compute $\widehat{\mu}_{k,r}$ and $\widehat{\mu}_{k,r}$. First, let us denote by 
    \begin{align}
    \mathcal{A}_{k,l,r}:=\{\forall i \in \{1,...n\},  \left(Y_i=-k \textrm{ or } (M_{i,r}=1 \textrm{ and } M_{i,l}=1)\right) \},    
    \end{align}
    the event in which there is no available samples to estimate any of the means $\widehat{\mu}_{k,r}$ and $\widehat{\mu}_{k,r}$, that is each sample either belongs to the other class or is missing at both coordinates. We have
     \begin{align}
        \notag
            & \P(\mathcal{A}_{k,l,r})\\
            &=\P\left( \forall i \in \{1,...n\},   Y_i=-k \textrm{ or } (M_{i,r}=1 \textrm{ and } M_{i,l}=1)\right)\\ \notag
            &=\left(\P(Y_i=-k)+\P(M_{i,r}=1 \textrm{ and } M_{i,l}=1)-\P(Y_i=-k \textrm{ and } M_{i,r}=1 \textrm{ and } M_{i,l}=1)\right)^n\\
            &=\left(\frac{\eta^2+1}{2}\right)^n. \label{eq:probAklr}
        \end{align}
        Besides, on the event $\mathcal{A}_{k,l,r}$, we have
        \begin{align}
            \e{\left(\widehat{\mu}_{k,r}-\mu_{k,r}\right)\left(\widehat{\mu}_{k,l}-\mu_{k,l}\right) |\mathcal{A}_{k,l,r}}=\mu_{k,r}\mu_{k,l}.
            \label{eq:espAklr}
        \end{align}  

    We now consider the second case and denote by 
    \begin{align}
     \mathcal{B}_{k,l,r}:=\{\exists i\in \{1,...,n\},  (Y_i=k\land M_{i,l}=0 )=1\}\cap \{\exists i\in \{1,...n\}, (Y_i=k\land M_{i,r}=0)=1\},   
    \end{align}
    the event in which the there is at least one available sample to estimate both means $\widehat{\mu}_{k,r}$ and $\widehat{\mu}_{k,r}$.
    Observe that 
        \begin{align*}
            \P\left(\mathcal{B}_{k,l,r}\right)& =1-\P\left(\{\forall i\in \{1,...,n\}, \quad (Y_i=-k \textrm{ or } M_{i,l}=1 )\}\right.\\&\left. \qquad \qquad\textrm{ or } \{\forall i\in \{1,...n\},\quad (Y_i=-k \textrm{ or } M_{i,r}=1)\}\right)\\
            &=1-\P\left(\{\forall i\in \{1,...,n\}, \quad (Y_i=-k \textrm{ or } M_{i,l}=1 )\}\right)\\&\quad-\P\left(\{\forall i\in \{1,...n\},\quad (Y_i=-k \textrm{ or } M_{i,r}=1)\}\right)\\&\quad+\P\left(\{\forall i\in \{1,...,n\}, \quad (Y_i=-k \textrm{ or } (M_{i,l}=1 \textrm{ and } M_{i,r}=1) )\}\right),
        \end{align*}
        where the last probability was already computed for $\mathcal{A}_{k,l,r}$. On the other hand, remark that
        \begin{align*}
            \P&\left(\{\forall i\in \{1,...n\},\quad (Y_i=-k \textrm{ or } M_{i,r}=1)\}\right)
            =\left(\frac{1+\eta}{2}\right)^n.
        \end{align*}
    Then, we have that 
    \begin{align}
        \P&\left(\mathcal{B}_{k,l,r}\right)=1-2\left(\frac{1+\eta}{2}\right)^n+\left(\frac{\eta^2+1}{2}\right)^n.\label{eq:probBklr}
    \end{align}
Besides,  
     \begin{align*}
            \mathbb{E}&\left[\left(\widehat{\mu}_{k,r}-\mu_{k,r}\right)\left(\widehat{\mu}_{k,l}-\mu_{k,l}\right)\mid \mathcal{B}_{k,l,r}\right]
            \\&=\e{\left(\frac{\sum_{i=1}^n(X_{i,r}-\mu_{k,r})\ind_{Y_i=k}\ind_{M_{i,r}=0}}{\sum_{i=1}^n\ind_{Y_i=k}\ind_{M_{i,r}=0}}\right)\left(\frac{\sum_{i=1}^n(X_{i,l}-\mu_{k,l})\ind_{Y_i=k}\ind_{M_{i,l}=0}}{\sum_{i=1}^n\ind_{Y_i=k}\ind_{M_{i,l}=0}}\right)\mid \mathcal{B}_{k,l,r}}
            \\&=\sum_{i=1}^n\sum_{j=1}^n\e{\left(\frac{(X_{i,r}-\mu_{k,r})\ind_{Y_i=k}\ind_{M_{i,r}=0}}{\sum_{i=1}^n\ind_{Y_i=k}\ind_{M_{i,r}=0}}\right)\left(\frac{(X_{j,l}-\mu_{k,l})\ind_{Y_j=k}\ind_{M_{j,l}=0}}{\sum_{i=1}^n\ind_{Y_i=k}\ind_{M_{i,l}=0}}\right)\mid \mathcal{B}_{k,l,r}}\\&=\sum_{i=1}^n\e{\left(\frac{(X_{i,r}-\mu_{k,r})\ind_{Y_i=k}\ind_{M_{i,r}=0}}{\sum_{i=1}^n\ind_{Y_i=k}\ind_{M_{i,r}=0}}\right)\left(\frac{(X_{i,l}-\mu_{k,l})\ind_{Y_i=k}\ind_{M_{i,l}=0}}{\sum_{i=1}^n\ind_{Y_i=k}\ind_{M_{i,l}=0}}\right)\mid \mathcal{B}_{k,l,r}}
            \\&\qquad+\sum_{i=1}^n\sum_{j\neq i}^n\e{\left(\frac{(X_{i,r}-\mu_{k,r})\ind_{Y_i=k}\ind_{M_{i,r}=0}}{\sum_{i=1}^n\ind_{Y_i=k}\ind_{M_{i,r}=0}}\right)\left(\frac{(X_{j,l}-\mu_{k,l})\ind_{Y_j=k}\ind_{M_{j,l}=0}}{\sum_{i=1}^n\ind_{Y_i=k}\ind_{M_{i,l}=0}}\right)\mid \mathcal{B}_{k,l,r}}.
        \end{align*}   
        Observe that this second sum is null. Indeed,
        \begin{align*}
            \mathbb{E}&\left[\left(\frac{(X_{i,r}-\mu_{k,r})\ind_{Y_i=k}\ind_{M_{i,r}=0}}{\sum_{i=1}^n\ind_{Y_i=k}\ind_{M_{i,r}=0}}\right)\left(\frac{(X_{j,l}-\mu_{k,l})\ind_{Y_j=k}\ind_{M_{j,l}=0}}{\sum_{i=1}^n\ind_{Y_i=k}\ind_{M_{i,l}=0}}\right)\mid \mathcal{B}_{k,l,r}\right]\\
            &=\mathbb{E}\left[\left(\frac{(X_{i,r}-\mu_{k,r})}{1+\ind_{M_{j,r}=0}+\sum_{s\neq i, j}^n\ind_{Y_s=k}\ind_{M_{s,r}=0}}\right)\left(\frac{(X_{j,l}-\mu_{k,l})}{1+\ind_{M_{i,l}=0}+\sum_{s\neq i, j}^n\ind_{Y_s=k}\ind_{M_{s,l}=0}}\right)\right.\\&\qquad\left.\mid Y_i=k, M_{i,r}=0, Y_j=k, M_{j,l}=0\right]\P\left(Y_i=k, M_{i,r}=0, Y_j=k, M_{j,l}=0\mid \mathcal{B}_{k,l,r}\right)\\
            &=\mathbb{E}\left[\frac{1}{\left(1+\ind_{M_{j,r}=0}+\sum_{s\neq i, j}^n\ind_{Y_s=k}\ind_{M_{s,r}=0}\right)\left(1+\ind_{M_{i,l}=0}+\sum_{s\neq i, j}^n\ind_{Y_s=k}\ind_{M_{s,l}=0}\right)}\right]\\&\qquad \times \e{(X_{i,r}-\mu_{k,r})\mid Y_i=k}\mathbb{E}\left[(X_{j,l}-\mu_{k,l})\mid Y_j=k\right]\\
            & \qquad \times \P\left(Y_i=k, M_{i,r}=0, Y_j=k, M_{j,l}=0\mid \mathcal{B}_{k,l,r}\right)\tag{using ~\cref{ass:MCAR} and independence}
            \\&=0.
        \end{align*}
        Then, 
        \begin{align}
        \notag
            \mathbb{E}&\left[\left(\widehat{\mu}_{k,r}-\mu_{k,r}\right)\left(\widehat{\mu}_{k,l}-\mu_{k,l}\right)\mid \mathcal{B}_{k,l,r}\right]
            \\ \notag &=\sum_{i=1}^n\e{\left(\frac{(X_{i,r}-\mu_{k,r})\ind_{Y_i=k}\ind_{M_{i,r}=0}}{\sum_{i=1}^n\ind_{Y_i=k}\ind_{M_{i,r}=0}}\right)\left(\frac{(X_{i,l}-\mu_{k,l})\ind_{Y_i=k}\ind_{M_{i,l}=0}}{\sum_{i=1}^n\ind_{Y_i=k}\ind_{M_{i,l}=0}}\right)\mid \mathcal{B}_{k,l,r}}
            \\ \notag &=\sum_{i=1}^n\e{\frac{(X_{i,r}-\mu_{k,r})}{1+\sum_{j\neq i}^n\ind_{Y_j=k}\ind_{M_{j,r}=0}}\frac{(X_{i,l}-\mu_{k,l})}{1+\sum_{j\neq i}^n\ind_{Y_j=k}\ind_{M_{j,l}=0}}\mid Y_i=k, M_{i,r}=0, M_{i,l}=0}\\ \notag &\qquad\P\left(Y_i=k, M_{i,r}=0, M_{i,l}=0\mid \mathcal{B}_{k,l,r}\right)
            \\ \notag &=\sum_{i=1}^n\e{(X_{i,r}-\mu_{k,r})(X_{i,l}-\mu_{k,l})\mid Y_i=k}\\
            & \qquad \times \e{\frac{1}{1+\sum_{j\neq i}^n\ind_{Y_j=k}\ind_{M_{j,r}=0}}\frac{1}{1+\sum_{j\neq i}^n\ind_{Y_j=k}\ind_{M_{j,l}=0}}} \notag \\ 
            \notag&\qquad \times \P\left(Y_i=k, M_{i,r}=0, M_{i,l}=0\mid \mathcal{B}_{k,l,r}\right)
            \\&=n\Sigma_{r,l}B(n,\eta)\frac{(1-\eta)^2}{2\P\left(\mathcal{B}_{k,l,r}\right)}, \label{eq:espBklr}
        \end{align}
        where $B(n,\eta):=\e{\frac{1}{1+\sum_{j=2}^n\ind_{Y_j=k}\ind_{M_{j,r}=0}}\frac{1}{1+\sum_{j=2}^n\ind_{Y_j=k}\ind_{M_{j,l}=0}}}$.

       Now, we consider the last case, and denote by 
       \begin{align}
           \mathcal{C}_{k,l,r} = \left(\mathcal{B}_{k,l,r}\cup\mathcal{A}_{k,l,r}\right)^c
       \end{align}
       the event in which only one mean can be estimated. We have
\begin{align*}
    \P(\mathcal{C}_{k,l,r})=\P&\left(\left(\mathcal{B}_{k,l,r}\cup\mathcal{A}_{k,l,r}\right)^c\right)=2\left(\frac{1+\eta}{2}\right)^n-2
    \left(\frac{\eta^2+1}{2}\right)^n.
\end{align*}
Let $\mathcal{C}_{k,l,r}=\mathcal{C}_{1,k,l,r} \cup \mathcal{C}_{2,k,l,r}$, where $\mathcal{C}_{1,k,l,r}$ is the event where the one that can be estimated is $\widehat{\mu}_{k,r}$. Then, 
\begin{align}
\notag
    \mathbb{E}&\left[\left(\widehat{\mu}_{k,r}-\mu_{k,r}\right)\left(\widehat{\mu}_{k,l}-\mu_{k,l}\right)\mid \mathcal{C}_{1,k,l,r}\right]\\ \notag
    &=-\mu_{k,l}\e{\frac{\sum^n_{i=1}(X_{i,r}-\mu_{k,r})\ind_{M_{i,r}=0}\ind_{Y_i=k}}{\sum^n_{i=1}\ind_{M_{i,r}=0}\ind_{Y_i=k}}\mid \mathcal{C}_{1,k,l,r}}\\ \notag
    &=-\mu_{k,l}n\e{\frac{(X_{1,r}-\mu_{k,r})}{1+\sum^n_{i=2}\ind_{M_{i,r}=0}\ind_{Y_i=k}}\mid M_{1,r}=0,Y_1=k}\P(M_{1,r}=0,Y_1=k|\mathcal{C}_{1,k,l,r})\\
    &=-\mu_{k,l}n\e{(X_{1,r}-\mu_{k,r})\mid Y_1=k}\e{\frac{1}{1+\sum^n_{i=2}\ind_{M_{i,r}=0}\ind_{Y_i=k}}}\P(M_{1,r}=0,Y_1=k|\mathcal{C}_{1,k,l,r})\tag{using MCAR and independence}
    \\&=0.\label{eq:espCklr}
\end{align}
By symmetry, we also have 
\begin{align}
    \mathbb{E}\left[\left(\widehat{\mu}_{k,r}-\mu_{k,r}\right)\left(\widehat{\mu}_{k,l}-\mu_{k,l}\right)\mid \mathcal{C}_{2,k,l,r}\right] =0.
\end{align}
Thus, 
\begin{align}
    \mathbb{E}\left[\left(\widehat{\mu}_{k,r}-\mu_{k,r}\right)\left(\widehat{\mu}_{k,l}-\mu_{k,l}\right)\mid \mathcal{C}_{k,l,r}\right] =0. \label{computation_matrix3}
\end{align}
Gathering \eqref{eq:espAklr}, \eqref{eq:espBklr} and \eqref{computation_matrix3}, we are able to compute $\mathcal{C}(k,m)_{r,l}$ as follows
    \begin{align}
     \notag
        \mathcal{C}(k,m)_{r,l} 
        &=\e{\left(\widehat{\mu}_{k,r}-\mu_{k,r}\right)\left(\widehat{\mu}_{k,l}-\mu_{k,l}\mid \mathcal{A}_{k,l,r}\right)}\P\left(\mathcal{A}_{k,l,r}\right)\\  \notag
        &\qquad+\e{\left(\widehat{\mu}_{k,r}-\mu_{k,r}\right)\left(\widehat{\mu}_{k,l}-\mu_{k,l}\mid \mathcal{B}_{k,l,r}\right)}\P\left(\mathcal{B}_{k,l,r}\right)\\ \notag
        &\qquad +\e{\left(\widehat{\mu}_{k,r}-\mu_{k,r}\right)\left(\widehat{\mu}_{k,l}-\mu_{k,l}\mid \mathcal{C}(k,m)_{k,l,r}\right)}\P\left(\mathcal{C}(k,m)_{k,l,r}\right)\\ \notag
        &=\mu_{k,r}\mu_{k,l}\left(\frac{\eta^2+1}{2}\right)^n+n\Sigma_{r,l}B(n,\eta)\frac{(1-\eta)^2}{2}, \label{eq:crlProofLestim}
    \end{align}
 using \eqref{eq:probAklr} and \eqref{eq:probBklr}. From \eqref{eq:cllProofLestim}, recall that 
\begin{align}
        \mathcal{C}(k,m)_{l,l}=\mu_{k,l}^2\left(\frac{1+\eta}{2}\right)^n+n\left(\frac{1-\eta}{2}\right) \Sigma_{l,l}A(n,\eta) .
    \end{align}
 Let $J$ be the matrix composed of $1$ in each entry, and let 
 \begin{align}
     F & = \left( \left(\frac{1+\eta}{2}\right)^n - \left(\frac{1+\eta^2}{2}\right)^n \right) I + \left(\frac{1+\eta^2}{2}\right)^n J\\
     G & = \left( A(n, \eta) - (1-\eta)B(n,\eta) \right) I + (1-\eta)B(n,\eta) J.
 \end{align}
   Thus,
    \begin{align*}
        \mathcal{C}(k,m)=F\odot \mu_{k,\mathrm{obs}(m)}\mu_{k,\mathrm{obs}(m)}^\top+n\frac{1-\eta}{2}G\odot \Sigma_{\mathrm{obs}(m)}.
    \end{align*}
    
    Then, according to inequality \eqref{proof_136_first_ineq},  we have that 
    \begin{align}
    \notag
         \mathbb{E}&\left[\left\|\Sigma^{-\frac{1}{2}}_{\mathrm{obs}(m)}(\widehat{\mu}_{k,\mathrm{obs}(m)}
    -\mu_{k, \mathrm{obs}(m)})\right\|\right] \\
    & \leq \mathrm{tr}\left(\Sigma^{-\frac{1}{2}}_{\mathrm{obs}(m)}\mathcal{C}(k,m)\Sigma^{-\frac{1}{2}}_{\mathrm{obs}(m)}\right)^\frac{1}{2}\\
    &=\left(\mathrm{tr}\left(\Sigma^{-\frac{1}{2}}_{\mathrm{obs}(m)}\left(F\odot \mu_{k,\mathrm{obs}(m)}\mu_{k,\mathrm{obs}(m)}^\top\right)\Sigma^{-\frac{1}{2}}_{\mathrm{obs}(m)}\right)+n\frac{1-\eta}{2}\mathrm{tr}\left(\Sigma^{-\frac{1}{2}}_{\mathrm{obs}(m)}\left(G\odot \Sigma_{\mathrm{obs}(m)}\right)\Sigma^{-\frac{1}{2}}_{\mathrm{obs}(m)}\right)\right)^\frac{1}{2}\label{eq:endProofSigmaGen}
    \end{align}
   The first term equals
\begin{align}
    \mathrm{tr}&\left(\Sigma^{-\frac{1}{2}}_{\mathrm{obs}(m)}\left(F\odot \mu_{k,\mathrm{obs}(m)}\mu_{k,\mathrm{obs}(m)}^\top\right)\Sigma^{-\frac{1}{2}}_{\mathrm{obs}(m)}\right)
    \\&= \left(\frac{1+\eta^2}{2}\right)^n\mathrm{tr}\left(\Sigma^{-\frac{1}{2}}_{\mathrm{obs}(m)}\left( J \odot \mu_{k,\mathrm{obs}(m)}\mu_{k,\mathrm{obs}(m)}^\top\right)\Sigma^{-\frac{1}{2}}_{\mathrm{obs}(m)}\right)\\ &\qquad+
    \left(\left(\frac{1+\eta}{2}\right)^n -\left(\frac{1+\eta^2}{2}\right)^n \right)\mathrm{tr}\left(\Sigma^{-\frac{1}{2}}_{\mathrm{obs}(m)}\left(I_{d-\left\|m\right\|_0}\odot \mu_{k,\mathrm{obs}(m)}\mu_{k,\mathrm{obs}(m)}^\top\right)\Sigma^{-\frac{1}{2}}_{\mathrm{obs}(m)}\right)\\
    &= \left(\frac{1+\eta^2}{2}\right)^n\mathrm{tr}\left(\Sigma^{-\frac{1}{2}}_{\mathrm{obs}(m)}\mu_{k,\mathrm{obs}(m)}\mu_{k,\mathrm{obs}(m)}^\top\Sigma^{-\frac{1}{2}}_{\mathrm{obs}(m)}\right)\\ &\qquad+
    \left(\left(\frac{1+\eta}{2}\right)^n -\left(\frac{1+\eta^2}{2}\right)^n \right)\mathrm{tr}\left(\Sigma^{-\frac{1}{2}}_{\mathrm{obs}(m)}\mathrm{diag}\left(\mu_{k,\mathrm{obs}(m)}\mu_{k,\mathrm{obs}(m)}^\top\right)\Sigma^{-\frac{1}{2}}_{\mathrm{obs}(m)}\right).
    \end{align}
    Then, by \cref{lemma:traceIneq}, 
    \begin{align}
     \mathrm{tr}&\left(\Sigma^{-\frac{1}{2}}_{\mathrm{obs}(m)}\left(F\odot \mu_{k,\mathrm{obs}(m)}\mu_{k,\mathrm{obs}(m)}^\top\right)\Sigma^{-\frac{1}{2}}_{\mathrm{obs}(m)}\right) \\
     &\leq \left(\frac{1+\eta^2}{2}\right)^n\mathrm{tr}\left(\left(\Sigma^{-\frac{1}{2}}_{\mathrm{obs}(m)}\mu_{k,\mathrm{obs}(m)}\right)\left(\Sigma^{-\frac{1}{2}}_{\mathrm{obs}(m)}\mu_{k,\mathrm{obs}(m)}\right)^\top\right)\\ &\qquad+
    \left(\left(\frac{1+\eta}{2}\right)^n -\left(\frac{1+\eta^2}{2}\right)^n \right)\left\|\mu\right\|_\infty^2\mathrm{tr}\left(\Sigma^{-1}_{\mathrm{obs}(m)}\right)\\
    &=\left(\frac{1+\eta^2}{2}\right)^n\left\|\Sigma^{-\frac{1}{2}}_{\mathrm{obs}(m)}\mu_{k,\mathrm{obs}(m)}\right\|^2+
    \left(\left(\frac{1+\eta}{2}\right)^n -\left(\frac{1+\eta^2}{2}\right)^n \right)\frac{\left\|\mu\right\|_\infty^2(d-\left\|m\right\|_0)}{\lambda_{\min}\left(\Sigma\right)}\\
    &\leq\left(\frac{1+\eta^2}{2}\right)^n\frac{\left\|\mu_{k,\mathrm{obs}(m)}\right\|^2}{\lambda_{\min}\left(\Sigma\right)}+
    \left(\left(\frac{1+\eta}{2}\right)^n -\left(\frac{1+\eta^2}{2}\right)^n \right)\frac{\left\|\mu\right\|_\infty^2(d-\left\|m\right\|_0)}{\lambda_{\min}\left(\Sigma\right)}\\
    &\leq\left(\frac{1+\eta^2}{2}\right)^n\frac{\left\|\mu\right\|_\infty^2(d-\left\|m\right\|_0)}{\lambda_{\min}\left(\Sigma\right)}+
    \left(\left(\frac{1+\eta}{2}\right)^n -\left(\frac{1+\eta^2}{2}\right)^n \right)\frac{\left\|\mu\right\|_\infty^2(d-\left\|m\right\|_0)}{\lambda_{\min}\left(\Sigma\right)}\\
    & \leq\left(\frac{1+\eta}{2}\right)^n \frac{\left\|\mu\right\|_\infty^2(d-\left\|m\right\|_0)}{\lambda_{\min}\left(\Sigma\right)}. \label{eq:ineqTraceF}
    \end{align}

    Regarding the second term in \eqref{eq:endProofSigmaGen}, note that  $A(n,\eta)-(1-\eta)B(n,\eta)\geq 0$. Indeed, letting $Z:=\sum_{i=1}^{n-1}\ind_{Y_i=k}\sim \mathcal{B}(n-1, 1/2)$,
\begin{align*}
    B(n,\eta)&:=\e{\frac{1}{1+\sum_{j=1}^{n-1}\ind_{Y_j=k}\ind_{M_{j,r}=0}}\frac{1}{1+\sum_{j=1}^{n-1}\ind_{Y_j=k}\ind_{M_{j,l}=0}}}\\
    &=\e{\e{\frac{1}{1+\sum_{j=1}^{Z}\ind_{M_{j,r}=0}}\frac{1}{1+\sum_{j=1}^{Z}\ind_{M_{j,l}=0}}\mid Z}},
\end{align*}
using the exchangeability as the samples are i.i.d. By leveraging the independence between the missingness at coordinate $r$ and coordinate $l$, as well as the independence of each sample from the rest, we can conclude that
\begin{align*}
    \mathbb{E}&\left[\e{\frac{1}{1+\sum_{j=1}^{Z}\ind_{M_{j,r}=0}}\frac{1}{1+\sum_{j=1}^{Z}\ind_{M_{j,l}=0}}\mid Z}\right]\\
    &=\e{\e{\frac{1}{1+\sum_{j=1}^{Z}\ind_{M_{j,r}=0}}\mid Z}\e{\frac{1}{1+\sum_{j=1}^{Z}\ind_{M_{j,l}=0}}\mid Z}}\\
    &=\e{\e{\frac{1}{1+\sum_{j=1}^{Z}\ind_{M_{j,r}=0}}\mid Z}^2}\tag{using that $M_{j,r}\sim M_{j,l}$}\\
    &\leq \e{\e{\frac{1}{\left(1+\sum_{j=1}^{Z}\ind_{M_{j,r}=0}\right)^2}\mid Z}}\tag{using Jensen Inequality}\\
    &=A(n,\eta).
\end{align*}
        Thus, we have that 
        \begin{align}
            \mathrm{tr}&\left(\Sigma^{-\frac{1}{2}}_{\mathrm{obs}(m)}\left(G\odot \Sigma_{\mathrm{obs}(m)}\right)\Sigma^{-\frac{1}{2}}_{\mathrm{obs}(m)}\right)\\
            &=(1-\eta)B(n,\eta)\mathrm{tr}\left(\Sigma^{-\frac{1}{2}}_{\mathrm{obs}(m)}\left(\textbf{1}\odot \Sigma_{\mathrm{obs}(m)}\right)\Sigma^{-\frac{1}{2}}_{\mathrm{obs}(m)}\right)
            \\&\qquad+(A(n,\eta)-(1-\eta)B(n,\eta))\mathrm{tr}\left(\Sigma^{-\frac{1}{2}}_{\mathrm{obs}(m)}\left(I_{d-\left\|m\right\|_0}\odot \Sigma_{\mathrm{obs}(m)}\right)\Sigma^{-\frac{1}{2}}_{\mathrm{obs}(m)}\right)\\
            &=(1-\eta)B(n,\eta)\mathrm{tr}\left(\Sigma^{-\frac{1}{2}}_{\mathrm{obs}(m)} \Sigma_{\mathrm{obs}(m)}\Sigma^{-\frac{1}{2}}_{\mathrm{obs}(m)}\right)
            \\&\qquad+(A(n,\eta)-(1-\eta)B(n,\eta))\mathrm{tr}\left(\Sigma^{-\frac{1}{2}}_{\mathrm{obs}(m)}\mathrm{diag}\left(\Sigma_{\mathrm{obs}(m)}\right)\Sigma^{-\frac{1}{2}}_{\mathrm{obs}(m)}\right).
            \end{align}
            Using ~\cref{lemma:traceIneq} and $A(n,\eta)-(1-\eta)B(n,\eta)\geq 0$, we have
            \begin{align}
            \mathrm{tr}&\left(\Sigma^{-\frac{1}{2}}_{\mathrm{obs}(m)}\left(G\odot \Sigma_{\mathrm{obs}(m)}\right)\Sigma^{-\frac{1}{2}}_{\mathrm{obs}(m)}\right)\\
            &\leq (1-\eta)B(n,\eta)(d-\left\|m\right\|_0)
            \\&\qquad+(A(n,\eta)-(1-\eta)B(n,\eta))\max_{i\in [d]}\left(\Sigma_{i,i}\right)\mathrm{tr}\left(\Sigma^{-1}_{\mathrm{obs}(m)}\right)\\
            &\leq (1-\eta)B(n,\eta)(d-\left\|m\right\|_0)
            \\&\qquad+(A(n,\eta)-(1-\eta)B(n,\eta))\frac{\max_{i\in [d]}\left(\Sigma_{i,i}\right)}{\lambda_{\min}\left(\Sigma\right)}(d-\left\|m\right\|_0)\\
            &\leq \kappa(1-\eta)B(n,\eta)(d-\left\|m\right\|_0)
            +(A(n,\eta)-(1-\eta)B(n,\eta))\kappa(d-\left\|m\right\|_0) \\
            &=A(n,\eta)\kappa(d-\left\|m\right\|_0)\\
            & \leq \frac{2\kappa(d-\left\|m\right\|_0)}{n(n+1)\left(\frac{1-\eta}{2}\right)^2},\label{eq:ineqTraceG}
        \end{align}
        where $\kappa:=\frac{\max_{i\in [d]}\left(\Sigma_{i,i}\right)}{\lambda_{\min}\left(\Sigma\right)}\geq 1$.     Finally, combining \eqref{eq:ineqTraceF} and \eqref{eq:ineqTraceG} in  \eqref{eq:endProofSigmaGen}, we have
    \begin{align*}
        \mathbb{E}&\left[\left\|\Sigma^{-\frac{1}{2}}_{\mathrm{obs}(m)}(\widehat{\mu}_{k,\mathrm{obs}(m)}
    -\mu_{k, \mathrm{obs}(m)})\right\|\right] \\
    &\leq \Bigg(\mathrm{tr}\left(\Sigma^{-\frac{1}{2}}_{\mathrm{obs}(m)}\left(F\odot \mu_{k,\mathrm{obs}(m)}\mu_{k,\mathrm{obs}(m)}^\top\right)\Sigma^{-\frac{1}{2}}_{\mathrm{obs}(m)}\right)\\
    & \quad +n\frac{1-\eta}{2}\mathrm{tr}\left(\Sigma^{-\frac{1}{2}}_{\mathrm{obs}(m)}\left(G\odot \Sigma_{\mathrm{obs}(m)}\right)\Sigma^{-\frac{1}{2}}_{\mathrm{obs}(m)}\right)\Bigg)^\frac{1}{2}\\
    &\leq \left(\left(\frac{1+\eta}{2}\right)^n \frac{\left\|\mu\right\|_\infty^2(d-\left\|m\right\|_0)}{\lambda_{\min}\left(\Sigma\right)}+n\frac{1-\eta}{2}\frac{2\kappa(d-\left\|m\right\|_0)}{n(n+1)\left(\frac{1-\eta}{2}\right)^2}\right)^\frac{1}{2} \\
    &\leq \left(\left(\frac{1+\eta}{2}\right)^n \frac{\left\|\mu\right\|_\infty^2(d-\left\|m\right\|_0)}{\lambda_{\min}\left(\Sigma\right)}+\frac{4\kappa(d-\left\|m\right\|_0)}{(n+1)\left(1-\eta\right)}\right)^\frac{1}{2}.
    \end{align*}
\end{proof}
\subsubsection{Proof of Theorem \ref{th:BoundPbPMuEstimLDAGen}}\label{subsec:BoundPbPMuEstimLDAGen}

\begin{proof}
By ~\cref{lemma:generalDecompLestim},

    \begin{align*}
    \mathcal{R}_{\mathrm{mis}}(\widehat{h}) &-\mathcal{R}_{\mathrm{mis}}(h^\star)\\
    &\leq \sum_{m\in\M}\frac{1}{\sqrt{2\pi}}\Big(\mathds{E} \Big[\left\|\Sigma_{\mathrm{obs}(m)}^{-\frac{1}{2}}(-\widehat{\mu}_{1,\mathrm{obs}(m)}
    +\mu_{1, \mathrm{obs}(m)})\right\|\\
    & \quad +\left\|\Sigma_{\mathrm{obs}(m)}^{-\frac{1}{2}}(\widehat{\mu}_{-1,\mathrm{obs}(m)}-\mu_{-1, \mathrm{obs}(m)})\right\| \Big]\Big)p_m\\
    &\leq \frac{2}{\sqrt{2\pi}}\sum_{m\in\M}\left(\left(\frac{1+\eta}{2}\right)^n \frac{\left\|\mu\right\|_\infty^2(d-\left\|m\right\|_0)}{\lambda_{\min}\left(\Sigma\right)}+\frac{4\kappa(d-\left\|m\right\|_0)}{(n+1)\left(1-\eta\right)}\right)^\frac{1}{2}p_m.\tag{using ~\cref{lemma:boundNormDifferMuSigmaGen}}
    \end{align*}
    Now, using ~\cref{ass:constEtaMi}, we have that $\left\|M\right\|_0\sim \mathcal{B}(d, \eta)$, so that
    \begin{align*}
        \frac{2}{\sqrt{2\pi}}&\sum_{m\in\M}\left(\left(\frac{1+\eta}{2}\right)^n \frac{\left\|\mu\right\|_\infty^2(d-\left\|m\right\|_0)}{\lambda_{\min}\left(\Sigma\right)}+\frac{4\kappa(d-\left\|m\right\|_0)}{(n+1)\left(1-\eta\right)}\right)^\frac{1}{2}p_m\\
        &=\frac{2}{\sqrt{2\pi}}\e{\left(\left(\frac{1+\eta}{2}\right)^n \frac{\left\|\mu\right\|_\infty^2(d-B)}{\lambda_{\min}\left(\Sigma\right)}+\frac{4\kappa(d-B)}{(n+1)\left(1-\eta\right)}\right)^\frac{1}{2}}\tag{where $B\sim \mathcal{B}(d, \eta)$}\\
        &\leq \frac{2}{\sqrt{2\pi}}\e{\left(\left(\frac{1+\eta}{2}\right)^n \frac{\left\|\mu\right\|_\infty^2(d-B)}{\lambda_{\min}\left(\Sigma\right)}+\frac{4\kappa(d-B)}{(n+1)\left(1-\eta\right)}\right)}^\frac{1}{2}\tag{using Jensen Inequality}\\
        &\leq \frac{2}{\sqrt{2\pi}}\left(\left(\frac{1+\eta}{2}\right)^n \frac{\left\|\mu\right\|_\infty^2d(1-\eta)}{\lambda_{\min}\left(\Sigma\right)}+\frac{4\kappa d}{n}\right)^\frac{1}{2}.
    \end{align*}
\end{proof}

\subsubsection{Proof of Corollary \ref{cor:final_LDA_upper_bound_Sigma_qcq}}\label{subsec:final_LDA_upper_bound}

\begin{proof} 
From \cref{prop:BoundDiffLLcomp} and ~\cref{th:BoundPbPMuEstimLDAGen} we have that 
\begin{align*}
    L&(\widehat{h})-\mathcal{R}_{\mathrm{comp}}(h^\star_\mathrm{comp})=\mathcal{R}_{\mathrm{mis}}(\widehat{h})-\mathcal{R}_{\mathrm{mis}}(h^\star)+\mathcal{R}_{\mathrm{mis}}(h^\star)-\mathcal{R}_{\mathrm{comp}}(h^\star_\mathrm{comp})\\
 &\leq \frac{2}{\sqrt{2\pi}}\left(\left(\frac{1+\eta}{2}\right)^n \frac{\left\|\mu\right\|_\infty^2 d(1-\eta)}{\lambda_{\min}\left(\Sigma\right)}+\frac{4\kappa d}{n}\right)^\frac{1}{2} +\left(\frac{1}{2}-\Phi\left(-\frac{\mu}{2}\sqrt{\frac{d}{\lambda_{\min}(\Sigma)}}\right)\right)\eta^d\\&\qquad+\frac{\mu}{2\sqrt{2\pi}}\left(\sqrt{\frac{d}{\lambda_{\min}(\Sigma)}}\left(\left(\eta+e^{-\frac{\mu^2}{8\lambda_{\max}(\Sigma)}}(1-\eta)\right)^d-\eta^d\right)\right.\\&\qquad\left.-\sqrt{\frac{d}{\lambda_{\max}(\Sigma)}}\left(\eta+e^{-\frac{\mu^2}{8\lambda_{\max}(\Sigma)}}(1-\eta)\right)^{d-1}e^{-\frac{\mu^2}{8\lambda_{\max}(\Sigma)}}(1-\eta)\right)\\
 &= \frac{2}{\sqrt{2\pi}}\left(\left(\frac{1+\eta}{2}\right)^n \frac{\left\|\mu\right\|_\infty^2 d(1-\eta)}{\lambda_{\min}\left(\Sigma\right)}+\frac{4\kappa d}{n}\right)^\frac{1}{2} +\left(\frac{1}{2}-\Phi\left(-\frac{\mu}{2\sigma}\sqrt{d}\right)\right)\eta^d\\&\qquad+\frac{\mu\sqrt{d}}{2\sigma\sqrt{2\pi}}\left(\left(\eta+e^{-\frac{\mu^2}{8\sigma^2}}(1-\eta)\right)^d-\eta^d-\left(\eta+e^{-\frac{\mu^2}{8\sigma^2}}(1-\eta)\right)^{d-1}e^{-\frac{\mu^2}{8\sigma^2}}(1-\eta)\right)\\
 &=\frac{2}{\sqrt{2\pi}}\left(\left(\frac{1+\eta}{2}\right)^n \frac{\left\|\mu\right\|_\infty^2 d(1-\eta)}{\lambda_{\min}\left(\Sigma\right)}+\frac{4\kappa d}{n}\right)^\frac{1}{2} +\left(\frac{1}{2}-\Phi\left(-\frac{\mu}{2\sigma}\sqrt{d}\right)\right)\eta^d\\&\qquad+\frac{\eta\mu\sqrt{d}}{2\sigma\sqrt{2\pi}}\left(\left(\eta+e^{-\frac{\mu^2}{8\sigma^2}}(1-\eta)\right)^{d-1}-\eta^{d-1}\right).
\end{align*}    
\end{proof}

\section{(LDA + MNAR) Proofs of Section \ref{subsec:LDAMNAR}}\label{subsec:proofsLDAMNAR}
\subsection{Proof of Proposition \ref{prop:MNARLDA}}\label{subsubsec:MNARLDA}
\begin{proof}
    By definition of the Bayes classifier (see \eqref{eq:BayesClassifPbP}), 
    \begin{align*}
        h_m^{\star}&(X_{\mathrm{obs}(m)})\\&=\mathrm{sign}(\e{Y|X_{\mathrm{obs}(m)},M=m})\\
        &=\mathrm{sign}\left( \P\left(Y=1\mid  X_{\mathrm{obs}(m)}, M=m\right)-\P\left(Y=-1\mid  X_{\mathrm{obs}(m)}, M=m\right)\right)\\
        &=\mathrm{sign}\left( \frac{\P\left(Y=1,X_{\mathrm{obs}(m)}, M=m\right)}{\P(X_{\mathrm{obs}(m)}, M=m)}-\frac{\P\left(Y=-1, X_{\mathrm{obs}(m)}, M=m\right)}{\P(X_{\mathrm{obs}(m)}, M=m)}\right)\\
        &=\mathrm{sign}\left( \P\left(X_{\mathrm{obs}(m)}\mid  M=m,Y=1\right)\pi_{m,1}-\P\left( X_{\mathrm{obs}(m)}\mid  M=m, Y=-1\right)\pi_{m,-1}\right),
    \end{align*}
    with $\pi_{m,k}=\P\left( M=m, Y=k\right).$ Thus, our objective is to study when
    \begin{align*}
        \log\left(\frac{\P\left(X_{\mathrm{obs}(m)}\mid  M=m,Y=1\right)}{\P\left( X_{\mathrm{obs}(m)}\mid  M=m, Y=-1\right)}\right)>\log\left(\frac{\pi_{m,-1}}{\pi_{m,1}}\right).
    \end{align*}
    Note that by using \cref{ass:gpmmLDA}, we have $X_{\mathrm{obs}(m)}|M=m,Y=k \sim \mathcal{N}(\mu_{m,k},\Sigma_{m})$. Therefore,
    \begin{align*}
         \log & \left(\frac{f_{X_{\mathrm{obs}(m)} | M=m, Y=1}(x)}{f_{X_{\mathrm{obs}(m)} | M=m, Y=-1}(x)}\right) \\&= \log\left( \frac{(\sqrt{2\pi})^{-(d-\left\|m\right\|_0)}\sqrt{\det(\Sigma^{-1}_{m})}\exp\left(-\frac{1}{2}(x-\mu_{1, m})^{\top}\Sigma^{-1}_{m}(x-\mu_{1,m})\right)}{(\sqrt{2\pi})^{-(d-\left\|m\right\|_0)}\sqrt{\det(\Sigma^{-1}_{m})}\exp\left(-\frac{1}{2}(x-\mu_{-1, m})^{\top}\Sigma^{-1}_{m}(x-\mu_{-1,m})\right)} \right)\\
        &=-\frac{1}{2}(x-\mu_{1, m})^{\top}\Sigma^{-1}_{m}(x-\mu_{1,m})+\frac{1}{2}(x-\mu_{-1, m})^{\top}\Sigma^{-1}_{m}(x-\mu_{-1,m})\\
        &=(\mu_{1,m}-\mu_{-1,m})^{\top}\Sigma^{-1}_{m}\left(x-\frac{\mu_{1,m}+\mu_{-1,m}}{2}\right). 
    \end{align*}
    Consequently, 
     \begin{align}
        h_m^{\star}(x) 
        & = \textrm{sign}\left( (\mu_{1,m}-\mu_{-1,m})^{\top}\Sigma^{-1}_{m}\left(x-\frac{\mu_{1,m}+\mu_{-1,m}}{2}\right) - \log\left(\frac{\pi_{m,-1}}{\pi_{m,1}}\right) \right), 
    \end{align}
    which concludes the proof. 
    
\end{proof}

\subsection{General lemmas for LDA misclassification control under Assumption \ref{ass:gpmmLDA}. }\label{subsec:generalLemmasBoundEstimGPMM}
\begin{lemma}[$\widehat{\mu}_m$ misclassification probability]\label{lemma:probMisMuEstimGPMM}
        Grant \Cref{ass:gpmmLDA}. Then, 
    \begin{align}
     & \P\left(h_m^{\star}(X_{\mathrm{obs}(m)})=1\mid Y=-1, M=m\right) = \Phi\left(-\frac{1}{2}\left\|\Sigma_{m}^{-\frac{1}{2}}(\mu_{m,1}-\mu_{m,-1})\right\| \right),
        \end{align}
        and
        \begin{align}
        \P\left(h_m^{\star}(X_{\mathrm{obs}(m)})=1\mid Y=-1, \mathcal{D}_n\right) &= \Phi\left(\frac{\left(\Sigma_{m}^{-\frac{1}{2}}(\widehat{\mu}_{m,1}-\widehat{\mu}_{m,-1})\right)^\top\Sigma_{m}^{-\frac{1}{2}}\left(\mu_{m,-1}-\frac{\widehat{\mu}_{m,1}+\widehat{\mu}_{m,-1}}{2}\right)}{\left\|\Sigma_{m}^{-\frac{1}{2}}(\widehat{\mu}_{m,1}-\widehat{\mu}_{m,-1})\right\|}\right) \label{eq:probMissLDAEstimGPMM}
    \end{align}
    Symmetrically, 
    \begin{align}
        \P\left(h_m^{\star}(X_{\mathrm{obs}(m)})=-1\mid Y= 1, M=m\right)
        & = \Phi\left(-\frac{1}{2}\left\|\Sigma_{m}^{-\frac{1}{2}}(\mu_{m,1}-\mu_{m,-1})\right\| \right),
    \end{align}
    and
    \begin{align}
        \notag 
        \P&\left(\widehat{h}_m(X_{\mathrm{obs}(m)})=-1\mid Y=1, M=m, \mathcal{D}_n\right) \\
        &= \Phi\left(-\frac{\left(\Sigma_{m}^{-\frac{1}{2}}(\widehat{\mu}_{m,1}-\widehat{\mu}_{m,-1})\right)^\top\Sigma_{m}^{-\frac{1}{2}}\left(\mu_{m,1}-\frac{\widehat{\mu}_{m,1}+\widehat{\mu}_{m,-1}}{2}\right)}{\left\|\Sigma_{m}^{-\frac{1}{2}}(\widehat{\mu}_{m,1}-\widehat{\mu}_{m,-1})\right\|}\right)
        \label{eq:probMissLDAEstim2GPMM}
    \end{align}
    with $\Phi$ the c.d.f. of a standard Gaussian distribution.
\end{lemma}
\begin{proof}
Using Proposition \ref{prop:MNARLDA}, and recalling that the classes are balanced on each missing patterns ($\pi_{m,1} = \pi_{m,-1}$),
   \begin{align*}
   & \P\left(h_m^{\star}(X_{\mathrm{obs}(m)})=1\mid Y=-1, M=m\right)\\
   & = \P\left( \left(\mu_{m,1}-\mu_{m,-1}\right)^{\top}\Sigma_{m}^{-1}\left(X_{\mathrm{obs}(m)}-\frac{\mu_{m,1}+\mu_{m,-1}}{2}\right) >0\mid Y=-1, M=m\right).
    \end{align*}
   Let $N=\Sigma_{m}^{-\frac{1}{2}}(X_{\mathrm{obs}(m)}-\mu_{m,-1})$. By \Cref{ass:gpmmLDA}, 
   \begin{align}
   N|Y=-1, M=m\sim\mathcal{N}(0, Id_{d-\left\|m\right\|_0}).    
   \end{align}
   Letting $\gamma=\Sigma_{m}^{-\frac{1}{2}}(\mu_{m,1}-\mu_{m,-1})$,  we have
   \begin{align*}
        \P\left(h_m^{\star}(X_{\mathrm{obs}(m)})=1\mid Y=-1, M=m\right)
       &=\P\left(\gamma^{\top}N-\frac{1}{2}\left\|\gamma\right\|^2> 0\mid Y=-1, M=m\right)\\
       &=\P\left(\frac{\gamma^{\top}N}{\left\|\gamma\right\|}>\frac{1}{2}\left\|\gamma\right\| \mid Y=-1, M=m\right)\\
       &=\Phi\left(-\frac{1}{2}\left\|\gamma\right\| \right).
   \end{align*} 
Similarly, using Proposition \ref{prop:MNARLDA},  
   \begin{align*}
   & \P\left(h_m^{\star}(X_{\mathrm{obs}(m)})=-1\mid Y=1, M=m\right)\\
   & = \P\left( \left(\mu_{m,1}-\mu_{m,-1}\right)^{\top}\Sigma_{m}^{-1}\left(X_{\mathrm{obs}(m)}-\frac{\mu_{m,1}+\mu_{m,-1}}{2}\right) < 0\mid Y=1, M=m\right).
    \end{align*}
   Let $N=\Sigma_{m}^{-\frac{1}{2}}(X_{\mathrm{obs}(m)}-\mu_{m,1})$. By \Cref{ass:gpmmLDA}, 
   \begin{align}
    N|Y=1, M=m\sim\mathcal{N}(0, Id_{d-\left\|m\right\|_0}).
   \end{align}
   Letting $\gamma=\Sigma_{m}^{-\frac{1}{2}}(\mu_{m,1}-\mu_{m,-1})$, we have
   \begin{align*}
        \P\left(h_m^{\star}(X_{\mathrm{obs}(m)})=-1\mid Y=1, M=m\right)
       &=\P\left(\gamma^{\top}N + \frac{1}{2}\left\|\gamma\right\|^2 < 0\mid Y=1, M=m\right)\\
       &=\P\left(\frac{\gamma^{\top}N}{\left\|\gamma\right\|} < - \frac{1}{2}\left\|\gamma\right\| \mid Y=1, M=m\right)\\
       &=\Phi\left(-\frac{1}{2}\left\|\gamma\right\| \right).
   \end{align*} 
   This proves the first and third statements. Regarding the second and fourth statements, following the same strategy as in the proof of \cref{cor:GenBayesRiskpbpLDA}, we have
   \begin{align*}
   & \P\left(\widetilde{h}_m(X_{\mathrm{obs}(m)})=1\mid Y=-1, \mathcal{D}_n\right) \\
    &= \P\left( \left(\widetilde{\mu}_{1, m}-\widetilde{\mu}_{-1, m}\right)^{\top}\Sigma_{m}^{-1}\left(X_{\mathrm{obs}(m)}-\frac{\widetilde{\mu}_{1, m}+\widetilde{\mu}_{-1,m}}{2}\right)>0\mid  Y=-1, \mathcal{D}_n\right)
    \end{align*}
   Let $N=\Sigma_{m}^{-\frac{1}{2}}(X_{\mathrm{obs}(m)}-\mu_{-1, m})$. By \Cref{lemma:obsGaus},  $N|Y=-1\sim\mathcal{N}(0, Id_{d-\left\|m\right\|_0})$. Since $(X_{\mathrm{obs}(m)},Y)$ and  $\mathcal{D}_n$ are independent
   \begin{align}
    N|Y=-1, \mathcal{D}_n \sim\mathcal{N}(0, Id_{d-\left\|m\right\|_0}).
   \end{align}
   Letting $\widetilde{\gamma}=\Sigma_{m}^{-\frac{1}{2}}(\widetilde{\mu}_{1,m}-\widetilde{\mu}_{-1,m})$, we have
   \begin{align*}
       & \P\left(h_m^{\star}(X_{\mathrm{obs}(m)})=1\mid Y=-1, \mathcal{D}_n\right)\\
       &=\P\left(\widetilde{\gamma}^{\top}N+\widetilde{\gamma}^\top\Sigma_{m}^{-\frac{1}{2}}\left(\mu_{-1,m}-\frac{\widetilde{\mu}_{1, m}+\widetilde{\mu}_{-1,m}}{2}\right)>0\mid Y=-1, \mathcal{D}_n\right)\\
       &=\P\left(\frac{\widetilde{\gamma}^{\top}N}{\left\|\widetilde{\gamma}\right\|}>-\frac{\widetilde{\gamma}^\top}{\left\|\widetilde{\gamma}\right\|}\Sigma_{m}^{-\frac{1}{2}}\left(\mu_{-1,m}-\frac{\widetilde{\mu}_{1, m}+\widetilde{\mu}_{-1,m}}{2}\right)\mid Y=-1, \mathcal{D}_n\right)\\
       &=\Phi\left(\frac{\widetilde{\gamma}^\top}{\left\|\widetilde{\gamma}\right\|}\Sigma_{m}^{-\frac{1}{2}}\left(\mu_{-1,m}-\frac{\widetilde{\mu}_{1, m}+\widetilde{\mu}_{-1,m}}{2}\right)\right).
   \end{align*} 
Regarding the fourth statement, the proof is similar. Indeed, 
   \begin{align*}
        & \P\left(\widetilde{h}_m(X_{\mathrm{obs}(m)})=-1\mid Y=1, \mathcal{D}_n\right) \\
         & = \P\left( \left(\widetilde{\mu}_{1, m}-\widetilde{\mu}_{-1, m}\right)^{\top}\Sigma_{m}^{-1}\left(X_{\mathrm{obs}(m)}-\frac{\widetilde{\mu}_{1, m}+\widetilde{\mu}_{-1,m}}{2}\right)<0\mid  Y=1\right).
    \end{align*}
   Let $N=\Sigma_{m}^{-\frac{1}{2}}(X_{\mathrm{obs}(m)}-\mu_{1, m})$. By \Cref{lemma:obsGaus}, and since $(X_{\mathrm{obs}(m)},Y)$ and  $\mathcal{D}_n$ are independent, 
   \begin{align}
   N|Y=1, \mathcal{D}_n \sim\mathcal{N}(0, Id_{d-\left\|m\right\|_0}).    
   \end{align}
   Letting $\widetilde{\gamma}=\Sigma_{m}^{-\frac{1}{2}}(\widetilde{\mu}_{1,m}-\widetilde{\mu}_{-1,m})$, we have
   \begin{align*}
       & \P\left(\widetilde{h}_m(X_{\mathrm{obs}(m)})=-1\mid Y=1, \mathcal{D}_n\right)\\
       &=\P\left(\widetilde{\gamma}^{\top}N+\widetilde{\gamma}^\top\Sigma_{m}^{-\frac{1}{2}}\left(\mu_{1,m}-\frac{\widetilde{\mu}_{1, m}+\widetilde{\mu}_{-1,m}}{2}\right)<0\mid Y=1, \mathcal{D}_n\right)\\
       &=\P\left(\frac{\widetilde{\gamma}^{\top}N}{\left\|\widetilde{\gamma}\right\|}<-\frac{\widetilde{\gamma}^\top}{\left\|\widetilde{\gamma}\right\|}\Sigma_{m}^{-\frac{1}{2}}\left(\mu_{1,m}-\frac{\widetilde{\mu}_{1, m}+\widetilde{\mu}_{-1,m}}{2}\right)\mid Y=1, \mathcal{D}_n\right)\\
       &=\Phi\left(-\frac{\widetilde{\gamma}^\top}{\left\|\widetilde{\gamma}\right\|}\Sigma_{m}^{-\frac{1}{2}}\left(\mu_{1,m}-\frac{\widetilde{\mu}_{1, m}+\widetilde{\mu}_{-1,m}}{2}\right)\right).
   \end{align*} 
\end{proof}

\begin{lemma}\label{lemma:boundMissProbMuEstimGPMM}
Grant Assumption \ref{ass:gpmmLDA}. Assume that we are given two estimates $\widetilde{\mu}_1$ and $\widetilde{\mu}_{-1}$. Then, for all $m \in \mathcal{M}$, the classifier $\widetilde{h}_m$ defined in Equation \eqref{eq:PbPLDAEstimMuAllPatterns_tildebis} satisfies 
    \begin{align}
    \notag
        \bigl| 
    \P&\left(\widetilde{h}_m(X_{\mathrm{obs}(m)})=1\mid  Y=-1,M=m, \mathcal{D}_n\right)-\P\left(h_m^\star(X_{\mathrm{obs}(m)})=1 \mid Y=-1, M=m\right) 
    \bigr|\\
    &\leq \frac{3}{2\sqrt{2\pi}}\left\|\Sigma_{m}^{-\frac{1}{2}}(\widetilde{\mu}_{m,-1} - \mu_{m,-1})\right\|+
    \frac{1}{2\sqrt{2\pi}}\left\|\Sigma_{m}^{-\frac{1}{2}}(\widetilde{\mu}_{m,1} - \mu_{m,1})\right\|\label{eq:boundMissProbMuEstim1GPMM}
    \end{align}
    and symmetrically, 
    \begin{align}\notag
        \bigl| 
    \P&\left(\widetilde{h}_m(X_{\mathrm{obs}(m)})=-1\mid  Y=1, M=m, \mathcal{D}_n\right)-\P\left(h_m^\star(X_{\mathrm{obs}(m)})=-1 \mid Y=1, M=m\right) 
    \bigr|\\
    &\leq\frac{3}{2\sqrt{2\pi}}\left\|\Sigma_{m}^{-\frac{1}{2}}(\widetilde{\mu}_{m,1}
    - \mu_{m,1})\right\|+\frac{1}{2\sqrt{2\pi}}\left\|\Sigma_{m}^{-\frac{1}{2}}(\mu_{m,-1} - \widetilde{\mu}_{m,-1})\right\|\label{eq:boundMissProbMuEstim2GPMM}
    \end{align}
\end{lemma}
\begin{proof}
To prove Inequality \eqref{eq:boundMissProbMuEstim1GPMM}, notice that, by \Cref{lemma:probMisMuEstimGPMM},  
\begin{align*}
    \bigl| 
    \P&\left(\widetilde{h}_m(X_{\mathrm{obs}(m)})=1\mid  Y=-1, M=m, \mathcal{D}_n\right)-\P\left(h_m^\star(X_{\mathrm{obs}(m)})=1 \mid Y=-1, M=m\right) 
    \bigr|\\
    &=\left|\Phi\left(\frac{\left(\Sigma_{m}^{-\frac{1}{2}}(\widetilde{\mu}_{m,1}-\widetilde{\mu}_{m,-1})\right)^\top\Sigma_{m}^{-\frac{1}{2}}\left(\mu_{m,-1}-\frac{\widetilde{\mu}_{m,1}+\widetilde{\mu}_{m,-1}}{2}\right)}{\left\|\Sigma_{m}^{-\frac{1}{2}}(\widetilde{\mu}_{m,1)}-\widetilde{\mu}_{m,-1})\right\|}\right)\right.\\&\qquad\qquad\left.-\Phi\left(-\frac{\left\|\Sigma_{m}^{-\frac{1}{2}}(\mu_{m,1}-\mu_{m,-1})\right\|}{2}\right)\right|.
     \end{align*} 
We can then apply the same steps as in the  proof of \Cref{lemma:boundMissProbMuEstim}, and the result follows. The proof of   Inequality \eqref{eq:boundMissProbMuEstim2GPMM} is similar.
     \end{proof}

\begin{lemma}\label{lemma:generalDecompLestimGPMM}
Grant Assumption \ref{ass:gpmmLDA}, with balanced classes. Assume that we are given two estimates $\widetilde{\mu}_1$ and $\widetilde{\mu}_{-1}$. Then, for all $m \in \mathcal{M}$, the classifier $\widetilde{h}_m$ defined in Equation \eqref{eq:PbPLDAEstimMuAllPatterns} satisfies 
    \begin{align*}
         \mathcal{R}_{\mathrm{mis}}(\widetilde{h}) &-\mathcal{R}_{\mathrm{mis}}(h^\star)\\
    &\leq\sum_{m\in\M}\frac{1}{\sqrt{2\pi}}\left(\e{\left\|\Sigma_{m}^{-\frac{1}{2}}(-\widetilde{\mu}_{m,1}
    +\mu_{m,1})\right\|+\left\|\Sigma_{m}^{-\frac{1}{2}}(\widetilde{\mu}_{m,-1}-\mu_{m,-1})\right\|}\right)p_m.
    \end{align*}
\end{lemma}
\begin{proof}
    \begin{align*}
    &     \mathcal{R}_{\mathrm{mis}}(\widetilde{h}) -\mathcal{R}_{\mathrm{mis}}(h^\star)\\
    &=\P\left(\widetilde{h}(X_{\mathrm{obs}(M)},M)\neq Y\right)-\P\left(h^\star(X_{\mathrm{obs}(M)},M)\neq Y\right)\\ 
    &=\sum_{m\in\M}\left(\P\left(\widetilde{h}(X_{\mathrm{obs}(M)},M)\neq Y\mid  M=m\right)-\P\left(h^\star(X_{\mathrm{obs}(M)},M)\neq Y\mid M=m\right)\right)p_m\\
    &=\sum_{m\in\M}\left(\P\left(\widetilde{h}_m(X_{\mathrm{obs}(m)})\neq Y\mid  M=m\right)-\P\left(h_m^\star(X_{\mathrm{obs}(m)})\neq Y\mid M=m\right)\right)p_m\tag{using \eqref{eq:PbPLDAEstimMuAllPatterns}}\\ 
    &=\sum_{m\in\M}\pi_{m,-1}\left(\P\left(\widetilde{h}_m(X_{\mathrm{obs}(m)})=1\mid  Y=-1, M=m\right)-\P\left(h_m^\star(X_{\mathrm{obs}(m)})=1 \mid Y=-1, M=m\right)\right)\\ 
    &\quad+\sum_{m\in\M}\pi_{m,1}\left(\P\left(\widetilde{h}_m(X_{\mathrm{obs}(m)})=-1\mid  Y=1, M=m\right)-\P\left(h_m^\star(X_{\mathrm{obs}(m)})=-1 \mid Y=1, M=m\right)\right).
    \end{align*}
Note that 
\begin{align}
 & \P\left(\widetilde{h}_m(X_{\mathrm{obs}(m)})=1\mid  Y=-1, M=m\right)-\P\left(h_m^\star(X_{\mathrm{obs}(m)})=1 \mid Y=-1, M=m\right) \\
 &= \e{\P\left(\widetilde{h}_m(X_{\mathrm{obs}(m)})=1\mid  Y=-1, M=m, \mathcal{D}_n\right)-\P\left(h_m^\star(X_{\mathrm{obs}(m)})=1 \mid Y=-1, M=m\right)}\\
 & \leq \frac{1}{2 \sqrt{2\pi}}\e{3\left\|\Sigma_{m}^{-\frac{1}{2}}(\mu_{m,-1}-\widetilde{\mu}_{m,-1})\right\|+
    \left\|\Sigma_{m}^{-\frac{1}{2}}(\mu_{m,1}-\widetilde{\mu}_{m,1})\right\|},
\end{align}
according to \Cref{lemma:boundMissProbMuEstimGPMM}. Similarly, 
\begin{align}
    & \P\left(\widetilde{h}_m(X_{\mathrm{obs}(m)})=-1\mid  Y=1, M=m\right)-\P\left(h_m^\star(X_{\mathrm{obs}(m)})=-1 \mid Y=1, M=m\right)\\
    & = \e{\P\left(\widetilde{h}_m(X_{\mathrm{obs}(m)})=-1\mid  Y=1, M=m,\mathcal{D}_n\right)-\P\left(h_m^\star(X_{\mathrm{obs}(m)})=-1 \mid Y=1, M=m\right)}\\
    & \leq \frac{\pi_{m,-1}}{2\sqrt{2\pi}}\left(\e{3\left\|\Sigma_{m}^{-\frac{1}{2}}(-\widetilde{\mu}_{m,1}
    +\mu_{m,1})\right\| 
    +\left\|\Sigma_{m}^{-\frac{1}{2}}(\widetilde{\mu}_{m,-1}-\mu_{m,-1})\right\|}\right).
\end{align}
    Consequently, since for all $m \in \mathcal{M}$, $\pi_{1,m} = \pi_{-1,m}$, 
  \begin{align*}
    &     \mathcal{R}_{\mathrm{mis}}(\widetilde{h}) -\mathcal{R}_{\mathrm{mis}}(h^\star)\\
    & \leq \sum_{m\in\M}\frac{1}{\sqrt{2\pi}}\left(\e{\left\|\Sigma_{m}^{-\frac{1}{2}}(-\widetilde{\mu}_{m,1}
    +\mu_{m,1})\right\|+\left\|\Sigma_{m}^{-\frac{1}{2}}(\widetilde{\mu}_{m,-1}-\mu_{m,-1})\right\|}\right)p_m.
    \end{align*}
\end{proof}

\subsection{Lemmas for Theorem \ref{th:MNAREstim}}

\begin{lemma}\label{lemma:MNARestim1}
    Grant \Cref{ass:gpmmLDA}. Then, for all $k \in \{-1,1\}$, 
    \begin{align*}
        \e{(\widetilde{\mu}_{m,k}-\mu_{m,k})(\widetilde{\mu}_{m,k}-\mu_{m,k})^\top}= \e{\ind_{\frac{N_{m,k}}{n}>\tau}\frac{1}{N_{m,k}}}\Sigma_m+\P\left(\frac{N_{m,k}}{n}\leq\tau\right)\mu_{m,k}\mu_{m,k}^\top
    \end{align*}
    where $\widetilde{\mu}_{m,k}$ is the estimate defined at \eqref{eq:estimMuGPMM2}.
\end{lemma}
\begin{proof}
    We have
    \begin{align*}
        \mathbb{E}&\left[(\widetilde{\mu}_{m,k}-\mu_{m,k})(\widetilde{\mu}_{m,k}-\mu_{m,k})^\top\right] \\
        &=\mathbb{E}\left[(\widehat{\mu}_{m,k}\ind_{\frac{N_{m,k}}{n}>\tau}-\mu_{m,k}\ind_{\frac{N_{m,k}}{n}>\tau}+\mu_{m,k}\ind_{\frac{N_{m,k}}{n}>\tau}-\mu_{m,k})\right.\\&\qquad \left.(\widehat{\mu}_{m,k}\ind_{\frac{N_{m,k}}{n}>\tau}-\mu_{m,k}\ind_{\frac{N_{m,k}}{n}>\tau}+\mu_{m,k}\ind_{\frac{N_{m,k}}{n}>\tau}-\mu_{m,k})^\top\right]\\
        &=\mathbb{E}\left[\ind_{\frac{N_{m,k}}{n}>\tau}(\widehat{\mu}_{m,k}-\mu_{m,k})(\widehat{\mu}_{m,k}-\mu_{m,k})^\top +\ind_{\frac{N_{m,k}}{n}>\tau}\left(\ind_{\frac{N_{m,k}}{n}>\tau}-1\right)(\widehat{\mu}_{m,k}-\mu_{m,k})\mu_{m,k}^\top\right.\\&\qquad+\left.\ind_{\frac{N_{m,k}}{n}>\tau}\left(\ind_{\frac{N_{m,k}}{n}>\tau}-1\right)\mu_{m,k}(\widehat{\mu}_{m,k}-\mu_{m,k})^\top+\left(\ind_{\frac{N_{m,k}}{n}>\tau}-1\right)^2\mu_{m,k}\mu_{m,k}^\top\right].
    \end{align*}
   Since $\ind_{\frac{N_{m,k}}{n}>\tau}\left(\ind_{\frac{N_{m,k}}{n}>\tau}-1\right)=0,$ we obtain 
    \begin{align*}
         \mathbb{E}&\left[\ind_{\frac{N_{m,k}}{n}>\tau}(\widehat{\mu}_{m,k}-\mu_{m,k})(\widehat{\mu}_{m,k}-\mu_{m,k})^\top +\left(1-\ind_{\frac{N_{m,k}}{n}>\tau}\right)\mu_{m,k}\mu_{m,k}^\top\right]\\
         &=\mathbb{E}\left[\ind_{\frac{N_{m,k}}{n}>\tau}(\widehat{\mu}_{m,k}-\mu_{m,k})(\widehat{\mu}_{m,k}-\mu_{m,k})^\top \right]+\P\left(\frac{N_{m,k}}{n}\leq\tau\right)\mu_{m,k}\mu_{m,k}^\top.
    \end{align*}
    Finally, remark that $\widehat{\mu}_{m,k}-\mu_{m,k}|N_{m,k}\sim \mathcal{N}(0, \Sigma_m/N_{m,k})$. Thus, we conclude, noticing that
    \begin{align*}
        & \mathbb{E}\left[\ind_{\frac{N_{m,k}}{n}>\tau}(\widehat{\mu}_{m,k}-\mu_{m,k})(\widehat{\mu}_{m,k}-\mu_{m,k})^\top \right]\\
        &=\e{\mathbb{E}\left[\ind_{\frac{N_{m,k}}{n}>\tau}(\widehat{\mu}_{m,k}-\mu_{m,k})(\widehat{\mu}_{m,k}-\mu_{m,k})^\top \mid N_{m,k}\right]}\\
        &=\e{\ind_{\frac{N_{m,k}}{n} >\tau}\mathbb{E}\left[(\widehat{\mu}_{m,k}-\mu_{m,k})(\widehat{\mu}_{m,k}-\mu_{m,k})^\top \mid N_{m,k}\right]}\\
        &=\e{\frac{\ind_{\frac{N_{m,k}}{n}>\tau}}{N_{m,k}}}\Sigma_m.
    \end{align*}
\end{proof}
\begin{lemma}\label{lemma:MNARestim2}
    Grant \Cref{ass:gpmmLDA}. Then, for all $k \in \{-1,1\}$, 
    \begin{align*}
        \e{\left\|\Sigma^{-\frac{1}{2}}_m(\widetilde{\mu}_{m,k}-\mu_{m,k})\right\|}\leq\left(\e{\ind_{\frac{N_{m,k}}{n}>\tau}\frac{1}{N_{m,k}}}(d-\left\|m\right\|_0)+\P\left(\frac{N_{m,k}}{n}\leq\tau\right)\left\|\Sigma^{-\frac{1}{2}}_m\mu_{m,k}\right\|^2\right)^\frac{1}{2}, 
    \end{align*}
    where $\widetilde{\mu}_{m,k}$ is the estimate defined in \eqref{eq:estimMuGPMM2}.
\end{lemma}
\begin{proof}
By Jensen's inequality, 
    \begin{align*}
        \mathbb{E}&\left[\left\|\Sigma^{-\frac{1}{2}}_m(\widetilde{\mu}_{m,k}-\mu_{m,k})\right\|\right]\\
        &\leq \e{\left\|\Sigma^{-\frac{1}{2}}_m(\widetilde{\mu}_{m,k}-\mu_{m,k})\right\|^2}^\frac{1}{2}\\
        &= \e{\mathrm{tr}\left(\left\|\Sigma^{-\frac{1}{2}}_m(\widetilde{\mu}_{m,k}-\mu_{m,k})\right\|^2\right)}^\frac{1}{2}\\
         &= \e{\mathrm{tr}\left(\left(\Sigma^{-\frac{1}{2}}_m(\widetilde{\mu}_{m,k}-\mu_{m,k})\right)^\top\Sigma^{-\frac{1}{2}}_m(\widetilde{\mu}_{m,k}-\mu_{m,k})\right)}^\frac{1}{2}\\
         &= \e{\mathrm{tr}\left(\Sigma^{-\frac{1}{2}}_m(\widetilde{\mu}_{m,k}-\mu_{m,k})\left(\Sigma^{-\frac{1}{2}}_m(\widetilde{\mu}_{m,k}-\mu_{m,k})\right)^\top\right)}^\frac{1}{2}\\
         &= \e{\mathrm{tr}\left(\Sigma^{-\frac{1}{2}}_m(\widetilde{\mu}_{m,k}-\mu_{m,k})(\widetilde{\mu}_{m,k}-\mu_{m,k})^\top\Sigma^{-\frac{1}{2}}_m\right)}^\frac{1}{2}\\
         &= \mathrm{tr}\left(\Sigma^{-\frac{1}{2}}_m\e{(\widetilde{\mu}_{m,k}-\mu_{m,k})(\widetilde{\mu}_{m,k}-\mu_{m,k})^\top}\Sigma^{-\frac{1}{2}}_m\right)^\frac{1}{2}\\
         &=\mathrm{tr}\left(\Sigma^{-\frac{1}{2}}_m\left(\e{\ind_{\frac{N_{m,k}}{n}>\tau}\frac{1}{N_{m,k}}}\Sigma_m+\P\left(\frac{N_{m,k}}{n}\leq\tau\right)\mu_{m,k}\mu_{m,k}^\top\right)\Sigma^{-\frac{1}{2}}_m\right)^\frac{1}{2}\tag{using \cref{lemma:MNARestim1}}\\
         &=\left(\e{\ind_{\frac{N_{m,k}}{n}>\tau}\frac{1}{N_{m,k}}}(d-\left\|m\right\|_0)+\P\left(\frac{N_{m,k}}{n}\leq\tau\right)\mathrm{tr}\left(\Sigma^{-\frac{1}{2}}_m\mu_{m,k}\mu_{m,k}^\top\Sigma^{-\frac{1}{2}}_m\right)
         \right)^\frac{1}{2}\\
         &=\left(\e{\ind_{\frac{N_{m,k}}{n}>\tau}\frac{1}{N_{m,k}}}(d-\left\|m\right\|_0)+\P\left(\frac{N_{m,k}}{n}\leq\tau\right)\left\|\Sigma^{-\frac{1}{2}}_m\mu_{m,k}\right\|^2\right)^\frac{1}{2}.
    \end{align*}
    
\end{proof}

\subsection{Proof of Theorem \ref{th:MNAREstim}}\label{subsec:proofMNAREstim}
\begin{proof}
Let $A_{\tau}:=\{m\in\{0,1\}^d|p_m<\tau\}$ be the set of missing pattern with occurrence probability smaller than $\tau$. According to \Cref{lemma:generalDecompLestimGPMM} and \Cref{lemma:MNARestim2}, we have
    \begin{align*}
      &   \mathcal{R}_{\mathrm{mis}}(\widetilde{h}) -\mathcal{R}_{\mathrm{mis}}(h^\star)\\
    &\leq\sum_{m\in\M}\frac{1}{\sqrt{2\pi}}\left(\e{\left\|\Sigma_{m}^{-\frac{1}{2}}(-\widetilde{\mu}_{m,1}
    +\mu_{m,1})\right\|+\left\|\Sigma_{m}^{-\frac{1}{2}}(\widetilde{\mu}_{m,-1}-\mu_{m,-1})\right\|}\right)p_m 
    \\&\leq \sum_{m\in\M}\sum_{k=\pm 1}\frac{1}{\sqrt{2\pi}}\left(\e{\ind_{\frac{N_{m,k}}{n}>\tau}\frac{1}{N_{m,k}}}(d-\left\|m\right\|_0)+\P\left(\frac{N_{m,k}}{n}\leq\tau\right)\left\|\Sigma^{-\frac{1}{2}}_m\mu_{m,k}\right\|^2\right)^\frac{1}{2}p_m \\
    &= \sum_{m\in A_\tau}\sum_{k=\pm 1}\frac{1}{\sqrt{2\pi}}\left(\e{\ind_{\frac{N_{m,k}}{n}>\tau}\frac{1}{N_{m,k}}}(d-\left\|m\right\|_0)+\P\left(\frac{N_{m,k}}{n}\leq\tau\right)\left\|\Sigma^{-\frac{1}{2}}_m\mu_{m,k}\right\|^2\right)^\frac{1}{2}p_m\ind_{p_m<\tau}\\
    &\quad+\sum_{m\notin A_{\tau}}\sum_{k=\pm 1}\frac{1}{\sqrt{2\pi}}\left(\e{\ind_{\frac{N_{m,k}}{n}>\tau}\frac{1}{N_{m,k}}}(d-\left\|m\right\|_0)+\P\left(\frac{N_{m,k}}{n}\leq\tau\right)\left\|\Sigma^{-\frac{1}{2}}_m\mu_{m,k}\right\|^2\right)^\frac{1}{2}p_m\ind_{p_m\geq\tau}. 
    \end{align*}
    Now, for all $m\in A_\tau$, recalling that $\tau \geq \sqrt{d/n}$,   
        \begin{align}
            \sum_{k=\pm 1}&\frac{1}{\sqrt{2\pi}}\left(\e{\ind_{\frac{N_{m,k}}{n}>\tau}\frac{1}{N_{m,k}}}(d-\left\|m\right\|_0)+\P\left(\frac{N_{m,k}}{n}\leq\tau\right)\left\|\Sigma^{-\frac{1}{2}}_m\mu_{m,k}\right\|^2\right)^\frac{1}{2}p_m\ind_{p_m<\tau}\\
            &\leq \sum_{k=\pm 1}\frac{1}{\sqrt{2\pi}}\left(\e{\ind_{\frac{N_{m,k}}{n}>\tau}\frac{1}{n\tau}}(d-\left\|m\right\|_0)+\left\|\Sigma^{-\frac{1}{2}}_m\mu_{m,k}\right\|^2\right)^\frac{1}{2}p_m\ind_{p_m<\tau}\\
            &\leq \sum_{k=\pm 1}\frac{1}{\sqrt{2\pi}}\left(\e{\ind_{\frac{N_{m,k}}{n}>\tau}}\tau+\left\|\Sigma^{-\frac{1}{2}}_m\mu_{m,k}\right\|^2\right)^\frac{1}{2}p_m\ind_{p_m<\tau}\\
            &\leq \frac{2}{\sqrt{2\pi}}\left(1+\frac{\left\|\mu_{m}\right\|^2}{\lambda_{\min}(\Sigma_m)}\right)^\frac{1}{2}p_m\ind_{p_m<\tau}. \label{eq_proof36}
        \end{align}
    On the other hand, for all $m\notin A_\tau$,
         \begin{align*}
        \P\left(\frac{N_{m,k}}{n}\leq\tau\right)&=\P\left(N_{m,k}\leq n\tau \right)\\
        &=\P\left(\frac{\ind_{N_{m,k}>0}}{N_{m,k}^2}\geq \frac{1}{n^2\tau^2}\right) + \P\left(  N_{m,k} = 0 \right)\\
        &\leq  \frac{32n^2\tau^2}{ p_m^2 (n+1)(n+2)} + \left( 1 - p_m \right)^n  \\
        &\leq \frac{32 \tau^2}{p_m^2} + \left( 1 - p_m \right)^n,
    \end{align*}
    using Markov Inequality and Inequality \eqref{eq:inv_bino_square2}.     Then, for all $m\notin A_\tau$, we have 
        \begin{align}
            & \sum_{k=\pm 1}\frac{1}{\sqrt{2\pi}}\left(\e{\ind_{\frac{N_{m,k}}{n}>\tau}\frac{1}{N_{m,k}}}(d-\left\|m\right\|_0)+\P\left(\frac{N_{m,k}}{n}\leq\tau\right)\left\|\Sigma^{-\frac{1}{2}}_m\mu_{m,k}\right\|^2\right)^\frac{1}{2}p_m\ind_{p_m\geq\tau}\\
            &\leq \frac{p_m\ind_{p_m\geq\tau}}{\sqrt{2\pi}} \sum_{k=\pm 1} \Bigg[ \left(\e{\frac{\ind_{\frac{N_{m,k}}{n}>\tau}}{N_{m,k}}}(d-\left\|m\right\|_0)\right)^\frac{1}{2} + \left(\P\left(\frac{N_{m,k}}{n}\leq\tau\right)\left\|\Sigma^{-\frac{1}{2}}_m\mu_{m,k}\right\|^2\right)^\frac{1}{2}\Bigg] \\
            &\leq \frac{1}{\sqrt{2\pi}} \sum_{k=\pm 1} \left[ \left(\frac{4(d-\left\|m\right\|_0)}{p_m(n+1)}\right)^\frac{1}{2}  + \left( \frac{32 \tau^2}{p_m^2} +  \left( 1 - p_m \right)^n  \right)^{1/2} \left\|\Sigma^{-\frac{1}{2}}_m\mu_{m,k}\right\|  \right]p_m\ind_{p_m\geq\tau} \tag{using Inequality \eqref{eq:inverse_binome2}}\\
            &\leq\frac{4 \tau\sqrt{p_m}\ind_{p_m\geq\tau}}{\sqrt{2\pi}}+  \left( \frac{4 \tau  }{\sqrt{\pi}} + \frac{1}{\sqrt{2 \pi}} p_m   (1 - p_m)^{n/2} \right) \ind_{p_m\geq\tau} \sum_{k=\pm 1} \left\|\Sigma^{-\frac{1}{2}}_m\mu_{m,k}\right\|  \\
            &\leq\frac{4 \tau\sqrt{p_m}\ind_{p_m\geq\tau}}{\sqrt{2\pi}} + \left( \frac{4 \tau  }{\sqrt{\pi}} + \frac{1}{\sqrt{2 \pi}} p_m   (1 - p_m)^{n/2} \right) \frac{2 \left\|\mu_{m}\right\|}{\sqrt{ \lambda_{\min}(\Sigma_m)}}\ind_{p_m\geq\tau}. \label{eq_proof37}
        \end{align}

    Combining \eqref{eq_proof36} and \eqref{eq_proof37}, we obtain
    \begin{align*}
         \mathcal{R}_{\mathrm{mis}}(\widetilde{h}) &-\mathcal{R}_{\mathrm{mis}}(h^\star)\\
         &\leq \sum_{m\in \{0,1\}^d}\frac{2}{\sqrt{2\pi}}\left(1+\frac{\left\|\mu_{m}\right\|^2}{\lambda_{\min}(\Sigma_m)}\right)^\frac{1}{2}p_m\ind_{p_m<\tau}
         +\left(\frac{4}{\sqrt{2\pi}}+\frac{8}{\sqrt{\pi}}\frac{\left\|\mu_{m}\right\|}{\sqrt{\lambda_{\min}(\Sigma_m)}}\right)\tau\ind_{p_m\geq\tau} \\
         & \quad +     \frac{\sqrt{2}  \left\|\mu_{m}\right\|}{\sqrt{ \pi \lambda_{\min}(\Sigma_m)}} p_m   (1 - p_m)^{n/2}  \ind_{p_m\geq\tau}.
    \end{align*}
    Since 
    \begin{align*}
        \frac{2}{\sqrt{2\pi}}\left(1+\frac{\left\|\mu_{m}\right\|^2}{\lambda_{\min}(\Sigma_m)}\right)^\frac{1}{2}\leq \frac{2}{\sqrt{2\pi}}+\frac{2}{\sqrt{2\pi}}\frac{\left\|\mu_{m}\right\|}{\sqrt{\lambda_{\min}(\Sigma_m)}}< \frac{4}{\sqrt{2\pi}}+\frac{8}{\sqrt{\pi}}\frac{\left\|\mu_{m}\right\|}{\sqrt{\lambda_{\min}(\Sigma_m)}},
    \end{align*}
    we have
    \begin{align}
       \mathcal{R}_{\mathrm{mis}}(\widetilde{h}) -\mathcal{R}_{\mathrm{mis}}(h^\star)  & \leq \sum_{m\in \{0,1\}^d} \left(  \frac{4}{\sqrt{2\pi}}+\frac{8}{\sqrt{\pi}}\frac{\left\|\mu_{m}\right\|}{\sqrt{\lambda_{\min}(\Sigma_m)}} \right) \tau\wedge p_m \nonumber \\
       & \qquad + \sum_{m\in \{0,1\}^d}  \frac{\sqrt{2}  \left\|\mu_{m}\right\|}{\sqrt{ \pi \lambda_{\min}(\Sigma_m)}} p_m   (1 - p_m)^{n/2}  \ind_{p_m\geq\tau}.
    \end{align}
         
\end{proof}

\section{Technical results}


 \begin{lemma}[Hoeffding's inequality]\label{lemma:Hoeffding}
Consider a sequence $(X_k)_{1\leq k\leq n}$ of independent real-valued random variables satisfying, for two sequences $(a_k)_{1\leq k\leq n}$, $(b_k)_{1\leq k\leq n}$ of real numbers such that $a_k<b_k$ for all $k$,
$$
\forall k,\qquad {\mathbb {P}}(a_{k}\leq X_{k}\leq b_{k})=1.
$$
Let 
$$
S_{n}=\sum_{i=1}^nX_i-\e{X_i}.
$$
Then, for all $\lambda\in\R$,
\begin{align*}
    \e{\exp(\lambda S_n)}\leq \exp\left(\frac{\lambda^2}{8}\sum_{i=1}^n(b_i-a_i)^2\right).
\end{align*}
\end{lemma}

 \begin{lemma}\citep[Lemma A2 p 587]{devroye2013probabilistic}
 \label{lem:inverse_bernoulli}
 Let $B\sim\mathcal{B}(p,n)$, we have 
 \begin{equation}
 \frac{1}{1+np}\leq\esp\left[\frac{1}{1+B}\right]\leq\frac{1}{p(n+1)}\label{eq:inverse_binome1}
 \end{equation}
 and 
 \begin{equation}
 \esp\left[\frac{\ind\{B>0\}}{B}\right]\leq\frac{2}{p(n+1)}.\label{eq:inverse_binome2}
 \end{equation}
 \end{lemma}

 \begin{proof}
 
 \begin{itemize}
 \item To prove the lower bound in \eqref{eq:inverse_binome1}, we use Jensen's inequality as follows:
 \[
 \frac{1}{1+np}=\frac{1}{1+\esp B}\leq\esp\left[\frac{1}{1+B}\right].
 \]
 \item To prove the upper bound in \eqref{eq:inverse_binome1}, note that 
 \begin{align*}
 \esp\left[\frac{1}{1+B}\right] & =\sum_{i=0}^{n}\binom{n}{i}\frac{1}{1+i}p^{i}(1-p)^{n-i}\\
  & =\sum_{i=0}^{n}\frac{n!}{i!(n-i)!(1+i)}p^{i}(1-p)^{n-i}\\
  & =\frac{1}{\left(n+1\right)p}\sum_{i=0}^{n}\frac{(n+1)!}{(i+1)!(n+1-i-1)!}p^{i+1}(1-p)^{n-i}\\
  & =\frac{1}{\left(n+1\right)p}\sum_{i=0}^{n}\binom{n+1}{i+1}p^{i+1}(1-p)^{n+1-i-1}\\
  & \leq\frac{1}{\left(n+1\right)p},
 \end{align*}
 using binomial formula. 
 \item For \eqref{eq:inverse_binome2}, we use $1/x\leq2/(x+1)$ for all $x\geq1$ together with the previous result. 
 \end{itemize}
 \end{proof}

 Following the same idea, we can establish an upper bound on the square in the following lemma.
 \begin{lemma}\label{lemma:BinBound}
     Given an $B\sim\mathcal{B}(n,p)$, we have that
     \begin{align}\label{eq:inv_bino_square1}
         \e{\frac{1}{(1+B)^2}}\leq \frac{2}{(n+1)(n+2)p^2}
     \end{align}
     and 
     \begin{align}\label{eq:inv_bino_square2}
         \e{\frac{\ind_{B>0}}{B^2}}\leq \frac{8}{(n+1)(n+2)p^2}
     \end{align}
 \end{lemma}
 \begin{proof}
 \begin{itemize}
     \item In order to prove \eqref{eq:inv_bino_square1}, note that    
     \begin{align*}
         \e{\frac{1}{(1+B)^2}}&=\sum_{i=0}^n\binom{n}{i}\frac{1}{(1+i)^2}p^i(1-p)^{n-i}\\
         &=\frac{1}{p(n+1)}\sum_{i=0}^n\frac{(n+1)!}{i!(n-i)!}\frac{1}{(1+i)^2}p^{i+1}(1-p)^{n-i}\\
         &=\frac{1}{p(n+1)}\sum_{i=0}^n\frac{(n+1)!}{(i+1)!(n+1-(i+1))!}\frac{1}{(1+i)}p^{i+1}(1-p)^{n+1-(i+1)}\\
         &=\frac{1}{p(n+1)}\sum_{j=1}^{n+1}\frac{(n+1)!}{j!(n+1-j)!}\frac{1}{j}p^{j}(1-p)^{n+1-j}\\
         &=\frac{1}{p(n+1)}\sum_{j=1}^{n+1}\frac{(n+1)!}{j!(n+1-j)!}\frac{1}{j+1}\frac{j+1}{j}p^{j}(1-p)^{n+1-j}
         \\
         &\leq\frac{2}{p(n+1)}\sum_{j=1}^{n+1}\frac{(n+1)!}{j!(n+1-j)!}\frac{1}{j+1}p^{j}(1-p)^{n+1-j}\\
         &=\frac{2}{p^2(n+1)(n+2)}\sum_{j=1}^{n+1}\frac{(n+2)!}{(j+1)!(n+2-(j+1))!}p^{j+1}(1-p)^{n+2-(j+1)}\\
         &=\frac{2}{p^2(n+1)(n+2)}\sum_{k=2}^{n+2}\frac{(n+2)!}{k!(n+2-k)!}p^{k}(1-p)^{n+2-k}\\
         &\leq\frac{2}{p^2(n+1)(n+2)}.
     \end{align*}
    \item Inequality \eqref{eq:inv_bino_square2} can be deduced using the fact that, for all $x\geq1$,   $1/x\leq2/(x+1)$.
     \end{itemize}
 \end{proof}

\begin{lemma}[Diagonal trace inequality]\label{lemma:traceIneq}
    Given a symmetric matrix $A\in \mathcal{M}_{n,n}(\R)$ and a diagonal matrix $B=(b_i)_{i,i}\in \mathcal{M}_{n,n}(\R)$ where all the terms are bounded by a constant $C\in \R$, we have that $$\mathrm{tr}(ABA)\leq C\mathrm{tr}(A^2).$$ 
\end{lemma}
\begin{proof}
Rewrite the product of the matrices block-by-block, where $A_i\in \mathcal{M}_{n,1}(\R)$ are the columns of $A$:
\begin{align*}
\mathrm{tr}&\left(
    \begin{bmatrix}
A_1 & A_2 & A_3 & \cdots & A_n 
\end{bmatrix}
\begin{bmatrix}
b_1 & 0 & 0 & \cdots & 0 \\
0 & b_2 & 0 & \cdots & 0 \\
0 & 0 & b_3 & \cdots & 0 \\
\vdots & \vdots & \vdots & \ddots & \vdots \\
0 & 0 & 0 & \cdots & b_n \\
\end{bmatrix}
\begin{bmatrix}
A_1^\top \\
A_2^\top \\
A_3^\top \\
\vdots \\
A_n^\top \\
\end{bmatrix}
\right)\\&=
\mathrm{tr}\left(
    \begin{bmatrix}
b_1A_1 & b_2A_2 & b_3A_3 & \cdots & b_nA_n 
\end{bmatrix}
\begin{bmatrix}
A_1^\top \\
A_2^\top \\
A_3^\top \\
\vdots \\
A_n^\top \\
\end{bmatrix}
\right)\\&=
\mathrm{tr}\left(\sum_{i=1}^nb_iA_iA_i^\top\right)\\
&=\sum_{i=1}^nb_i\mathrm{tr}\left(A_iA_i^\top\right)\\
&\leq C\sum_{i=1}^n\mathrm{tr}\left(A_iA_i^\top\right)\\
&=C\mathrm{tr}(A^2)
\end{align*}
\end{proof}

The subsequent lemma, which provides a bound on the maximum of sub-Gaussian random variables, has been derived from Section 8.2 of \cite{arlot:hal-01485506}.

\begin{lemma}[Maximum of sub-Gaussian variables]\label{lemma:maxSubGauss}
    Given $Z_1, ..., Z_k$ sub-Gaussian random variables with variance factor $v$, i.e.
    \begin{align*}
        \forall k\in [K], \qquad \e{Z_k}=0\qquad\text{and}\qquad\forall \lambda\in \R,\qquad \log\left(\e{\exp{\lambda Z_k}}\right)\leq \frac{v\lambda^2}{2},
    \end{align*}
    then 
    \begin{align*}
        \e{\max_{i\in[K]}Z_k}\leq \sqrt{2v\log(K)}.
    \end{align*}
    
\end{lemma}

\begin{lemma}[Projection of a Gaussian vector] 
\label{lemma:obsGaus}
    Given a missing pattern $m\in \{0,1\}^d$ and a Gaussian vector $X\sim\mathcal{N}(\mu,\Sigma)$, then the vector with missing values $X_{\mathrm{obs}(m)}$ is still a Gaussian vector and $X_{\mathrm{obs}(m)}\sim\mathcal{N}(\mu_{\mathrm{obs}(m)},\Sigma_{\mathrm{obs}(m)\times\mathrm{obs}(m)})$.
\end{lemma}
\begin{proof}
    Since $X$ is a Gaussian vector, every linear combination of its coordinates is a Gaussian variable. In particular,  every linear combination of the subset $\mathrm{obs}(m)$ of coordinates is a Gaussian variable, then $X_{\mathrm{obs}(m)}$ is a Gaussian vector. 

    To prove the second statement, for a given $u\in \R^{d-\left\|m\right\|_0}$, we will denote $u'\in\R^d$ the imputed-by-0 vector, i.e. $u'_j=0$ if $m_j=1$ and $u'_j=u_i$ with $i=j-\sum_{k=1}^jm_k$ otherwise. Then,
    \begin{align*}
        \forall u\in \R^{d-\left\|m\right\|_0},\qquad \Psi_{X_{\mathrm{obs}(m)}}(u)&=\e{\exp(iu^{\top}X_{\mathrm{obs}(m)})}
        \\&=\e{\exp(i(u')^{\top}X)}\\
        &=\exp(i(u')^{\top}\mu-\frac{1}{2}(u')^{\top}\Sigma(u'))\qquad \qquad \qquad (X\sim \mathcal{N}(\mu, \Sigma))
        \\&=\exp(iu^{\top}\mu_{\mathrm{obs}(m)}-\frac{1}{2}u^{\top}\Sigma_{\mathrm{obs}(m)\times\mathrm{obs}(m)} u)
    \end{align*}
\end{proof}

\section{Additional experiments}
\label{app:add_experiments}

\begin{center}
    \begin{figure}
        \centering
        \includegraphics[width=0.4\linewidth]{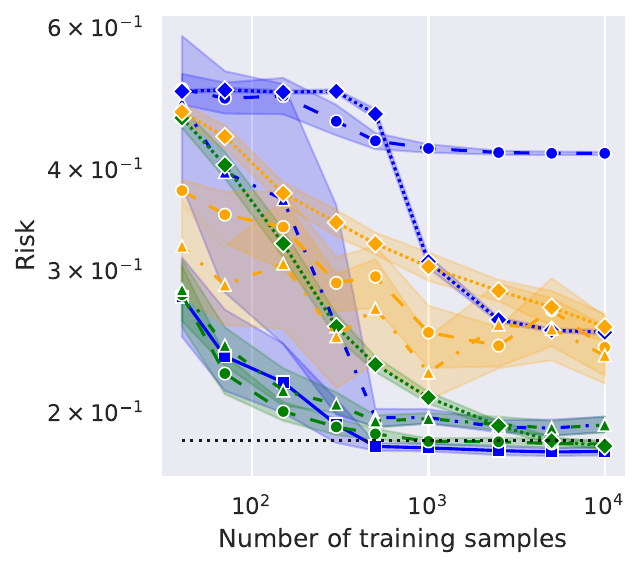}
        \caption{LDA generated data with MCAR missingness and $\Sigma$ the Toeplitz matrix defined above.}
        \label{fig:n-MCAR-cor2-LDA}
    \end{figure}
\end{center}

\begin{center}
    \begin{figure}
        \centering
        \includegraphics[width=0.4\linewidth]{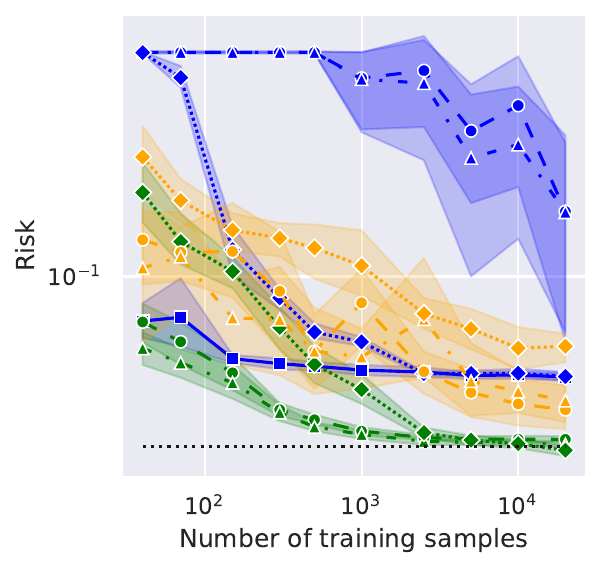}
        \caption{LDA generated data with MNAR missingness and $\Sigma$ the Toeplitz matrix defined above.}
        \label{fig:n-MNAR-cor2-LDA}
    \end{figure}
\end{center}

\begin{center}
    \begin{figure}
        \centering
        \includegraphics[width=0.4\linewidth]{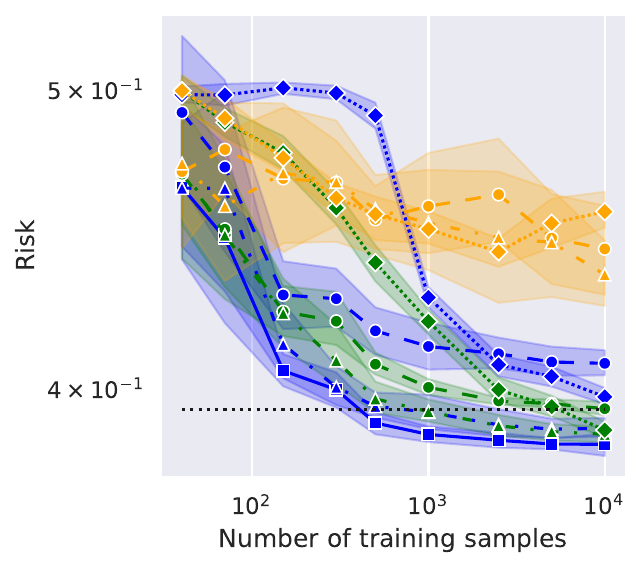}
        \caption{Logistic data with MCAR missingness and $\Sigma$ the Toeplitz matrix defined above.}
        \label{fig:n-MCAR-cor2-Logist}
    \end{figure}
\end{center}

\begin{center}
    \begin{figure}
        \centering
        \includegraphics[width=0.4\linewidth]{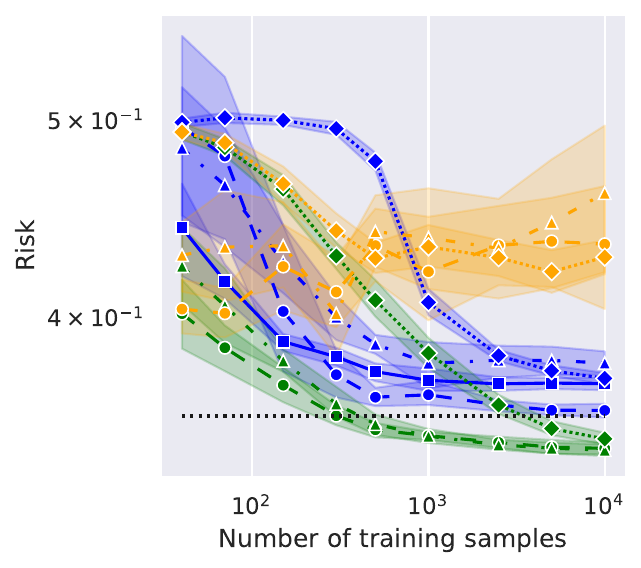}
        \caption{Logistic data with MNAR missingness and $\Sigma$ the Toeplitz matrix defined above.}
        \label{n-MNAR-cor2-Logist.pdf}
    \end{figure}
\end{center}

\end{document}